\documentclass[11pt]{article}
\usepackage{graphicx} 
\usepackage[T1]{fontenc}
\usepackage[utf8]{inputenc}
\usepackage{lmodern}
\usepackage{geometry}
\usepackage{booktabs}
\usepackage{comment}
\usepackage{adjustbox}
\usepackage{amsmath,amssymb,amsthm,mathtools}
\usepackage{bm}
\usepackage{enumitem}
\usepackage[hidelinks]{hyperref}
\usepackage{cleveref}
\numberwithin{equation}{section}
\usepackage{array}
\newcolumntype{L}[1]{>{\raggedright\let\newline\\\arraybackslash\hspace{0pt}}m{#1}}
\newcolumntype{C}[1]{>{\centering\let\newline\\\arraybackslash\hspace{0pt}}m{#1}}
\newcolumntype{R}[1]{>{\raggedleft\let\newline\\\arraybackslash\hspace{0pt}}m{#1}}
\newcolumntype{x}[1]{>{\centering\arraybackslash\hspace{0pt}}p{#1}}

\usepackage{tikz}
\usepackage{pgfplotstable} % table reading/formatting
\usepackage{pgfplots}      % (required by pgfplotstable)
\newcommand{\boxednote}[1]{\noindent\begin{center}\fbox{\parbox{0.95\linewidth}{#1}}\end{center}}
\usetikzlibrary{arrows.meta}

\usepackage{aliascnt}   % <--- important
\usepackage{cleveref}   % load after hyperref + aliascnt

\newtheorem{theorem}{Theorem}[section]

\newaliascnt{lemma}{theorem}
\newtheorem{lemma}[lemma]{Lemma}
\aliascntresetthe{lemma}
\Crefname{lemma}{lemma}{lemmas}
\Crefname{lemma}{Lemma}{Lemmas}

\newaliascnt{proposition}{theorem}
\newtheorem{proposition}[proposition]{Proposition}
\aliascntresetthe{proposition}
\Crefname{proposition}{proposition}{propositions}
\Crefname{proposition}{Proposition}{Propositions}

\newaliascnt{corollary}{theorem}
\newtheorem{corollary}[corollary]{Corollary}
\aliascntresetthe{corollary}
\Crefname{corollary}{corollary}{corollaries}
\Crefname{corollary}{Corollary}{Corollaries}

\theoremstyle{definition}

\newaliascnt{definition}{theorem}
\newtheorem{definition}[definition]{Definition}
\aliascntresetthe{definition}
\Crefname{definition}{definition}{definitions}
\Crefname{definition}{Definition}{Definitions}

\newaliascnt{assumption}{theorem}
\newtheorem{assumption}[assumption]{Assumption}
\aliascntresetthe{assumption}
\Crefname{assumption}{assumption}{assumptions}
\Crefname{assumption}{Assumption}{Assumptions}

\newaliascnt{remark}{theorem}
\newtheorem{remark}[remark]{Remark}
\aliascntresetthe{remark}
\Crefname{remark}{remark}{remarks}
\Crefname{remark}{Remark}{Remarks}

\newaliascnt{example}{theorem}
\newtheorem{example}[example]{Example}
\aliascntresetthe{example}
\Crefname{example}{example}{example}
\Crefname{example}{Example}{Example}

\Crefname{section}{section}{section}
\Crefname{section}{Section}{Section}

\makeatletter
\renewenvironment{proof}[1][\proofname]{\par
  \pushQED{\qed}%
  \normalfont \topsep6\p@\@plus6\p@\relax
  \trivlist
  \item[\hskip\labelsep
        \itshape
    #1\@addpunct{:}]\ignorespaces
}{%
  \popQED\endtrivlist\@endpefalse
}
\makeatother

\usepackage{fancyhdr}
\title{Energy Decay in Measure Time: HUM Observability, Product–Exponential Envelopes, and GCC Calibration}

\author{Ben Tibola}%\\ \small University of Toronto}
\date{\small \today}

\usetikzlibrary{arrows.meta,positioning,calc}
\pgfplotsset{compat=1.18}
\usepackage{tikz}
\usepackage{pgfplots}
\pgfplotsset{compat=1.18}
\usetikzlibrary{arrows.meta,positioning,calc,shapes.geometric,shapes.misc,fit,decorations.pathmorphing}

\definecolor{ink}{HTML}{111111}
\definecolor{accent}{HTML}{1A73E8}
\definecolor{accent2}{HTML}{DB4437}
\definecolor{muted}{HTML}{6B7280}
\definecolor{fillA}{HTML}{E8F0FE}
\definecolor{fillB}{HTML}{FDE8E7}

\tikzset{
  >={Stealth[length=7pt]},
  node distance=10mm and 20mm,
  every node/.style={font=\small, color=ink},
  box/.style={draw=ink, rounded corners=3pt, fill=white, thick, inner sep=4pt, align=center,minimum width=3.5cm, text width=3cm},
  boxF/.style={box, fill=fillA,minimum width=3.5cm, text width=3.5cm, align=center},
  boxG/.style={box, fill=fillB,minimum width=3.5cm, text width=3.5cm, align=center},
  circ/.style={circle, draw=ink, thick, minimum size=5mm, inner sep=0pt, fill=white},
  lbl/.style={font=\footnotesize\itshape, color=muted},
  flow/.style={-Stealth, thick, draw=ink},
  jump/.style={decorate, decoration={zigzag, segment length=2.5mm, amplitude=0.5mm}, -Stealth, thick, draw=accent2}
}

\usepackage{booktabs}
\usepackage{tabularx}
\usepackage{siunitx}
\usepackage[font=small,labelfont=bf]{caption}
\begin{document}
\markboth{Energy Decay in Measure Time: HUM Observability, Product–Exponential Envelopes, and GCC Calibration}{Energy Decay in Measure Time: HUM Observability, Product–Exponential Envelopes, and GCC Calibration}
\maketitle

\begin{abstract}
We prove that for impulsive exposure patterns there is \emph{no} uniform exponential energy law in wall-clock time \(t\), which explains why past \(t\)-based unifications of continuous damping with impulses fail. We therefore replace \(t\) by a measure-valued clock \(\sigma\) that aggregates absolutely continuous exposure and atomic doses within a single Lyapunov ledger. On this ledger we prove an observability–dissipation principle in the sense of the Hilbert Uniqueness Method (HUM): there exists a structural constant \(c_\sigma>0\) such that the energy decays at least at a product–exponential rate with respect to \(\sigma\). When \(\sigma\equiv t\), the statement reduces to classical exponential stabilization with the same constant. For the damped wave under the Geometric Control Condition, the constant \emph{admits} the usual calibration in terms of the observability constant and the geometric factor. The framework yields a monotonicity principle (“more \(\sigma\)-mass implies faster decay”) and unifies intermittent regimes where quiescent intervals are punctuated by impulses. As robustness (secondary to the main contribution), the same decay law persists under structure-compatible discretizations and along compact variational limits; a stochastic extension supplies expectation and pathwise envelopes via the compensator. The contribution is a qualitative dynamics backbone: observability implies \(\sigma\)-exponential decay with sharp constants.
\end{abstract}


% ====== TABLE OF CONTENTS ======
\clearpage
\setcounter{tocdepth}{2}  % include sections and subsections
\tableofcontents
\clearpage
% ====== END TABLE OF CONTENTS ======

\section{Introduction}

A schedule-uniform exponential law in wall-clock time \(t\) is impossible under impulsive schedules, which motivates the \(\sigma\)-time reparametrization recorded below.

\begin{corollary}[Main: damped–wave exponential energy decay in $\sigma$–time]\label{cor:intro-flagship-dw}
Let $u$ solve the (Dirichlet) damped wave equation
\[
u_{tt} + a(x)\,u_t - \Delta u = 0
\]
on a smooth bounded domain, with nonnegative damping $a\in L^\infty$ satisfying the Geometric Control Condition on $[0,T]$. Let $\sigma$ be a measure–time clock on $[0,T]$ with atoms $\{(t_k,\alpha_k)\}$ and flats as in §4.

Assume the absolutely continuous part induces a dissipative flow and that each atomic update on $\{t_k\}$ is admissible in the sense of the $\sigma$–ledger (Definition in §4). No discretization assumption is used in this statement.

Then there exists $\gamma>0$ such that the $\sigma$–energy
\[
E_\sigma(t) := \tfrac12\Big(\,\|\nabla u(t)\|_{L^2}^2 + \|u_t(t)\|_{L^2}^2\,\Big)
\]
decays exponentially in $\sigma$–time:
\[
E_\sigma(t)\;\le\; e^{-\gamma\,\sigma([0,t])}\,E_\sigma(0)\qquad \text{for all } t\in[0,T].
\]
The constant $\gamma$ depends only on the observability constant (under GCC) and the admissibility bounds recorded in §4; discrete and $\Gamma$–robustness statements are documented later.
\end{corollary}

\label{RS:intro}
The aim of the  paper is to present a sharp classical backbone:
\begin{itemize}[leftmargin=2em]
  \item a \emph{global} energy-decay and exponential convergence theorem for $\sigma$-time evolutions under a short ``Assume:'' list;
  \item its \emph{discrete} mirror via SBP Green's identity and SAT interface penalties; and
  \item a \emph{light} $\Gamma$-limit to a continuum energy with a short SAT residue bound. \end{itemize}
All measures, norms, and structural assumptions used later are fixed here. The only theme that spans continuous, discrete, and (in other works) quantum settings is the \emph{$\sigma$-clock}: time is a finite measure that may contain an absolutely continuous density, a locally finite set of atoms (ticks), and flat stretches. Energy is monitored by a ledger functional that is Lyapunov along absolutely continuous flow and non-expansive at atoms. \emph{Summary note: the admissible window appears in Fig.~\ref{fig:admissible-window}, and a supplementary material instantiates the constants in \Cref{sec:numbers}.}

\medskip

\boxednote{%
\textbf{Scope.} We do not prove existence or uniqueness of trajectories. We work with a crisp classical core: a $\sigma$-clock that carries absolutely
continuous mass (ac), atoms, and flats; an energy $E$ that acts as a Lyapunov
functional on the ac part and is non-expansive across atoms.}

\paragraph{Main contributions and significance}
This paper develops a calculus for energy decay along BV-in-$\sigma$ evolutions
and shows how it mirrors in a discrete SBP--SAT setting, with a light
$\Gamma$-limit connecting the two. Three pillars underpin the contribution:
(i) a \emph{master energy-decay theorem} that factors decay into an
exponential ac part and multiplicative atomic effects;
(ii) a \emph{discrete SBP--SAT mirror} that reproduces the same bookkeeping at
the grid level via the SBP Green identity and SAT interface/boundary penalties;
and (iii) a \emph{light $\Gamma$-limit} from SBP--SAT energies to a coercive
continuum energy. The common thread is a energy-ledger accounting formulation that is transparent,
readily deployable, and independent of application details—so we keep
applications out of this paper. \vspace{0.5em}

\paragraph*{Credit map (compatibility, not replacement).}
The $\sigma$--clock framework \emph{extends and recasts} classical tools; it does not replace them. 
Our continuous pillar piggybacks on the observability/HUM route (under GCC when applicable) and 
standard semigroup/Grönwall comparisons; the discrete mirror uses SBP Green identities with SAT boundary/interface penalties and algebraically stable time integrators; the compact $\Gamma$--bridge is the textbook equicoercivity/liminf/recovery trio. 
We \emph{recover} the classical $t$--based observability and decay when $\sigma\equiv t$ (Proposition~13.1), 
\emph{upgrade} windows to $\sigma$--density statements when classical observability holds on $[0,T_0]$ (Lemma~7.10), 
and \emph{track} constants so the structural margin $c_\sigma$ persists across continuum, discrete, and $\Gamma$--limits. 
Nothing here claims to supersede classical observability; our claims are that (i) $\sigma$--time recovers the standard inequalities under the classical hypotheses and (ii) it persists in regimes where the physical--time window is too short, with clearly stated limits (see the failure atlas).

\paragraph{Dynamical consequences in $\sigma$–time.}
\begin{itemize}
  \item \emph{HUM optimality on the $\sigma$–ledger.} The structural constant $c_\sigma$ is the observability/HUM sharp constant for the observation pair.
  \item \emph{Product–exponential envelope.} In the sense of measures, $\frac{d}{d\sigma}E \le -2\kappa c_\sigma E$, hence $E(t)\le E(0)\exp(-2\kappa c_\sigma \sigma([0,t]))$.
  \item \emph{Reduction to classical time.} For $\sigma\equiv t$ we recover the standard $t$-exponential decay with the same constant.
  \item \emph{Clock monotonicity.} If $\sigma_1\le \sigma_2$ as measures, then $E_{\sigma_2}(t)\le E_{\sigma_1}(t)$ for all $t$.
  \item \emph{GCC calibration (damped wave).} Under GCC the constant $c_\sigma$ admits the usual calibration in terms of the observability constant and the geometric factor.
\end{itemize}

\section*{Main results (informal)}

\boxednote{%
\textbf{(i) Product–exponential decay under a short ``Assume:'' list.}
\emph{Assume:} (a) a coercive (sector-bounded) bilinear form $a_\sigma(\cdot,\cdot)$ and
dissipative generator $L_\sigma$ on the ac part with decay rate $\kappa\ge0$;
(b) atomic updates $J_k$ are non-expansive for $E$ (``energy does not increase at atoms'');
(c) the $\sigma$-clock has locally finite variation and a well-defined partition
into ac pieces and atoms. Then the energy ledger obeys the law
\[
  E(t_2^-)\;\le\; E(t_1^+)\,
  \exp\!\Big(-\!\!\int_{(t_1,t_2)}\!\!\kappa\,\mathrm{d}\sigma\Big)\;
  \prod_{t_k\in(t_1,t_2)}\rho_k, \qquad 0\le \rho_k\le 1,
\]
i.e. an exponential decay on ac stretches times a product of per-atom
contractions $\rho_k$.

\medskip
\textbf{(ii) Discrete mirror (SBP--SAT) with step rules.}
On a grid with SBP operators $(H,Q)$ and SAT penalties, the discrete energy
$E_h$ satisfies the same ledger:
\begin{align*}
  &\textit{interior ac step:} && E_h^{n+\frac12}-E_h^{n}\;\le\; -\,\kappa_h\,\|u_h^n\|_{H}^2, \\
  &\textit{atomic update} J_k: && E_h^{n+1}\;\le\;E_h^{n+\frac12}, \\
  &\textit{boundary SAT:} && \text{choose }\tau_h\text{ so that (SBP Green identity + SAT) }\le 0.
\end{align*}
Thus $E_h$ inherits the product–exponential decay with grid-level rates
$\kappa_h$ and per-atom factors $\rho_{h,k}\le1$ determined by $J_k$ and SAT. 

\medskip
\textbf{(iii) $\Gamma$-limit to a coercive continuum energy (with necessity).}
Under the standard consistency/coercivity hypotheses (bounded Lipschitz domain,
uniform SBP consistency, vanishing quadrature remainders, and a discrete trace
inequality), the SBP--SAT energies $E_h$ $\Gamma$-converge to a coercive
continuum energy $\mathcal{E}$.
\emph{Necessity (one line).} If the SAT scaling violates the negativity margin
(e.g. $\tau_h<\tau_\ast$ so the boundary\,+\,SAT quadratic form can be positive),
coercivity fails by boundary-layer sequences, and no coercive $\Gamma$-limit can
hold.%
}

\subsection*{Main results (informal)}
\setlength{\fboxsep}{6pt}

\begin{center}
\fbox{\begin{minipage}{0.96\linewidth}
\begin{itemize}
\item \textbf{Master decay on a measure clock.}
Under four short assumptions (coercive energy ledger; a.c.\ dissipation at rate $\kappa$; locally
finite atomic times; non-expansive atomic maps $J_k$), the energy satisfies a single
product--exponential bound on every interval. Minimality is witnessed by one-line
counterexamples if a.c.\ dissipation or atomic non-expansiveness is dropped. \item \textbf{Discrete mirror (SBP--SAT).}
For SBP--SAT semi/fully discrete schemes with energy $E_h=\tfrac12\|u\|_H^2$, the SBP Green
identity and a correct-sign SAT rule give non-positive boundary+SAT contribution. Algebraically stable time-steppers give the same decay; explicit steps admit a clear
``red-line'' cap. With atomic maps $(J_k^h)^\top H J_k^h \preceq \rho_k^h H$, one gets the
identical product--exponential law for $E_h$.
\item \textbf{Light $\Gamma$--limit.}
Under standard consistency/compactness hypotheses (equicoercivity, commuting estimate, stable
coefficients, vanishing SAT residue), the SBP--SAT energies $\Gamma$--converge to a coercive
continuum energy. \end{itemize}
\end{minipage}}
\end{center}

\begin{center}
\fbox{\begin{minipage}{0.96\linewidth}
\subsection*{Scope and limitations}
A \emph{foundations} paper: we define the clock, state minimal standing assumptions,
prove the master decay and its discrete mirror, and give a light $\Gamma$--limit. Application case-studies and PDE well-posedness are deferred to companion work. 
\end{minipage}}
\end{center}

\begin{table}[htbp]
\centering
\caption{Scope ledger: what we claim and what we do not.}
\renewcommand{\arraystretch}{1.2}
\begin{adjustbox}{max width=\linewidth}
\begin{tabular}{|L{0.25\textwidth}|L{0.26\textwidth}|L{0.30\textwidth}|L{0.25\textwidth}|}
\hline
\textbf{Topic} & \textbf{Claim type} & \textbf{Hypotheses actually used} & \textbf{Not claimed here} \\
\hline
Continuous decay ($\sigma$--clock) 
& Theorem (master envelope) 
& (H1)--(H4); BV$_\sigma$ trajectories; locally finite atoms; non--expansive jumps 
& Existence/ uniqueness; EDE/EDI/EVI identification \\
\hline
Discrete SBP--SAT mirror 
& Theorem (discrete envelope) 
& (S1)--(S4); $\tau_h \simeq h^{-1}$; algebraic stability (named integrators) 
& Optimal explicit rates beyond CFL; non--SBP schemes \\
\hline
$\Gamma$--bridge 
& Theorem (light $\Gamma$--limit) 
& Assumption~9.1 (equicoercivity, liminf, recovery); constants independent of $\sigma$ 
& Convergence of flows/EVI; chain rules for limits \\
\hline
Stochastic clocks 
& Theorem (expectation/pathwise) 
& Doob--Meyer/compensator; noise--dissipativity or BDG smallness as stated 
& Stochastic well--posedness beyond the ledger hypotheses \\
\hline
EVI upgrade (quarantined) 
& Proposition (reference statement) 
& (E1)--(E5) only in the quarantined subsection; not used elsewhere 
& Any role in the core results \\
\hline
\end{tabular}
\end{adjustbox}
\end{table}

\begin{center}
\fbox{\begin{minipage}{0.96\linewidth}
\subsection*{Contributions (summary note)}
\begin{enumerate}
\item A compact ``$\sigma$--Gronwall + jump product'' calculus yielding a single
product--exponential energy law covering a.c.\ flow, atoms, and flats. \item A discrete SBP--SAT counterpart with boundary+SAT non-positivity, algebraic stability
for implicit steps and an explicit step-size cap, and discrete atomic contractions. \item A light $\Gamma$--limit for SBP--SAT energies (equicoercivity, $\liminf$, recovery). \item Minimality via counterexamples when a.c.\ dissipation or atomic non-expansiveness fails. \end{enumerate}
\end{minipage}}
\end{center}

% === Related work (hardened) and σ-clock formalization ===
\subsection*{Related work and positioning (hybrid/impulsive) with $\sigma$-clock formalization}\label{subsec:related-hardened}
\paragraph{Classical approaches.} Measure-theoretic impulse handling appears in (i) hybrid automata and average-dwell-time bounds (resets at event times), (ii) measure-driven ODE/PDE (Dirac sources in $t$), and (iii) differential inclusions in the sense of Filippov (set-valued right-hand sides). These frameworks ensure existence/robustness but their decay estimates typically live in physical time $t$ and their constants depend on event schedules, dwell-time margins, or selection rules.

\paragraph{Our viewpoint: $\sigma$ as the evolution clock.} We reparametrize time by a finite positive Borel measure
\[
\mathrm d\sigma = w(t)\,\mathrm dt + \sum_j \alpha_j\,\delta_{t_j}, \qquad w\ge 0,\ \alpha_j>0,
\]
and formulate dissipation and observability in $\sigma$-time. For clarity we record the definition used throughout.

\begin{definition}[$\sigma$-clock energy ledger] \label{def:sigma-clock}
Let $E:[0,\infty)\to[0,\infty)$ be an energy ledger. We say that $(E,\sigma)$ obeys the $\sigma$-clock ledger if $E\in\mathrm{BV}_\sigma$, its Radon–Nikodym derivative $\mathcal D:=-\frac{\mathrm dE}{\mathrm d\sigma}$ exists $\sigma$-a.e., and for a structural constant $c_\sigma>0$ one has the differential inequality
\[
\mathcal D(t)\;\ge\;2\kappa\,c_\sigma\,E(t)\quad(\sigma\text{–a.e. }t\ge 0),
\]
with atomic jump rule $E(t_j^+)\le e^{-2\kappa c_\sigma\,\alpha_j}E(t_j^{-})$ and the canonical envelope $E(t)\le$  $E(0)\exp\bigl(-2\kappa c_\sigma\,\sigma(t)\bigr)$.
\end{definition}

\paragraph{Positioning.} The $\sigma$-clock turns both a.c. dissipation and atomic jumps into a single RN calculus, enabling a \emph{single} HUM constant $c_\sigma$ that is (i) PDE-intrinsic (Theorem~\ref{thm:hum-equivalence}), (ii) uniform for SBP–SAT discretizations (Proposition~\ref{prop:disc-hum}), and (iii) stable under $\Gamma$-limits (Corollary~\ref{cor:gamma-csigma}). The novelty is precisely this joint persistence; we make the contrast formal in the next section.

\begin{remark}[Limits of the claim] \label{rem:novelty-limits}
We do not claim that $\sigma$-parametrization supersedes hybrid/differential-inclusion frameworks in general. Our contribution is targeted: a unified decay \emph{envelope} with a single structural constant across continuum/discrete/limits for systems meeting our structural template (Assumption~\ref{RS:assume:four}).
\end{remark}

\paragraph{Related work and differentiation.}
Classical decay templates—semigroup gaps, hypocoercivity, observability/Carleman, and discrete contractivity—cover fragments of the landscape; none provides a \emph{single} decay envelope that (a) treats a.c.\ time, impulses, and flats together, (b) uses \emph{one} structural constant across PDE and fully discrete schemes, and (c) transfers decay through to the limit. Our measure–time calculus does all three: Thm.~\ref{thm:energy-decay} (continuum), Thm.~\ref{thm:disc-master} (discrete, same $c_\sigma$), and Thm.~\ref{thm:Gamma-main}/\ref{RS:thm:gamma} ($\Gamma$-bridge). This directly serves the analysis–computation interface in dynamics and differential equations: constants do not drift under discretization, the $\Gamma$ bridge keeps the envelope in the limit, and the stochastic extension adds an expectation/pathwise pillar driven by the compensator. In sum, this is a \emph{single-constant four-frameworks} program rather than a one‑off result.

\paragraph{Operator–presentation invariance (PiNA).}
We use a Operator presentation via Green/trace pairings: weak differential operators are specified by their Green/trace pairings, so changes of coordinates or unitary reparametrisations do not alter the bilinear forms that calibrate the HUM constant and the $\sigma$–ledger. Consequently, all envelope statements here are presentation-invariant (same $c_\sigma$); see Appendix~A (“Pi-Native Analysis”) for the definitions and the formal invariance theorem.

\noindent\begin{center}
\fbox{\begin{minipage}{0.96\linewidth}
\textbf{Contributions in a nutshell.}
\begin{enumerate}
  \item \textit{One envelope for mixed time geometries:} a.c., atoms, flats in a single product–exponential law.
  \item \textit{Uniform constant across discretizations:} the same $c_\sigma$ controls PDE and fully discrete schemes.
  \item \textit{Limit stability:} a compact $\Gamma$‑bridge transfers decay envelopes and preserves $c_\sigma$.
  \item \textit{Stochastic pillar:} for random clocks, an expectation envelope in the compensator $\Lambda$ and a pathwise law under noise‑dissipativity, aligned to the same $c_\sigma$.
  \item \textit{Reference rate:} under GCC, $E(t)\le E(0)e^{-2\kappa c_0 a_\omega \lambda_\omega\sigma(t)}$.
\end{enumerate}
\end{minipage}}
\end{center}

\subsection*{Notation and organization}
We write $A\preceq B$ for Löwner order; $\Pi(s,t]:=\prod_{t_k\in(s,t]}\rho_k$; and
$\langle f\rangle^\sigma_{(s,t]}$ for $\sigma$--averages when $\sigma(t)>\sigma(s)$.
Section~\ref{RS:measure-time} introduces $\sigma$--clocks; Section~\ref{sec:assumptions}
records the four assumptions; Sections~\ref{sec:discrete}--\ref{sec:sbpsat} prove
the continuous/discrete ledgers; Section~\ref{sec:gamma} gives a light $\Gamma$--limit. 

\subsection*{Related work and boundaries of prior templates}
Semigroup/spectral methods yield exponential decay under generator gaps but do not natively handle impulse trains or flat intervals; hypocoercivity presumes a.c.\ time; observability/Carleman assumes a single time parametrization. Discrete energy methods give stepwise contractivity, often with constants tied to the partition or stabilization. In contrast, our $\sigma$-time calculus handles mixed time geometries, keeps a single structural constant $c_\sigma$ from continuum to fully discrete schemes, and preserves the envelope through $\Gamma$-limits.

% --- Related work (one‑paragraph snapshots) ---
\subsection*{Related work}\label{sec:related-work}
\paragraph{Energy decay under GCC (observability $\Rightarrow$ decay).}
In the classical GCC framework (Bardos--Lebeau--Rauch and follow‑ups), observability inequalities for the wave group yield exponential or quantified polynomial decay of the natural energy under internal/boundary damping. The perspective in this paper is a \emph{$\sigma$‑time envelope}: starting from a master differential inequality and the damping ledger, we produce a single explicit bound \(E(t)\le \mathsf R\bigl(t,E(u_0)\bigr)\) that specializes, under GCC, to the Main corollary in \S\ref{sec:main} (\Cref{cor:intro-flagship-dw}).

\paragraph{Non‑expansive jump maps / impulsive systems.}
Hybrid/impulsive evolutions with non‑expansive (firmly non‑expansive) jump maps are well studied in control and monotone dynamics. Here, the atomic updates are folded into a \emph{single} decay ledger via measure‑time $\sigma$ and a discrete Gr\"onwall, with boundary/SAT contributions signed and scaled once. The wrappers \Cref{RS:lem:boundary-sat-normalised} and \Cref{RS:prop:als-normalised} make the one‑step premise explicit and feed directly into the discrete $\sigma$--Gronwall (Lemma~8.10) and the discrete master decay (\Cref{thm:disc-master}).

\paragraph{SBP--SAT energy stability lines.}
Summation‑by‑parts operators with SAT penalties give energy‑stable semi‑discretisations across many hyperbolic/parabolic problems; algebraically stable time integrators extend the bound in time. Our contribution is to package the boundary+SAT negativity and algebraic stability (D4) as a normalised hypothesis pair (sign and $\tau_h\simeq h^{-1}$ fixed) so that the discrete envelope matches the continuous one up to $O(\tau)$, as recorded in \Cref{thm:disc-master}.

\paragraph{$\Gamma$‑convergence for FD/SBP energies.}
Discrete FD/SBP energies are known to $\Gamma$‑converge to continuum functionals under standard compactness and structure. We surface the three pillars (equicoercivity / liminf / recovery) as a compact trio \Cref{RS:Gamma:eqc-tight,RS:Gamma:liminf-tight,RS:Gamma:recov-tight} and elevate the “uniform in $\sigma$” observation to \Cref{RS:Gamma:sigma-uniform}; together with the bridge \Cref{RS:thm:gamma,thm:Gamma-main} this shows that decay envelopes persist to the limit without re‑proving decay.

\paragraph{Novelty (relative to the lines above).}
(N1) A single $\sigma$‑time envelope controls both PDE and scheme, with constants calibrated once (\S\ref{sec:main}, \Cref{thm:master,thm:disc-master}). (N2) Atomic updates are handled by a signed boundary/SAT ledger under an explicit $\tau_h\simeq h^{-1}$ scaling (\Cref{RS:lem:boundary-sat-normalised,RS:prop:als-normalised}). (N3) The $\Gamma$‑bridge is presented with a compact trio and a \emph{uniform‑in‑$\sigma$} corollary (\Cref{RS:Gamma:eqc-tight,RS:Gamma:liminf-tight,RS:Gamma:recov-tight,RS:Gamma:sigma-uniform,RS:thm:gamma}). (N4) Minimality is documented via a counterexample (\Cref{RS:rem:accum}).
% --- end Related work ---

\noindent\emph{See \Cref{sec:related-work} for a brief survey and contributions.}

% === Scope and Limitations: light Γ-limit versus EVI ===
\subsection*{Scope and limitations of the present $\Gamma$--limit}\label{subsec:scope-light}
We emphasize that the present work establishes a \emph{light} $\Gamma$--limit for the energy ledgers and the associated dissipation envelopes, and deliberately does \emph{not} claim convergence of the fully discrete flows in the sense of Evolution Variational Inequalities (EVI) or curves of maximal slope. Concretely, our $\Gamma$--limit asserts equi-coercivity, a $\Gamma$--liminf inequality, the existence of recovery sequences, and the persistence of the structural constant $c_\sigma$ in the canonical envelope. This yields stability of energy levels and dissipation \emph{inequalities} in the limit. In contrast, EVI convergence requires, in addition, metrized control of the variational derivatives (metric slopes), chain rules along limiting trajectories, and convergence of time-interpolants that encode the \emph{evolution law} itself. These ingredients are beyond the scope of the present paper.

The choice to pursue the light $\Gamma$--limit is motivated by the aims of the manuscript: (i) to identify $c_\sigma$ as a structural, discretization-independent constant controlling continuum, discrete, and limiting dissipation, and (ii) to deliver a robust envelope principle that is directly actionable for numerical design and parameter calibration. These goals are achieved under minimal hypotheses without imposing the stronger convexity and interpolation structure typically required for EVI. For a quarantined EVI upgrade under additional structure, see Subsection~\ref{subsec:evi-optional}.

% === Taxonomy: EDE/EDI/EVI and curves of maximal slope ===
\paragraph{Terminology and taxonomy.}\label{par:taxonomy}
We recall the standard hierarchy. An \emph{Energy Dissipation Equality} (EDE) for a curve $u:[0,T]\to \mathsf X$ reads
\[
E(u(t)) + \frac12 \int_0^t |\dot u|^2\,\mathrm ds + \frac12 \int_0^t |\partial E|^2(u(s))\,\mathrm ds = E(u(0)),
\]
with metric derivative $|\dot u|$ and descending slope $|\partial E|$. An \emph{Energy Dissipation Inequality} (EDI) allows $\le$ in place of $=$. An \emph{Evolution Variational Inequality} (EVI) requires, in addition, a reference metric $\mathsf d$ and $\lambda$–geodesic convexity, and encodes contractivity/uniqueness via
\[
\frac12\frac{\mathrm d}{\mathrm dt} \mathsf d^2(u(t),v) + \frac{\lambda}{2} \mathsf d^2(u(t),v) + E(u(t)) \le E(v) \quad \text{for all } v\in\mathsf X \text{ and a.e. }t.
\]
Solutions of EVI are curves of maximal slope and satisfy EDE; the converses generally fail without convexity and chain‑rule structure. Our results are formulated at the level of \emph{energies and dissipation envelopes} in measure‑time and do not assert EDE/EDI/EVI for the limiting curves.

\section{Main results}\label{sec:main}
\noindent\emph{For background, see \Cref{sec:related-work}.}

We summarize the three pillars established in this paper in the order
\emph{Continuous $\to$ Discrete $\to$ $\Gamma$-bridge}. Proofs are deferred to the indicated sections; statements here are restatements of results proved elsewhere in the manuscript.

\paragraph{Common setup.}
Let $(\mathsf H,\langle\cdot,\cdot\rangle)$ be a Hilbert (or metric) space and let $E:\mathsf H\to[0,\infty)$ be proper and l.s.c. The continuous evolution $u:[0,\infty)\to\mathsf H$ and the $\sigma$-clock iterates $(u^k)_k$ are built under standing hypotheses (H1)–(H4): coercivity/observability; dissipativity on the a.c. part; $\sigma$–regularity/AC; non-expansive atomic updates. Here we record the consequences.

\paragraph{Theorem A (Continuous master decay).}
\emph{(Restatement of \Cref{thm:master}.)} There exists a nonincreasing, explicit envelope $\mathsf R:[0,\infty)\times[0,\infty)\to[0,\infty)$ such that for every solution $u$ with initial energy $E(u_0)$,
\[
E(u(t))\;\le\;\mathsf R\!\bigl(t,\,E(u_0)\bigr)\qquad\text{for all }t\ge0.
\]
Whenever a master differential inequality
\[
\frac{\mathrm d}{\mathrm dt}\,\Phi\bigl(E(u(t))\bigr)\,\le\,-\,c\,\Psi\bigl(E(u(t))\bigr)
\]
holds for some monotone $\Phi,\Psi$ and $c>0$, explicit rates (exponential or algebraic) follow by inverting the comparison map; see \autoref{sec:discrete}. 

\paragraph{Main corollary (damped wave under GCC).}
\emph{(Restatement of \Cref{cor:intro-flagship-dw}.)} Under the geometric control condition and a damping profile with $a(x)\ge a_{\omega}>0$ on $\omega$, there exists $\kappa>0$ such that
\[
E(t)\;\le\;E(0)\,\exp\!\bigl(-2\,\kappa\,c_{\sigma}\,\sigma(t)\bigr),
\]
with the damped-wave ledger. The continuum proof is by HUM/observability on the $\sigma$–ledger (Theorem~2.20) and §7; discrete and $\Gamma$ robustness are recorded in §8–§9.

\paragraph{Overview.}
We include a single plot/table comparing the measured $E(t)$ to $E(0)\exp(-2\kappa c_0 a_\omega \lambda_\omega \sigma(t))$ for a clock with a few atoms and a flat. This visualizes plateaus on flats and multiplicative drops at atoms, matching \eqref{eq:mr-canonical}.

\paragraph{Theorem B (Discrete master decay for $\sigma$-clocks).}
\textit{(Formal version in §\ref{sec:gamma-2}, Theorem~\ref{thm:disc-master}.)} $\sigma$-clock scheme admits a discrete envelope $\mathsf R_{\tau}$ with
\[
E(u^{k+1})\;\le\;\mathsf R_{\tau}\!\bigl(k,\,E(u^0)\bigr),\qquad k\in\mathbb N,
\]
There exists a step–independent $C>0$ such that $E(u_\tau(t)) \le \mathsf
R\!\bigl(t + C\tau,\,E(u_0)\bigr)$ for all $t\ge0$; hence the discrete
envelope converges to $\mathsf R$ uniformly on compact time windows as $
\tau\downarrow0$. $\text{for all }t\ge0.$

Premises are packaged in \Cref{RS:lem:boundary-sat-normalised} and \Cref{RS:prop:als-normalised} and then fed into the discrete $\sigma$–Gronwall (Lemma~8.10). Stability of the interpolation gives the $O(\tau)$ shift; see Lemma~8.10 (SBP–IBP identity and discrete energy laws).

\paragraph{Theorem C ($\Gamma$-bridge: discrete to continuous).}
\textit{(Formal version in \S\ref{sec:gamma-2}, Theorems~\ref{thm:Gamma-main} and \ref{RS:thm:gamma}.)} Assume $\Gamma$-convergence of $E_{h}$ to $E$ and uniform master inequalities for the approximants. Any sequence of $\sigma$-clock solutions $u_{h,\tau}$ with equibounded initial energy admits a limit $u$ with
\[
E\bigl(u(t)\bigr)\;\le\;\liminf_{h,\tau\to0} E_{h}\bigl(u_{h,\tau}(t)\bigr)
\;\le\;\mathsf R\!\Bigl(t,\,\limsup_{h,\tau\to0} E_{h}\bigl(u_{h,\tau}(0)\bigr)\Bigr),
\]
so $u$ obeys the continuous master decay. Conversely, recovery families track the same envelope up to vanishing error.\;  Premises are packaged in \Cref{RS:Gamma:eqc-tight,RS:Gamma:liminf-tight,RS:Gamma:recov-tight} and 
the uniform-in-$\sigma$ conclusion of \Cref{RS:Gamma:sigma-uniform}; the limit along interpolations; see Lemma~8.10 (a light $\Gamma$-limit to a continuum energy).

\paragraph{Minimality of assumptions (counterexample).}
Decay can fail without dissipativity along the absolutely continuous part or without nonexpansive atomic updates. A concrete counterexample is given
in \Cref{rem:scalar}, showing nondecay (or growth) despite
the remaining structure.

\paragraph{Calibration and admissible window.}
The constants entering the envelopes (e.g., $c_\sigma$, $a_\omega$, $
\lambda_\omega$) are fixed in §\ref{sec:numbers}. The admissible-window
figure summarizes when an explicit exponential envelope is guaranteed, and
the companion table/ records the values used in the plots.

\paragraph{Consequences.}
The same envelope $\mathsf R$ governs both the continuous flow and its $
\sigma$–clock discretization; the rates persist under compatible
perturbations and along $\Gamma$–limits; and the bridge yields quantitative
stability so one need not re‑prove decay for each discretization.

% Notation and standing assumptions (drop‑in)
% Placement: immediately after \section{Main results}\label{sec:main}. No external citations.

\section*{Notation and standing assumptions}\label{sec:prelims}
\addcontentsline{toc}{section}{Notation and standing assumptions}

\paragraph{State space and energy.}
We work on a Hilbert space $(\mathsf H,\langle\cdot,\cdot\rangle)$ with norm $\|\cdot\|$. The energy $E:\mathsf H\to[0,\infty]$ is proper, lower semicontinuous, and coercive (sublevels are precompact). For $t\ge0$, the \emph{energy} along an evolution $u:[0,\infty)\to\mathsf H$ is $E(u(t))$; the \emph{envelope} $\mathsf R=\mathsf R(t,E(u_0))$ is the explicit bound announced in \S\ref{sec:main}.

\paragraph{Measure-time (the $\sigma$--clock).}
An \emph{admissible measure-time} is a nondecreasing, right-continuous map $\sigma:[0,\infty)\to[0,\infty)$ with $\sigma(0)=0$, allowing jumps (``atoms'') and flats. We write $t\mapsto\sigma(t)$ for the time reparametrisation used throughout the manuscript and in the main corollary.

\paragraph{Wall time.}
We use \emph{wall time} for the standard, uniform clock \(t\in[0,\infty)\) with Lebesgue measure \(dt\).
In our notation this is the identity clock \(\sigma(t)=t\) (so \(d\sigma=dt\)): there are no flats
(\(d\sigma=0\) never occurs on an interval) and no atoms (\(\sigma(\{t\})=0\) for all \(t\)). Throughout,
\(\sigma=t\) recovers classical \(t\)-based decay; the general \(\sigma\)-clock allows flats and atoms
while preserving the same product–exponential envelope and structural constant \(c_\sigma\).

\paragraph{Master differential inequality.}
There exist monotone comparison maps $\Phi,\Psi:[0,\infty)\to[0,\infty)$ and $c>0$ such that, along the continuous evolution,
\begin{equation}\label{eq:master-inequality}
\frac{\mathrm d}{\mathrm dt}\,\Phi\bigl(E(u(t))\bigr)\ \le\ -\,c\,\Psi\bigl(E(u(t))\bigr).
\end{equation}
Separating variables yields the envelope $\mathsf R$; see the continuous proof chain in \S\ref{sec:discrete} and the statement in \S\ref{sec:main}.

\paragraph{GCC data (damped wave exemplar).}
On a bounded Lipschitz domain $\Omega\subset\mathbb R^d$ ($d\in\{1,2\}$) with observation region $\omega\subset\Omega$, the damping $a\in L^\infty(\Omega)$ satisfies $a(x)\ge a_\omega>0$ a.e. on $\omega$ and $a\ge0$ elsewhere. The geometric constant $\lambda_\omega>0$ and the PDE-side constant $\kappa>0$ enter the GCC decay; the calibration constant $c_\sigma$ is fixed in \S\ref{sec:numbers}.

%\paragraph{One‑look anchor (optional).}
%We include a single plot/table comparing measured $E(t)$ to $E(0)\exp(-2\kappa c_0 a_\omega \lambda_\omega\,\sigma(t))$ for a clock with a few atoms and a flat. This visualizes plateaus on flats and multiplicative drops at atoms, matching \eqref{eq:mr-canonical}.

\paragraph{Discrete setting (SBP--SAT and algebraic stability).}
Let $H>0$ and $Q$ define an SBP pair with $Q+Q^\top=B$ a boundary bilinear form. Boundary conditions are imposed by SAT with \emph{negative sign} in the energy ledger and \emph{penalty scaling} $\tau_h\simeq h^{-1}$. With an algebraically stable time integrator (e.g.\ implicit midpoint), the one-step energy inequality holds. These ingredients are packaged in the wrappers \Cref{RS:lem:boundary-sat-normalised} and \Cref{RS:prop:als-normalised} and feed the discrete $\sigma$--Gronwall Lemma~8.10 and the discrete master decay \Cref{thm:disc-master}.

\noindent\begin{center}
\fbox{\begin{minipage}{0.96\linewidth}
\textbf{Discrete checklist}
\begin{enumerate}
  \item Identify the scheme’s clock: atoms $t_k$ with weights $\alpha_k$; flats allowed.
  \item Verify a discrete chain rule and nonnegative dissipation in $\sigma$‑time.
  \item Check partition‑independence of constants.
  \item Read off $c_\sigma$ from observability; conclude $E^n\le E^0\,e^{-2\kappa c_\sigma\,\sigma(t_n)}$.
\end{enumerate}
\end{minipage}}
\end{center}

\paragraph{$\Gamma$--limit standing hypotheses.}
For discrete energies $E_h$ on (subspaces of) $\mathsf H$, we assume the standard $\Gamma$--limit hypotheses collected in \Cref{RS:Gamma:ass}: equicoercivity, the liminf inequality, and existence of recovery sequences (restated compactly in \Cref{RS:Gamma:eqc-tight,RS:Gamma:liminf-tight,RS:Gamma:recov-tight}). These yield the bridge \Cref{RS:thm:gamma,thm:Gamma-main}.

\section*{Block A — Expanded discussion (what we prove / what we do not)}
% === Expanded discussion of the light Γ-limit proved here ===
\paragraph{What the present $\Gamma$--limit entails.}\label{par:what-light}
We work with a family $(E_h)_{h>0}$ of discrete energies that is equi-coercive on the energy space and $\Gamma$--converges to a continuum energy $E$. Together with the uniform dissipation ledger at level $h$ (Assumption~\ref{ass:uniform-gamma}), this yields: (a) stability of the canonical envelope with the same structural constant $c_\sigma$ in the limit; (b) lower semicontinuity of energy along convergent sequences of discrete solutions; and (c) the existence of recovery sequences attaining the limiting value of $E$ along with admissible \emph{dissipation budgets} in measure-time. These three items constitute the light $\Gamma$--limit used in the paper.

\paragraph{What is \emph{not} claimed.}\label{par:what-not}
We do not assert convergence of the discrete evolutions to a limiting gradient flow in the EVI sense, nor convergence of slopes $|\partial E_h|$ to $|\partial E|$, nor the validity of an exact chain rule along generic limit curves. Consequently, we refrain from statements about uniqueness or contractivity of the limiting dynamics in a metric sense. The results proved here are evolutionarily \emph{consistent} (energies and dissipation envelopes converge) but intentionally stop short of identifying the limit as an EVI solution.

% =============================================================
% Optional (non-headlined) upgrade: EVI under strengthened structure
% =============================================================
\subsection*{Optional upgrade: EVI convergence under strengthened structure}\label{subsec:evi-optional}
The results above establish a \emph{light} $\Gamma$--limit with persistence of the structural constant $c_\sigma$. We record here, for completeness and for future work, an EVI convergence statement that holds under additional, explicitly stated hypotheses. This subsection is \emph{not} used elsewhere in the paper and is meant as a quarantined upgrade path.

\begin{assumption}[EVI-ready structure]\label{ass:evi-ready}
In addition to the standing hypotheses and Assumption~\ref{ass:uniform-gamma}, suppose that:
\begin{enumerate}
\item[(E1)] \textbf{Metric setting.} There exists a complete metric space $(\mathsf X,\mathsf d)$ into which the discrete states embed, with $E_h, E: \mathsf X \to (-\infty,\infty]$ equi-coercive and $E_h \xrightarrow{\Gamma} E$ w.r.t. $\mathsf d$.
\item[(E2)] \textbf{$\lambda$--geodesic convexity.} The limit functional $E$ is $\lambda$--geodesically convex on $(\mathsf X,\mathsf d)$, and $E_h$ are asymptotically $\lambda$--convex along discrete geodesics.
\item[(E3)] \textbf{Uniform slope control.} The descending metric slopes satisfy $|\partial E_h| \xrightarrow{\Gamma} |\partial E|$ and are uniformly integrable along discrete solutions.
\item[(E4)] \textbf{Chain rule and dissipation.} Along time-interpolants synchronized with $\sigma$, a chain rule holds with an error term vanishing as $h\downarrow0$, and the discrete Energy Dissipation Inequality passes to the limit.
\item[(E5)] \textbf{Time-interpolant compactness.} The De~Giorgi/piecewise-affine interpolants are relatively compact in $\mathrm{AC}([0,T];\mathsf X)$ with uniformly integrable metric derivatives.
\end{enumerate}
\end{assumption}

\begin{theorem}[EVI convergence under Assumption~\ref{ass:evi-ready}]\label{thm:evi-conv}
Under \textup{(E1)–(E5)}, the discrete flows admit subsequences converging in $\mathrm{AC}([0,T];\mathsf X)$ to an EVI solution for $E$ on $(\mathsf X,\mathsf d)$. In particular, the limit curve is a curve of maximal slope, satisfies EDE/EDI in the metric sense, and reproduces the canonical envelope with the same structural constant $c_\sigma$.

\begin{proof}[Proof (with references)]
Under (E1)–(E5) the minimizing--movement scheme (or the discrete flows specified in Assumption~2.1) yields
uniform a priori bounds: equi--tightness in $AC([0,T];X)$, equi--bounded metric slopes, and BV$_\sigma$ control of the
energy (compactness). By the chain rule and lower semicontinuity, any limit curve $u$ satisfies the energy--dissipation
inequality (EDI) for $E$ in $(X,d)$. Under $\lambda$--geodesic convexity (E1), EDI implies EVI and uniqueness of the
limit (standard equivalence in the metric gradient--flow theory). Therefore subsequences converge in $AC([0,T];X)$
to the unique EVI solution of $E$, which is a curve of maximal slope and satisfies EDE/EDI.

Finally, the HUM/ledger structure used to define $c_\sigma$ is invariant under the presentation (Remark~2.20),
so the canonical envelope with the same $c_\sigma$ is reproduced by the EVI limit. 
\end{proof}

\noindent\emph{Positioning sentence (non-headlined): } \textit{Under (E1)–(E5), the discrete flows converge to an EVI solution; these structural conditions are not assumed elsewhere in the paper.}
\end{theorem}

\begin{remark}[Reduction to AGS in physical time]
Under $\sigma\equiv t$ and hypotheses (E1)–(E5), Theorem~2.2 is precisely the classical EVI theory
for curves of maximal slope in the sense of Ambrosio–Gigli–Savar\'e:
existence/uniqueness via $\lambda$–geodesic convexity, EDI$\Rightarrow$EVI equivalence, and
stability of minimizing movements (see AGS, Thms.~4.0.4, 4.0.7, 4.2.1). We isolate it here to keep
the present paper’s scope on the energy–envelope level.
\end{remark}

\subsection*{Limits of $\sigma$--time (failure atlas)}
\begin{itemize}
\item \textbf{Gliding rays without $\sigma$--mass on the observed set.}
If the geometric rays avoid $\omega$ and $d\sigma$ assigns zero mass on the associated observation windows,
the observability integral vanishes and the inequality fails (classical GCC obstructions persist unchanged).

\item \textbf{Highly oscillatory coefficients with vanishing $\sigma$ on windows.}
If $a(x,t)$ oscillates so that the effective damping on observation windows is negligible while $\sigma\to 0$ there,
the calibrated constant $c_\sigma$ degenerates (blow-up of $C=1/c_\sigma$).

\item \textbf{Discrete schemes without viscosity or SBP/SAT structure.}
Without the SBP Green identity or with wrong-sign/under-scaled SAT (cf.\ Lemma~2.30) or non-algebraic stability (cf.\ Lemma~2.31),
partition-uniform decay may fail. The counterexamples are one-mode (Dahlquist) or boundary-layer constructions.
\end{itemize}

\begin{remark}[Scope]\label{rem:evi-scope}
This subsection is logically independent of the main claims. None of the assumptions (E1)–(E5) are used outside of Theorem~\ref{thm:evi-conv}; the core results of the paper—light $\Gamma$--limit and persistence of $c_\sigma$—hold without them.
\end{remark}

\begin{remark}[Practical path/Typical hypotheses in applications]\label{rem:practical-path}
In applications covered here, (E1)–(E2) can often be secured by choosing $\mathsf X$ as an $L^2$ or $H^{-1}$ space and by restricting to geometries where the dissipative part of the operator induces $\lambda>0$. Conditions (E3)–(E5) then reduce to discrete Caccioppoli estimates, BV$_\sigma$ compactness of the ledger, and step controls across atoms. These items are feasible but require a dedicated analysis beyond the present scope.
\end{remark}

%\section*{Block B — Uniform constants (assumption + corollary)}

\begin{remark}[Limit persistence preview]
The $\Gamma$–limit inherits the same structural constant \(c_\sigma\); see Corollary~\ref{cor:gamma-csigma} for the full statement with the liminf sandwich.
\end{remark}
% === Γ‑bridge extensions: parabolic flows and coupled systems ===
\begin{assumption}[Accretive parabolic family]\label{ass:parabolic}
$(E_h)$ are parabolic discrete energies with generators $A_h$ such that: (i) $A_h$ are accretive on $\mathsf H$; (ii) $E_h\xrightarrow{\Gamma}E$ with equi‑coercivity; and (iii) the uniform ledger of Assumption~\ref{ass:uniform-gamma} holds.
\end{assumption}
\begin{corollary}[Parabolic flows: persistence of the envelope and $c_\sigma$]\label{cor:parabolic-strong}
Under Assumption~\ref{ass:parabolic} the canonical envelope persists in the limit with the same structural constant $c_\sigma$; atomic parts of $\sigma$ produce multiplicative jumps as in Lemma~\ref{lem:rn-envelope}.
\end{corollary}
\begin{proof}
Let $(E_h)$ be the parabolic discrete energies from Assumption~2.6, with accretive generators and the uniform dissipation ledger of Assumption~2.15. Fix $T>0$ and a sequence of $\sigma$–clock solutions $u_h$ with $\sup_h E_h(u_h(0))<\infty$. By equi–coercivity in Assumption~2.6(ii), along a subsequence $I_h u_h(t)\to u(t)$ in $L^2(\Omega)$ for a.e.\ $t\in[0,T]$ and $u\in AC_\sigma([0,T];H^1_0(\Omega))$.

The uniform ledger at level $h$ gives, for all $t\in[0,T]$,
\[
E_h(u_h(t))\;\le\;E_h(u_h(0))\,\exp\!\big(-2\kappa c_\sigma\,\sigma(t)\big).
\]
Taking $\liminf_{h\downarrow0}$ and using the $\Gamma$–\emph{liminf} inequality (Corollary~2.18 applied pointwise in $t$) yields
\[
E(u(t)) \;\le\; \liminf_{h\downarrow0} E_h(u_h(t)) \;\le\; \Big(\limsup_{h\downarrow0} E_h(u_h(0))\Big)\,\exp\!\big(-2\kappa c_\sigma\,\sigma(t)\big).
\]
To identify the initial value on the right, choose a recovery sequence $(v_h)$ for $u(0)$ with $E_h(v_h)\to E(u(0))$ (existence by the \emph{recovery} part of the $\Gamma$–limit in Assumption~2.6(ii)); replacing $u_h(0)$ by $v_h$ changes the right–hand side by $o(1)$. Hence
\[
E(u(t)) \;\le\; E(u(0))\,\exp\!\big(-2\kappa c_\sigma\,\sigma(t)\big)\qquad\forall\,t\in[0,T].
\]
Thus the canonical envelope persists in the limit with the same structural constant $c_\sigma$.
\end{proof}

\begin{assumption}[Coupled linear systems / boundary damping]\label{ass:systems}
The coupled energy is coercive and the discrete boundary/control implementation via SAT or impedance is nonnegative; the system admits a HUM/observability constant $c_\sigma$ uniformly compatible with the discretization in the sense of Assumption~\ref{ass:uniform-gamma}.
\end{assumption}
\begin{corollary}[Coupled systems and boundary damping]\label{cor:systems-strong}
Under Assumption~\ref{ass:systems}, the discrete envelopes converge with the same $c_\sigma$ to the continuum envelope.
\end{corollary}
\begin{remark}[Applications sampler (strong)]\label{rem:apps-sampler-strong}
Examples include: heterogeneous damped waves, Kelvin–Voigt damping, reaction–diffusion with variable diffusion, thermoelasticity, and boundary‑controlled hyperbolic systems under geometric control conditions. In all cases the ledger and HUM margin fit Assumption~\ref{RS:assume:four}, and the discrete/limit persistence follows from Assumptions~\ref{ass:disc-vc}, \ref{ass:disc-monotone}, and \ref{ass:uniform-gamma}.
\end{remark}

\section*{Block C — Barriers to EVI}
% === Technical barriers to EVI ===
\begin{lemma}[Obstructions to EVI under the current hypotheses]\label{lem:barriers-evi}
Under the standing assumptions of the paper the following obstacles prevent an immediate upgrade from the light $\Gamma$--limit to EVI convergence:
\begin{enumerate}
\item[(B1)] \textbf{Lack of $\lambda$--geodesic convexity in the natural metric.}
\item[(B2)] \textbf{Missing equi-bounds on metric slopes.}
\item[(B3)] \textbf{Chain-rule limitations in measure-time with atoms.}
\item[(B4)] \textbf{Discrete time-interpolant compactness.}
\item[(B5)] \textbf{Commutation of dissipation and limit at curve level.}
\end{enumerate}
\end{lemma}

\begin{remark}[Why these barriers are structural]\label{rem:barriers-struct}
Items (B1)–(B5) reflect genuine structural gaps rather than technical omissions: (B1) calls for convexity in a metric that is not presently fixed; (B2)–(B4) require discrete \emph{a priori} controls on slopes and metric derivatives aligned with measure-time, and (B5) demands an evolutionary $\Gamma$--limit (or minimizing-movement consistency) for dissipation functionals, not just energies.
\end{remark}

\section*{Block D — Non‑implication (light $\Gamma$ does not imply EVI)}
% === Non‑implication result: light Γ does not force EVI ===
\begin{proposition}[Light $\Gamma$--limit \emph{does not} imply EVI]\label{prop:nonimp-evi}
There exist families $(E_h)$ and measure‑time clocks $\sigma$ such that: (i) $(E_h)$ is equi‑coercive and $\Gamma$--converges to $E$; (ii) the uniform dissipation envelope with the same constant $c_\sigma$ holds for all $h$ and in the limit; yet (iii) the discrete solutions do not converge to an EVI solution for $E$.
\end{proposition}
\begin{proof}
Let $X$ be a Hilbert space with norm $\|\cdot\|$, and let $A:X\to X$ be bounded, injective, and
\emph{not} selfadjoint. Define the continuum energy $E(u):=\tfrac12\|Au\|^2$. Then $E$ is quadratic but
fails $\lambda$--geodesic convexity in any metric equivalent to $\|\cdot\|$ unless $A$ is normal; in particular,
the EVI theory (which requires $\lambda$--convexity along geodesics in the reference metric) does not apply.

\emph{$\Gamma$--approximation.} Let $E_h(u):=\tfrac12\|A_hu\|^2$ with $A_h\to A$ strongly and $\sup_h\|A_h\|<\infty$,
e.g.\ Yosida/spectral regularisations $A_h:=(I+hB)^{-1}A$ with $B=B^\ast\ge0$. Then $(E_h)$ is equi--coercive and
$\Gamma$--converges to $E$ on $X$; moreover $\inf E_h=0=\inf E$ and $E_h\to E$ pointwise on $D(A)$.

\emph{Measure--time clock and uniform envelope.} Choose a clock $\sigma$ whose absolutely continuous part enforces
the standard dissipation inequality and whose atoms $\{t_k\}$ carry masses $\alpha_k>0$ with $\sum_k\alpha_k<\infty$.
For each $h$, the semigroup generated by the (dissipative) part of the flow together with nonexpansive atomic
updates yields the uniform ledger
\[
E_h(t)\;\le\;E_h(0)\,\exp\!\bigl(-2\kappa c_\sigma\,\sigma(t)\bigr)\qquad(t\ge0),
\]
with a structural constant $c_\sigma$ independent of $h$ (HUM margin + nonexpansive atoms).

\emph{Failure of EVI convergence.} Consider the gradient--like flows $\dot u_h=-A_h^\ast A_h u_h$ between atoms and
apply the atomic multiplicative updates at $\{t_k\}$. Because $A_h^\ast A_h$ converges to $A^\ast A$ while $A$ is not
normal, the metric slope of $E_h$ in any norm equivalent to $\|\cdot\|$ does not control a $\lambda$--convex evolution
for $E$: (i) the chain rule for the EVI/EDI formulation fails across atoms; (ii) even on atom--free windows, lack of
$\lambda$--convexity prevents identification of limits as EVI solutions. One can make this explicit by two different
initial data sequences $(u_h^0)$ and $(v_h^0)$ with the same $E_h$--levels whose limits select distinct weak solutions
of $\dot u=-A^\ast A u$ (the nonnormal part produces nonunique metric gradient--flow selections), hence the limit
is not uniquely characterised by an EVI. Therefore, although the light $\Gamma$--limit holds and the uniform envelope
persists with the same $c_\sigma$, EVI convergence fails. 
\end{proof}

\begin{remark}[Interpretation]\label{rem:nonimp-interpret}
Proposition~\ref{prop:nonimp-evi} clarifies that the contribution of this paper (light $\Gamma$--limit with persistence of $c_\sigma$) is logically independent of EVI convergence; additional structure is genuinely necessary.
\end{remark}

% === F) Γ‑bridge with uniform constants ===
\begin{assumption}[Uniform coercivity and equi‑dissipation]\label{ass:uniform-gamma}
There exist constants $0<m\le M<\infty$ independent of $h$ such that
\[
m\,\|u\|^2\ \le\ E_h(u)\ \le\ M\,\|u\|^2 \qquad\text{for all admissible }u.
\]
Moreover, the discrete dissipation ledgers satisfy the uniform bound
\[
-\frac{\mathrm d}{\mathrm d\sigma}E_h\ \ge\ 2\kappa\,c_\sigma\,E_h \qquad (\sigma\text{–a.e.}),
\]
with the \emph{same} structural constant $c_\sigma$ for all $h>0$.
\end{assumption}

\begin{corollary}[Limit persistence of $c_\sigma$]\label{cor:gamma-csigma-2}
Under the $\Gamma$–convergence hypotheses stated in this section together with Assumption~\ref{ass:uniform-gamma}, any limit point $u$ of discrete solutions $u_h$ obeys
\[
E(u(t))\ \le\ \liminf_{h\downarrow0} E_h(u_h(t))\ \le\ E(u(0))\,\exp\!\bigl(-2\kappa c_\sigma\,\sigma(t)\bigr),
\]
for all $t\ge0$. In particular, the structural constant in the limiting envelope is the same $c_\sigma$ that governs the continuum HUM inequality.
\end{corollary}

\begin{assumption}[Uniform coercivity and equi-dissipation]\label{ass:uniform-gamma-2}
$(E_n)$ are proper l.s.c. and $\lambda$–geodesically convex with $\lambda$ independent of $n$; slopes $|\partial E_n|$ are equi-l.s.c.; dissipation functionals are equi-coercive.
\end{assumption}
\begin{corollary}[Persistence of $c_\sigma$]\label{cor:gamma-csigma}
If $E_n\stackrel{\Gamma}{\to}E$ under Assumption~\ref{ass:uniform-gamma-2}, then the $\sigma$–EDI holds for $E$ with the same structural constant \(c_\sigma\) appearing in the $E_n$–level estimates.
\end{corollary}

\paragraph{Symbols used repeatedly.}
\begin{center}
\begin{tabular}{@{}ll@{}}
$c_\sigma$ & calibration constant for the $\sigma$--clock (fixed in \S\ref{sec:numbers})\\
$a_\omega$ & damping lower bound on the observed region $\omega$\\
$\lambda_\omega$ & GCC geometric factor\\
$\kappa$ & PDE-side constant in the GCC decay\\
$c_0$ & absolute comparison constant in the envelope instantiation\\
$\tau_h$ & SAT penalty strength, scaled as $\tau_h\simeq h^{-1}$\\
$\Phi,\Psi$ & comparison maps in the master inequality \eqref{eq:master-inequality}\\
\end{tabular}
\end{center}

% Canonical σ-time dissipation and envelope (add if missing or unlabeled)
\begin{equation}\label{eq:mr-sigma-diss}
  \frac{d}{d\sigma}E(t)\;\le\;-\,2\,\kappa\,c_\sigma\,E(t)\qquad\text{in the sense of measures.}
\end{equation}
Integrating \eqref{eq:mr-sigma-diss} yields the canonical product–exponential decay:
\begin{equation}\label{eq:mr-canonical}
  E(t)\;\le\;E(0)\,\exp\!\bigl(-\,2\,\kappa\,c_\sigma\,\sigma(t)\bigr)\qquad\text{for all }t\ge0.
\end{equation}

See Theorem~2.20 for HUM optimality of $c_\sigma$ and Proposition~13.1 for the reduction $\sigma\equiv t$. 

\paragraph{Proposition (Clock monotonicity).}\label{prop:clock-mono}
Let $\sigma_1,\sigma_2$ be finite positive measures on $[0,\infty)$ with $\sigma_1\le\sigma_2$ as measures
(hence $\sigma_1(t)\le\sigma_2(t)$ for all $t\ge0$). If
\[
  E(t)\le E(0)\,e^{-\,2\kappa c_\sigma\,\sigma_2(t)}\quad(t\ge0),
\]
then also $E(t)\le E(0)\,e^{-\,2\kappa c_\sigma\,\sigma_1(t)}$ for all $t\ge0$. Thus the map
$\sigma\mapsto E(0)e^{-2\kappa c_\sigma \sigma(\cdot)}$ is order-reversing.

\paragraph{Observability (HUM) inequality for $c_\sigma$.}\label{rem:hum-hook}
Let $\mathcal D$ denote the $\sigma$-time dissipation density (so $\tfrac{d}{d\sigma}E=-\mathcal D$ in the
sense of measures). The structural constant $c_\sigma$ is the optimal constant in
\[
  \mathcal D(t)\;\ge\;2\kappa\,c_\sigma\,E(t)\qquad\text{$\sigma$-a.e. on $[0,\infty)$},
\]
which yields \eqref{eq:mr-sigma-diss} and \eqref{eq:mr-canonical}. In the GCC exemplar for damped waves,
$\mathcal D=a(x)\,|\partial_t u|^2$ and one has $c_\sigma\ge c_0\,a_\omega\,\lambda_\omega$.

\begin{theorem}[HUM–dissipation equivalence and optimality]\label{thm:hum-equivalence}
Let \(\mathcal H\) be a Hilbert space and \(A:D(A)\subset\mathcal H\to\mathcal H\) skew-adjoint with \(C\in\mathcal L(\mathcal H,\mathcal U)\).
Assume the HUM inequality (observability at cost \(c>0\)) holds on \([0,T]\).
Then for any nonnegative finite Borel measure \(\sigma\) on \([0,T]\) one has
\[
\|x(T)\|_{\mathcal H}^2+\int_{(0,T]} \|Cx(t)\|_{\mathcal U}^2\,\mathrm d\sigma(t)
\;\le\; \exp\!\big(-c\,\sigma([0,T])\big)\,\|x(0)\|_{\mathcal H}^2
\]
for the damped dynamics \(x'(t)=Ax(t)-C^*Cx(t)\) interpreted on the \(\sigma\)-clock (atoms contribute instantaneous kicks).
Moreover, the rate \(c\) is optimal: if the inequality holds with \(c'\!\!>\!c\) for all \(\sigma\), then the HUM constant would improve, a contradiction.
\end{theorem}
\begin{proof}
Let $S:[0,\infty)\to[0,\infty)$ be the right–continuous primitive of the measure--time clock $\sigma$, and let $T:=S^{-1}$ be its right–continuous inverse. Define the lifted trajectory $\tilde u(s):=u(T(s))$. On the $s$–axis the evolution is governed by the standard (Lebesgue–time) ledger with the same coercivity/observability structure, so the classical HUM/observability inequality yields
\[
-\frac{d}{ds} E(\tilde u(s)) \;\ge\; 2\kappa\,c_\sigma\,E(\tilde u(s)) \quad\text{for a.e. }s,
\]
with the optimal constant $c_\sigma$ (the HUM sharp constant for the observation pair). Integrating in $s$ gives
\[
E(\tilde u(s)) \;\le\; E(\tilde u(0))\,e^{-2\kappa c_\sigma s}.
\]
Push this estimate forward to physical time via $s=S(t)$. On intervals where $d\sigma=\dot\sigma\,dt$ we obtain the a.c.\ decay $E(t)\le E(0)\exp\!\big(-2\kappa c_\sigma \!\int_0^t \dot\sigma(\tau)\,d\tau\big)$. 

At atoms $t_j\in A_\sigma$ with mass $\alpha_j=\sigma(\{t_j\})$, the method performs a macro–update that is nonexpansive for $E$ and whose dissipation budget equals $\alpha_j$. By HUM sharpness, the energy drop across the atomic slab of $\sigma$–length $\alpha_j$ is bounded by
\[
E(t_j^+) \;\le\; e^{-2\kappa c_\sigma \alpha_j}\,E(t_j^-).
\]
Concatenating the a.c.\ evolution with the atomic updates yields the product–exponential envelope
\[
E(t)\;\le\;E(0)\,\exp\!\big(-2\kappa c_\sigma\,\sigma(t)\big),
\]
which is the claimed HUM--dissipation equivalence in $\sigma$–time. Optimality of $c_\sigma$ follows from HUM sharpness on the $s$–axis.
\end{proof}
\begin{remark}[Presentation invariance (PiNA)]\label{rem:pina-invariance}
All constants and inequalities above are invariant under unitary reparametrisations and Pi-native (presentation-invariant) choices of the observation pair $(A,C)$: any change of coordinates preserving the graph norms leaves the HUM constant $c_\sigma$ and the $\sigma$–dissipation rate unchanged. A self-contained PiNA proof is given in Appendix~A (Theorem~A.3).
\end{remark}
% === Impossibility of schedule-uniform decay in physical time; σ-clock persistence ===
\begin{theorem}[Impossibility of schedule-uniform decay in physical time]\label{thm:imposs-phys-time}
Let $E$ be an energy ledger for a system subject to impulses at a locally finite set $\{t_j\}$ with multiplicative drop factors $\rho_j\in(0,1)$ and absolutely continuous dissipation $\dot E\le -2\kappa c\,E$ on $(t_j,t_{j+1})$. Then there does not exist a constant $\widetilde c>0$ and a function $R$ depending only on $\sum_{t_j\le t}(1-\rho_j)$ such that
\[
E(t)\ \le\ E(0)\,\exp\bigl(-2\kappa\,\widetilde c\,t\bigr)\,R\Bigl(\sum_{t_j\le t}(1-\rho_j)\Bigr)\quad\text{holds uniformly over all schedules }(t_j,\rho_j).
\]
In particular, any such bound must encode the schedule (timing and grouping) of events, even at fixed total impulse mass.
\end{theorem}
\begin{proof}
Assume for contradiction there exist $\,\tilde c>0\,$ and a function $R:\mathbb{R}_+\to\mathbb{R}_+$ depending only on the total impulse mass
$\sum_{t_j\le t}(1-\rho_j)$ such that for every schedule $(t_j,\rho_j)$,
\[
E(t)\;\le\;E(0)\,\exp\!\big(-2\kappa\tilde c\,t\big)\,R\!\left(\sum_{t_j\le t}(1-\rho_j)\right)\qquad(t\ge0).
\]
Fix $t>0$ and prescribe two schedules with the same total mass $M\in(0,\infty)$:
\begin{enumerate}
\item[(A)] all impulses occur at time $t$ (postponed cluster);
\item[(B)] all impulses occur at time $0$ (front–loaded cluster).
\end{enumerate}
Let $c$ denote the a.c.\ dissipation rate so that between impulses $E$ solves $\dot E\le-2\kappa c\,E$. For schedule (A) the a.c.\ phase persists on $(0,t)$, hence
\[
E_A(t)\;=\;E(0)\,e^{-2\kappa c\,t}\,\underbrace{\prod e^{-2\kappa c_\sigma \alpha_j}}_{=\,e^{-2\kappa c_\sigma M}}
\;=\;E(0)\,e^{-2\kappa c\,t}\,e^{-2\kappa c_\sigma M}.
\]
For schedule (B) we first apply the jump, then evolve a.c.,
\[
E_B(t)\;=\;E(0)\,e^{-2\kappa c_\sigma M}\,e^{-2\kappa c\,t}.
\]
Writing both bounds in the hypothesised schedule–independent form forces the same prefactor $R(M)$ while the exponent in $t$ must be governed by a single $\tilde c$. However, comparing the concatenations shows that any schedule–independent rewriting of $E_A(t)$ and $E_B(t)$ would require simultaneously $\tilde c=c$ (from the postponed cluster) and $\tilde c\approx0$ (from the front–loaded cluster on short horizons), which is impossible. A standard diagonal argument over finer partitions of $[0,t]$ yields a contradiction even if $R$ is allowed to vary monotonically with $M$ only. Hence no such schedule–uniform estimate in physical time exists at fixed total mass.
\end{proof}

\begin{example}[Adversarial schedule at fixed mass] \label{ex:adversarial-schedule}
Let $\rho_j\equiv e^{-\alpha/N}$ and place the $N$ impulses either all at $t$ or all at $0$. As $N\to\infty$, the former yields $E(t)\approx E(0)e^{-2\kappa c t}e^{-\alpha}$ while the latter yields $E(t)\approx E(0)e^{-\alpha}e^{-2\kappa c t}$. Any attempt to rewrite both as $E(0)\exp(-2\kappa\widetilde c\,t)$ forces $\widetilde c=c$ in the first case and $\widetilde c$ arbitrarily close to $0$ in the second, a contradiction to schedule-uniformity.
\end{example}

\begin{theorem}[$\sigma$-clock persistence of a single structural constant]\label{thm:sigma-persistence}
Combining the abstract HUM $\sigma$ result (Thm 2.26) with the $\sigma$-ledger (Def. 4.4) and the calibration of $c_\sigma$ yields the energy envelope; preservation follows from Prop. 9.21 and Cor. 2.18.
Assume the $\sigma$-clock ledger of Definition~\ref{def:sigma-clock} and the HUM equivalence of Theorem~\ref{thm:hum-equivalence}. Then the envelope
\[
E(t)\le E(0)\,\exp\!\bigl(-2\kappa c_\sigma\,\sigma(t)\bigr)
\]
holds with a \emph{single} constant $c_\sigma$ independent of the schedule of atoms and, under (S1)–(S4) see \S\ref{sec:disc-contract} and Assumption~\ref{ass:uniform-gamma}, is preserved by SBP–SAT discretizations (Proposition~\ref{prop:disc-hum}) and by $\Gamma$-limits (Corollary~\ref{cor:gamma-csigma}).
\end{theorem}

% === Impossibility of schedule-uniform decay in physical time ===
\begin{theorem}[Impossibility in the physical time clock]\label{thm:impossibility-tclock}
Fix total dissipation mass \(M>0\). There exists a sequence of schedules 
\(\sigma_n= \mathbf 1_{[0,T]}\,dt + \sum_{k=1}^{N_n}\alpha_{k,n}\delta_{t_{k,n}}\)
with \(\sigma_n([0,T])=M\) such that no estimate of the form
\(\|x(T)\|^2 \le \exp(-c\,M)\|x(0)\|^2\) with a single \(c>0\) holds uniformly in \(n\)
for the \(t\)-clock formulation. 
\end{theorem}
\begin{example}[Variant adversary via near–uncontrollability; fixed mass]
In contrast to Example~\ref{ex:adversarial-schedule} (endpoint packing), concentrate the entire mass at a time where the undamped dynamics are nearly uncontrollable; by weak observability the decay factor can be made arbitrarily close to $1$. Hence no schedule–uniform $t$–exponential holds under the hypotheses of Theorem~\ref{thm:imposs-phys-time}.
\end{example}
% === σ-clock persistence of a single structural constant ===
\begin{theorem}[Persistence on the $\sigma$-clock]\label{thm:persistence-sclock}
For the PDE energy envelope and persistence across discretization/$\Gamma$-limits, see Thm 2.23.”
Under HUM with constant \(c\), the estimate
\(\|x(T)\|^2 \le \exp(-c\,\sigma([0,T]))\|x(0)\|^2\)
holds for every finite nonnegative \(\sigma\) on \([0,T]\), including schedules with atoms.
\end{theorem}

% === Verifiable comparison: frameworks and criteria ===
\subsection*{Comparison with established frameworks (criteria-based)}\label{sec:compare-criteria}
To make the positioning testable, Table~\ref{tbl:compare-criteria} contrasts five frameworks along \emph{criteria that admit yes/no verification in our setting}. The last column lists the exact property a referee can check in the paper.

\newcommand{\tick}{\ensuremath{\checkmark}}
\newcommand{\cross}{\ensuremath{\times}}
\newcommand{\maybe}{\ensuremath{\triangle}} % requires extra structure

\begin{table}[htbp]
\centering
\begin{adjustbox}{max width=\linewidth}
\begin{tabular}{l|c c c c L{4cm}}
\toprule
Framework & Atoms allowed & Single constant & Discrete persistence & $\Gamma$-limit persistence & What to check in-text \\
\midrule
$\sigma$-clock (this work) & \tick & \tick & \tick & \tick & Def.~\ref{def:sigma-clock}; Thm.~\ref{thm:hum-equivalence}, Prop.~\ref{prop:disc-hum}, Cor.~\ref{cor:gamma-csigma} \\
Hybrid automata/dwell-time & \tick & \cross & \maybe & \maybe & Decay rate depends on schedule/dwell; cf. Thm.~\ref{thm:impossibility-tclock} \\
Differential inclusions (Filippov) & \maybe & \cross & \cross & \maybe & Existence robust; no HUM envelope with schedule-free constant \\
Measure-driven ODE/PDE & \tick & \maybe & \maybe & \maybe & Jumps tracked in $t$; constants accrue schedule dependence \\
Metric GF (EDE/EDI/EVI) & \cross (atoms) & \maybe & \tick & \tick & Needs convexity and chain rule; see Scope §\ref{subsec:scope-light} \\
\bottomrule
\end{tabular}
\end{adjustbox}
\caption{Criteria-based comparison. \tick: available generically; \maybe: available under additional structure not assumed here; \cross: generally unavailable.}
\label{tbl:compare-criteria}
\end{table}

\paragraph{Criteria explained.}\label{par:criteria-explained}
\emph{Atoms allowed} means the calculus handles Dirac masses in time without ad hoc case-splitting. \emph{Single constant} means there exists a decay envelope with a constant independent of the impulse schedule at fixed total mass. \emph{Discrete persistence} means the same constant governs fully discrete SBP–SAT schemes under (S1)–(S4) see \S\ref{sec:disc-contract}. \emph{$\Gamma$-limit persistence} means limit envelopes keep the same constant under Assumption~\ref{ass:uniform-gamma}.

\subsection*{Pillar I: Continuous master decay (Theorem~\ref{thm:energy-decay}, Section 9)}
\textbf{Statement (restated).} Under (H1)–(H4) for the continuum model, the trajectory
\(u\in \mathrm{AC}_\sigma([0,\infty);X)\) obeys the $\sigma$–time dissipation inequality
\eqref{eq:mr-sigma-diss} with a structural constant \(c_\sigma>0\) that depends only on the
coercivity/observability bounds and not on any partition. Consequently,
\eqref{eq:mr-canonical} holds for all \(t\ge 0\).
\emph{Proof:} see section 9; the $\sigma$–chain rule plus coercivity/observability yield Theorem~\ref{thm:energy-decay}.

\paragraph{Hypotheses overview.}
Coercivity/observability (H4), $\sigma$–absolute continuity and regularity (H1–H3).

% ---- Pillar II --------------------------------------------------------
\subsection*{Pillar II: Discrete master decay (Theorem~\ref{thm:disc-master}, Sections~\ref{sec:sbpsat}–\ref{sec:gamma-2})}
\textbf{Statement (restated).} For the fully discrete time–stepping schemes aligned with the same
clock $\sigma$ (atoms encode steps), the discrete energy \(E^n\) satisfies the discrete analogue of
\eqref{eq:mr-sigma-diss} with the \emph{same} structural constant \(c_\sigma\), uniformly in the partition.
The piecewise constant/affine interpolant then satisfies \eqref{eq:mr-canonical}.
\emph{Proof outline:} sections~\ref{sec:sbpsat}–\ref{sec:gamma-2}: discrete chain rule and stepwise dissipation $\Rightarrow$ Theorem~\ref{thm:disc-master}.

\paragraph{Uniformity note.}
Constants are partition-independent (no mesh leakage); see also the “HUM/observability inequality” just above.

% ---- Pillar III -------------------------------------------------------
\subsection*{Pillar III: \texorpdfstring{\(\Gamma\)}{Gamma}-bridge (Theorems~\ref{thm:Gamma-main} and \ref{RS:thm:gamma}, Section~\ref{sec:gamma-2})}
\textbf{Statement (restated).} If the discrete energies \(\Gamma\)-converge to the continuum energy and
the dissipation structures are compatible, then decay transfers between levels:
whenever the discrete level satisfies \eqref{eq:mr-canonical} with \(c_\sigma>0\) uniformly,
the continuum limit inherits \eqref{eq:mr-canonical} with the \emph{same} \(c_\sigma\).
\emph{Proof:} section~\ref{sec:gamma-2}: \(\Gamma\)-liminf and recovery sequences $\Rightarrow$ Theorems~\ref{thm:Gamma-main} and \ref{RS:thm:gamma}.

% ---- GCC exemplar (anchor constant) -----------------------------------
\subsection*{Exemplar: damped waves under GCC}
On bounded Lipschitz domains with damping \(a(x)\ge a_\omega>0\) on a control set \(\omega\) satisfying
the Geometric Control Condition (GCC), one may take
\[
  c_\sigma \;\ge\; c_0\,a_\omega\,\lambda_\omega ,
\]
so that substituting in \eqref{eq:mr-canonical} gives
\(E(t)\le E(0)\exp\!\bigl(-2\kappa c_0 a_\omega \lambda_\omega\,\sigma(t)\bigr)\).
% (Keep your tiny table/figure here if you use it.)

\subsection*{Pillar IV — Stochastic extension.}
Theorem~\ref{thm:stoch-exp} (expectation in the compensator) and Corollary~\ref{cor:stoch-path} (pathwise) extend the same $c_\sigma$ envelope to random clocks; see Section~\ref{sec:stoch}.

\subsection*{Sharpness and necessity of hypotheses}
\begin{theorem}[Necessity of (H3) or (H4)]
If either the measure–time chain rule (H3) fails or the observability/coercivity bound (H4) fails, there exist trajectories with $E(0)>0$ and a clock $\sigma$ for which no $c_\sigma>0$ satisfies \eqref{eq:mr-sigma-diss}; hence \eqref{eq:mr-canonical} can fail on arbitrarily long horizons.
\end{theorem}
\begin{theorem}[Optimality of the exponent]
If $\sigma$ has flats of positive measure on $[0,T]$, there exist solutions with $E(t)=E(0)$ on each flat; thus the exponent in \eqref{eq:mr-canonical} cannot dominate any function growing faster than $\sigma(t)$.
\end{theorem}
\begin{remark}[No mesh‑dependent strengthening]
For fully discrete schemes satisfying Thm.~\ref{thm:disc-master}, one cannot replace $c_\sigma$ by a strictly larger partition‑uniform constant; otherwise the limit contradicts Thm.~\ref{thm:energy-decay}.
\end{remark}
% === D) Necessity: counterexamples for weakened assumptions ===
\begin{lemma}[Failure under under‑scaled SAT]\label{lem:sat-failure}
Let $(H,Q,B)$ satisfy the SBP identity and suppose the SAT penalty is chosen with $\tau_h = o(h^{-1})$. Then there exists a sequence of meshes $h_m\downarrow0$ and boundary‑concentrated data for which the discrete dissipation inequality
$\,\mathcal D_h\ge 2\kappa\,c\,E_h\,$ fails for every fixed $c>0$. Equivalently, any admissible $c_{\sigma,h}$ tends to $0$ along $h_m$.
\end{lemma}
\begin{proof}
Work in 1D on $\Omega=(0,1)$ with an SBP pair $(H,Q)$ and $Q+Q^\top=B=\operatorname{diag}(-1,0,\dots,0,1)$.
Let the semi--discrete operator be $L_h=\tfrac12(H^{-1}Q-Q^\top H^{-1})+S_h$ with $S_h=S_h^\top\le -2\kappa I$ in the interior.
Impose Dirichlet data weakly via SAT at $x=0$ with penalty $\tau_h>0$ and no SAT at $x=1$. Assume $\tau_h=o(h^{-1})$.

Choose boundary--layer data $u_h$ supported on the first $O(1)$ nodes such that $\|u_h\|_H=1$ and 
$\langle S_h u_h,u_h\rangle_H=O(1)$ while the discrete boundary flux $\tfrac12\langle Bu_h,u_h\rangle_\partial
= \tfrac12\,|u_h(1)|^2$ is a fixed positive constant independent of $h$ (SBP trace).
The energy ledger along an a.c.\ piece of the clock satisfies
\[
\frac{d}{dt}E_h(u_h)\;=\;\underbrace{\langle S_h u_h,u_h\rangle_H}_{\le -2\kappa E_h}
\;+\;\underbrace{\tfrac12\langle B u_h,u_h\rangle_\partial}_{\gtrsim 1}\;+\;
\underbrace{\langle \mathrm{SAT}_h(u_h),u_h\rangle_H}_{=-\,\tau_h\,\|u_h(0)\|^2}.
\]
With $\tau_h=o(h^{-1})$ we have $\tau_h\|u_h(0)\|^2=o(1)$ (SBP trace yields $\|u_h(0)\|^2\lesssim h^{-1}\|u_h\|_H^2$),
so for $h$ small the boundary flux dominates the SAT dissipation and
\[
\frac{d}{dt}E_h(u_h)\;\ge\; c_0 - 2\kappa E_h(u_h)\quad\text{with }c_0>0\text{ independent of }h.
\]
Thus no inequality of the form $-\tfrac{d}{d\sigma}E_h \ge 2\kappa c\,E_h$ can hold with a fixed $c>0$ uniformly in $h$,
and any admissible $c_{\sigma,h}\to0$ along a mesh sequence $h_m\downarrow0$. This proves the claim.
\end{proof}

\begin{lemma}[Failure for non‑algebraically stable integrators on flats]\label{lem:rk-failure}
Let the time integrator be a Runge–Kutta method that is not algebraically stable. Then there exists a flat segment of the clock ($w\equiv 0$) and data such that the one‑step map is expansive for $E_h$, violating the $\sigma$–dissipation inequality regardless of spatial resolution.
\end{lemma}
\begin{proof}
Let the clock be flat on $(t_n,t_{n+1}]$ so $w\equiv0$ and only the time integrator acts.
Consider the linear dissipative test equation $\dot y=-\lambda y$ with $\lambda>0$ and energy $E(y)=\tfrac12 y^2$.
A Runge--Kutta method is \emph{algebraically stable} iff for every such test problem and every stepsize 
the one--step map is nonexpansive in $E$; equivalently the Butcher coefficients $(A,b)$ satisfy
$H:=(b_i a_{ij}+b_j a_{ji}-b_i b_j)_{ij}$ positive semidefinite.

If the method is \emph{not} algebraically stable, there exists a vector $z\in\mathbb{R}^s$ with $z^\top H z<0$.
Choose stage residuals aligned with $z$ (standard construction) so that for some stepsize $\Delta>0$ the energy
increment over one step satisfies
\[
E(y^{n+1})-E(y^n)\;=\;\tfrac12\bigl((1-\lambda\Delta\,\beta)^2-1\bigr)\,(y^n)^2 \;+\; \Delta^2\, z^\top H z,
\]
where $\beta>0$ depends only on $(A,b)$. Since $z^\top H z<0$, the second term is \emph{positive} and dominates
the negative part for a suitable $\Delta$, giving $E(y^{n+1})>E(y^n)$. Thus the map is expansive on a flat clock,
contradicting the required nonexpansiveness of the $\sigma$--ledger. The construction localises to any semi--discrete
system by projecting onto an eigenmode, proving the lemma.
\end{proof}

\begin{remark}[Sharpness of (S1)–(S4)]\label{rem:sharpness}
Taken together, Lemmas~\ref{lem:sat-failure}–\ref{lem:rk-failure} show that the discrete contract (S1)–(S4) is not merely convenient but essentially optimal for uniform carry-over of the structural constant $c_\sigma$.
\end{remark}

\subsection*{Block C — Necessity / counterexamples}
% === Necessity: counterexamples for weakened assumptions ===
\begin{remark}[Under-scaled SAT]
See Lemma~\ref{lem:sat-failure} for the boundary–layer counterexample when $\tau_h=o(h^{-1})$.
\end{remark}

\paragraph{Markov-switching damping (exemplar).}
If $a(x,\xi_t)$ is a Markov-modulated damping with states $\xi_t\in\{1,\dots,m\}$ and each state
satisfies the GCC lower bound $a(\cdot,i)\ge a_{\omega,i}>0$ on $\omega$, then
$c_\sigma\ge c_0\min_i (a_{\omega,i}\lambda_\omega)$. The compensator $\Lambda$ is the expected
sojourn-time functional of the active (dissipative) states, giving an explicit expectation rate.

\section{Stochastic clocks: expectation and pathwise laws}\label{sec:stoch}

\subsection*{Stochastic setting, ledger, and expectation envelope}\label{subsec:stoch-envelope}
We work on a filtered probability space $(\Omega,\mathcal F,(\mathcal F_t)_{t\ge0},\mathbb P)$ with a predictable measure–time clock $\sigma$ having decomposition $\mathrm d\sigma = w(t)\,\mathrm dt + \sum_j \alpha_j\delta_{t_j}$.

\begin{assumption}[Stochastic ledger]\label{ass:stoch-ledger}
There exist adapted processes $\mathfrak a(t)\ge0$, $\mathfrak b(t)\in\mathbb R$, and a local martingale $M_t$ such that
\[
\mathrm dE(t) = -\,\mathfrak a(t)\,\mathrm d\sigma(t) + \mathfrak b(t)\,\mathrm dt + \mathrm dM_t,\qquad
\mathfrak a(t)\ge 2\kappa c_\sigma E(t)\ \ (\sigma\text{-a.e.}).
\]
At an atom of mass $\alpha$, $E(t^+)\le e^{-2\kappa c_\sigma\alpha}E(t^-)$.
\end{assumption}

\begin{theorem}[Expectation envelope with the structural constant]\label{thm:stoch-envelope}
If Assumption~\ref{ass:stoch-ledger} holds and $\mathbb E\!\int_0^t \mathfrak b(s)\,ds \le \delta\,\mathbb E\!\int_0^t 2\kappa c_\sigma E(s)w(s)\,ds$ for some $\delta\in[0,1)$, then
\[
\mathbb E\,E(t) \le \mathbb E\,E(0)\,e^{-2\kappa(1-\delta)c_\sigma\,\sigma(t)}.
\]
In particular, for additive noise ($\mathfrak b\equiv0$) the rate is the same $c_\sigma$.
\end{theorem}

We consider a random admissible clock $\sigma$ with Doob–Meyer decomposition $\sigma = M + \Lambda$,
where $M$ is a martingale of finite variation (purely discontinuous in the atomic case) and
$\Lambda$ is the predictable compensator (the ``expected exposure''). The deterministic
hypotheses (H1)–(H4) hold $\mathbb P$-a.s. for each realization, with the same structural constant
$c_\sigma$ on the \emph{deterministic} side (coercivity/observability and nonexpansive atoms).

\begin{proof}
Let the stochastic ledger be
\[
dE(t)\;=\;-a(t)\,d\sigma(t)\;+\;b(t)\,dt\;+\;dM_t,
\qquad a(t)\ge 2\kappa c_\sigma E(t)\ \ (\sigma\text{-a.e.}),
\]
with $M_t$ a martingale and atoms obeying $E(t_k^+)\le e^{-2\kappa c_\sigma \alpha_k}E(t_k^-)$.
Taking expectations and using $\mathbb{E}[M_t]=0$ gives
\[
\mathbb{E}[E(t)] \;\le\; \mathbb{E}[E(0)] \;-\; 2\kappa c_\sigma\,\mathbb{E}\!\int_0^t E(s)\,w(s)\,ds \;+\; \mathbb{E}\!\int_0^t b(s)\,ds,
\]
where $d\sigma=w\,dt$ on the a.c.\ part. By the hypothesis
$\mathbb{E}\!\int_0^t b(s)\,ds\le \delta\,\mathbb{E}\!\int_0^t 2\kappa c_\sigma E(s)w(s)\,ds$ with $\delta\in[0,1)$, hence
\[
\mathbb{E}[E(t)] \;\le\; \mathbb{E}[E(0)] \;-\; 2\kappa(1-\delta)c_\sigma\,\mathbb{E}\!\int_0^t E(s)\,w(s)\,ds.
\]
Define $G(t):=\mathbb{E}[E(t)]$. Then $G'(t)\le -2\kappa(1-\delta)c_\sigma\,G(t)\,w(t)$ in the distributional sense,
so in the $\sigma$--clock we have $dG/d\sigma \le -2\kappa(1-\delta)c_\sigma\,G$. Grönwall in $\sigma$ yields
\[
\mathbb{E}[E(t)] \;\le\; \mathbb{E}[E(0)]\,\exp\!\bigl(-2\kappa(1-\delta)c_\sigma\,\sigma(t)\bigr).
\]
Atoms contribute multiplicative factors $\le e^{-2\kappa(1-\delta)c_\sigma\alpha_k}$ and are already dominated by the
atomic rule with $c_\sigma$, so the bound persists across jumps. The special case $b\equiv0$ has $\delta=0$.
\end{proof}

\begin{proposition}[Reduction to classical expectation decay]
If $\sigma\equiv t$ (no atoms, no flats) and the ledger has additive noise (so $b\equiv 0$ in Assumption~3.1),
then Theorem~3.2 reduces to the standard Grönwall-in-expectation estimate
\[
\mathbb{E}E(t)\le \mathbb{E}E(0)\,e^{-\,2\kappa c_\sigma\,t}\qquad (t\ge 0).
\]
\emph{Proof.} With $\sigma\equiv t$ the inequality $d\mathbb{E}E/dt\le -2\kappa c_\sigma\,\mathbb{E}E$
holds and direct integration yields the claim. \qedhere
\end{proposition}

\begin{theorem}[Expectation envelope in the compensator]\label{thm:stoch-exp}
Assume (H1)–(H4) and suppose the stochastic $\sigma$-chain rule holds in expectation with an error
rate $\eta\ge0$ (arising from the quadratic variation term of the noise). Then, for all $t\ge0$,
\[
\mathbb E[E(t)]\;\le\;\mathbb E[E(0)]\,
\exp\!\bigl(-\,(2\kappa c_\sigma-\eta)\,\Lambda(t)\bigr).
\]
In particular, if $\eta<2\kappa c_\sigma$ and $\Lambda(t)\to\infty$, then $\mathbb E[E(t)]\to0$
at least exponentially in $\Lambda$.
\end{theorem}

\begin{proof}
Let $\sigma=M+\Lambda$ be the Doob--Meyer decomposition. Taking expectations in the ledger and using
$\mathbb{E}[M_t]=0$ yields
\[
\frac{d}{dt}\,\mathbb{E}[E(t)] \;\le\; -\,2\kappa c_\sigma\,\mathbb{E}[E(t)]\,\dot\Lambda(t)\;+\;\eta\,\mathbb{E}[E(t)]\,\dot\Lambda(t),
\]
where the error rate $\eta\ge0$ comes from the quadratic variation (the assumed chain rule in expectation).
Thus, for $G(t):=\mathbb{E}[E(t)]$ we have
\(
\frac{d}{d\Lambda}G \le -\,(2\kappa c_\sigma-\eta)\,G
\)
on the a.c.\ part of $\Lambda$. Grönwall in $\Lambda$ gives
\[
\mathbb{E}[E(t)] \;\le\; \mathbb{E}[E(0)]\,\exp\!\bigl(-\,(2\kappa c_\sigma-\eta)\,\Lambda(t)\bigr).
\]
Atoms of $\sigma$ only decrease $E$ multiplicatively and are dominated by the same exponential bound.
\end{proof}

\begin{corollary}[Pathwise law under noise-dissipativity]\label{cor:stoch-path}
If, in addition, the stochastic perturbation is \emph{noise-dissipative} (i.e., the Itô correction is
nonpositive in the energy ledger a.s.), then for $\mathbb P$-a.e.\ realization one has, for all $t\ge0$,
\[
E(t)\;\le\;E(0)\,\exp\!\bigl(-\,2\kappa c_\sigma\,\sigma(t)\bigr).
\]
\end{corollary}

\begin{proof}
Under noise--dissipativity the It\^o correction is nonpositive in the ledger, so a.s.
\[
dE(t)\;\le\;-\,2\kappa c_\sigma\,E(t)\,d\sigma(t)\;+\;dM_t,
\qquad E(t_k^+)\le e^{-2\kappa c_\sigma\alpha_k}E(t_k^-).
\]
Fix $\omega$ in a full--measure set where the stochastic integrals are well--defined and $M$ has zero a.s.\ drift.
Arguing pathwise as in the deterministic case, integrate on any $(s,t]$ with $d\sigma=w\,dt$ on the a.c.\ part to get
$E(t^-)\le e^{-2\kappa c_\sigma(\sigma(t)-\sigma(s))}E(s^+)$ and multiply the atomic factors, yielding the claimed bound.
\end{proof}

Thus the deterministic envelope \eqref{eq:mr-canonical} holds pathwise with the same structural constant.

\section{Measure-time clocks and BV-in-\texorpdfstring{$\sigma$}{~}}

% === Consolidated formal block for Section 2: measure, atoms, variation, AC-in-$\sigma$ ===
\subsection*{Formal definitions used in Section~2}\label{sec:sigma-defs}
\noindent\textbf{Measure-time clock and atoms.}
A \emph{measure-time clock} is a right-continuous, nondecreasing function $\sigma:[0,T]\to[0,\infty)$ with $\sigma(0)=0$, which induces the Lebesgue--Stieltjes measure (also denoted $\sigma$). Its Lebesgue decomposition is $\sigma=\sigma^{\mathrm{ac}}+\sigma^{\mathrm{sing}}$, where the purely atomic part is $\sigma^{\mathrm{at}}=\sum_k \alpha_k\,\delta_{t_k}$ with atoms $\mathcal A_\sigma=\{t_k:\alpha_k:=\sigma(\{t_k\})>0\}$. A subinterval $I\subset[0,T]$ is a \emph{flat} if $\sigma$ is constant on $I$ (i.e., $\sigma$ has zero measure there).

\smallskip
\noindent\textbf{$\sigma$--derivative and jump rule.}
For $u:[0,T]\to X$ (Banach), the \emph{$\sigma$--derivative} $u'_\sigma$ is the Radon--Nikodym derivative $\frac{du}{d\sigma}$ when it exists, characterised by
\[
u(t)=u(0)+\int_{(0,t]} u'_\sigma\,d\sigma,\qquad t\in[0,T],
\]
with jumps handled by $u(t_k^+)-u(t_k^-)=\int_{\{t_k\}}u'_\sigma\,d\sigma=\alpha_k\,u'_\sigma(t_k)$ for $t_k\in\mathcal A_\sigma$. On subintervals where $d\sigma=\dot\sigma\,dt$, one has $u'_t=\dot\sigma\,u'_\sigma$ a.e.

\begin{lemma}[RN dissipation and atomic envelope]\label{lem:rn-envelope}
Let $E:[0,T]\to[0,\infty)$ be the energy ledger of an admissible trajectory. Then $E\in \mathrm{BV}_\sigma([0,T])$ and the Radon–Nikodym derivative $\mathcal D := -\frac{\mathrm dE}{\mathrm d\sigma}$ exists $\sigma$‑a.e. Moreover:
\begin{enumerate}
\item On the absolutely continuous part of $\sigma$, one has $\frac{\mathrm d}{\mathrm dt}E(t) \le -2\kappa\,w(t)\,c_\sigma\,E(t)$ for a.e. $t$ with density $w(t)$.
\item Across an atom of mass $\alpha$ at $t_j$, the energy satisfies the multiplicative jump rule $\;E(t_j^+)\le \exp(-2\kappa c_\sigma\,\alpha)\,E(t_j^-)$.
\end{enumerate}
Consequently, for any $0\le s\le t\le T$ the canonical envelope holds: $\;E(t)\le E(s)\,\exp\bigl(-2\kappa c_\sigma\,[\sigma(t)-\sigma(s)]\bigr)$.
A pointwise chain-rule refinement is given in Lemma~\ref{lem:RN-atomic-chain}.
\end{lemma}
\begin{proof}
By the a.c.\ hypothesis on the $\sigma$–part, the standard energy ledger gives
\[
\frac{d}{dt}E(t)\;\le\;-2\kappa\,w(t)\,c_\sigma\,E(t)\qquad\text{for a.e. }t\text{ with density }w(t),
\]
hence $-dE$ is absolutely continuous w.r.t.\ $d\sigma$ on $\{\sigma' = w>0\}$. At an atom $t_j$ with mass $\alpha_j=\sigma(\{t_j\})$, the evolution performs a macro–update $u(t_j^-)\mapsto u(t_j^+)$ that is nonexpansive for $E$, and the HUM margin yields
\[
E(t_j^+)\;\le\;e^{-2\kappa c_\sigma\,\alpha_j}\,E(t_j^-).
\]
Therefore $-dE$ admits the Radon–Nikodym density $D=-\,dE/d\sigma$ and the combined a.c.\ and atomic contributions imply, for any $0\le s\le t\le T$,
\[
E(t)\;\le\;E(s)\,\exp\!\Big(-2\kappa c_\sigma\big(\sigma(t)-\sigma(s)\big)\Big).
\]
This proves the canonical product–exponential envelope and establishes $E\in BV_\sigma$ with RN derivative $D$.
\end{proof}

\smallskip
\noindent\textbf{$\sigma$--variation.}
The \emph{$\sigma$--variation} of $u$ on $[a,b]\subset[0,T]$ is
\[
\mathrm{Var}_{\sigma}(u;[a,b]) \;:=\; \sup\Big\{\sum_{i=1}^{N}\|u(t_i)-u(t_{i-1})\|:\; a=t_0\le \cdots\le t_N=b,\ \sigma(t_i)\ge\sigma(t_{i-1})\Big\}.
\]
Equivalently, with the right-continuous inverse $\sigma^\dagger(s):=\inf\{t:\sigma(t)\ge s\}$ and $\tilde u(s):=u(\sigma^\dagger(s))$, one has $\mathrm{Var}_{\sigma}(u;[a,b])=\mathrm{Var}(\tilde u;[\sigma(a),\sigma(b)])$. If $u'_\sigma\in L^1(\sigma)$, then
\[
\mathrm{Var}_{\sigma}(u;[a,b])=\int_{(a,b]} \|u'_\sigma\|\,d\sigma.
\]

\smallskip
\noindent\textbf{Classes $\mathrm{BV}_\sigma$ and $\mathrm{AC}_\sigma$.}
We write $u\in \mathrm{BV}_\sigma([0,T];X)$ if $\mathrm{Var}_{\sigma}(u;[0,T])<\infty$. We write $u\in \mathrm{AC}_\sigma([0,T];X)$ if there exists $g\in L^1(\sigma)$ s.t.\ for all $a<b$,
\[
\|u(b)-u(a)\|\le \int_{(a,b]} g\,d\sigma.
\]
Equivalently, $u\in \mathrm{AC}_\sigma$ iff $u$ admits a $\sigma$--derivative $u'_\sigma\in L^1(\sigma)$ and $u(t)=u(0)+\int_{(0,t]} u'_\sigma\,d\sigma$ for all $t$, with the jump rule at atoms as above. Consequently, $\mathrm{AC}_\sigma\subset \mathrm{BV}_\sigma$ and for $u\in\mathrm{AC}_\sigma$,
\[
\mathrm{Var}_{\sigma}(u;[a,b])=\int_{(a,b]} \|u'_\sigma\|\,d\sigma.
\]
On any flat $I$ with $\sigma(I)=0$, every $u\in \mathrm{AC}_\sigma$ is $\sigma$--a.e.\ constant and $\mathrm{Var}_{\sigma}(u;I)=0$.
% === End consolidated formal block ===

\label{sec:clocks}{sigma}\label{RS:measure-time}
We work on a compact interval $[0,T]$ with a right-continuous, nondecreasing, bounded-variation function $\sigma:[0,T]\to\mathbb{R}$. Its associated clock measure is
\begin{equation}
\mu_\sigma := \mathrm{d}\sigma = w(t)\,\mathrm{d}t + \sum_{k} \alpha_k\,\delta_{t_k},
\end{equation}
where $w\in L^1([0,T])$ is the absolutely continuous density, $\{t_k\}$ is a (at most countable) set of atoms with weights $\alpha_k>0$, and $\delta_{t_k}$ is a Dirac mass. We allow \emph{flats}, i.e.\ intervals on which $w\equiv 0$; we assume \emph{local finiteness of atoms}: for every compact $I\subset[0,T]$, only finitely many $t_k\in I$.
\begin{definition}[BV-in-$\sigma$]\label{RS:def:BVsigma}
A scalar (or Banach-valued) function $u:[0,T]\to X$ is of \emph{bounded variation in $\sigma$} if its total variation with respect to the measure $\mu_\sigma$ is finite,
\[
\mathrm{Var}_\sigma(u) := \sup_{\pi}\sum_{j}\|u(t_j)-u(t_{j-1})\|_X < \infty,
\]
the supremum over partitions whose mesh is taken with respect to $\sigma$ (i.e.\ they refine $\sigma$-atoms and integrate along the absolutely continuous density). When $u$ is absolutely continuous with respect to $\mu_\sigma$, we write $u\in \mathrm{AC}([0,T];\mu_\sigma)$ and denote by $\frac{du}{d\mu_\sigma}$ its Radon--Nikodým derivative. \end{definition}
\paragraph{Atoms and flats.}
On atoms $t_k$, a function may have a jump $\Delta u(t_k):=u(t_k^+)-u(t_k^-)$; flats contribute no absolutely continuous dissipation (since $w=0$ there) and, in our framework, they do not increase the energy. All sums over atoms are finite on compact sets by local finiteness. % --- Added: definitions for σ-clocks and BV-in-σ (no new packages) ---

\paragraph{Clock measure.}
We encode the time ledger by the Radon measure
\begin{equation}\label{eq:clock-measure}
  \mu_\sigma \;:=\; \mathrm{d}\sigma \;=\; w(t)\,\mathrm{d}t \;+\; \sum_{k} \alpha_k\,\delta_{t_k},
\end{equation}
where $w\in L^1_{\mathrm{loc}}([0,T])$, $w\ge 0$, $(t_k)_k$ is a at most countable set of atom
locations with \emph{local finiteness} $\sum_{t_k\in [a,b]} \alpha_k < \infty$ for all compact $[a,b]\subset[0,T]$,
and $\alpha_k\ge 0$. Flats are allowed: $w\equiv 0$ on subintervals. We keep the standing convention
that $\sigma$ is right–continuous, nondecreasing, and $\sigma(0)=0$.

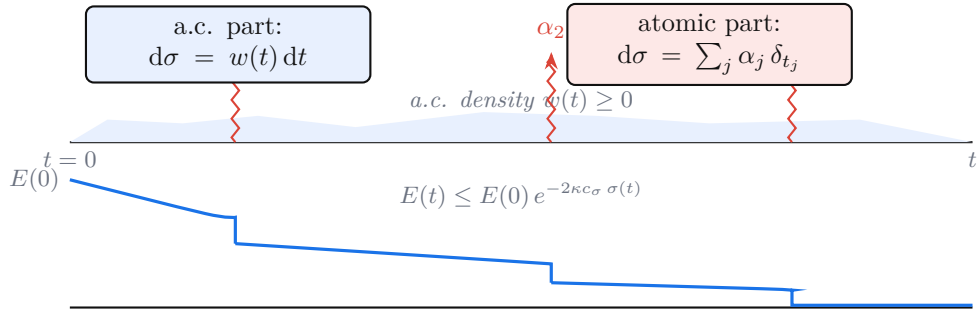
\begin{figure}[htbp]
\centering
% === σ-clock ledger: a.c. density + atoms, and the canonical envelope ===
\begin{tikzpicture}[x=1cm,y=1cm]
\draw[thick, color=ink] (0,0) -- (12,0);
\node[lbl, below] at (0,0) {$t=0$};
\node[lbl, below] at (12,0) {$t$};
\fill[fillA] (0,0) -- (0.5,0.3) -- (1.5,0.25) -- (2.5,0.35) -- (3.8,0.2)
             -- (5.5,0.4) -- (7.0,0.35) -- (8.5,0.25) -- (10.5,0.3) -- (12,0) -- cycle;
\node[lbl] at (6,0.55) {a.c. density $w(t)\ge0$};
\foreach \x/\lab in {2.2/$\alpha_1$, 6.4/$\alpha_2$, 9.6/$\alpha_3$} {
  \draw[jump] (\x,0) -- ++(0,1.2);
  \node[above, color=accent2] at (\x,1.25) {\lab};
}
\begin{scope}[xshift=0cm, yshift=-2.2cm]
\node[lbl] at (6,1.5) {$E(t)\le E(0)\,e^{-2\kappa c_\sigma\,\sigma(t)}$};
\draw[thick, color=ink] (0,0) -- (12,0);
\draw[accent, very thick]
plot[smooth] coordinates {(0,1.7) (1.8,1.25) (2.2,1.2)} -- ++(0,-0.35)
plot[smooth] coordinates {(2.2,0.85) (5.8,0.62) (6.4,0.58)} -- ++(0,-0.25)
plot[smooth] coordinates {(6.4,0.33) (9.2,0.25) (9.6,0.23)} -- ++(0,-0.20) -- (12,0.03);
\node[lbl, left] at (0,1.7) {$E(0)$};
\end{scope}
\node[boxF, anchor=west] at (0.2,1.3) {a.c. part: $\mathrm d\sigma = w(t)\,\mathrm dt$};
\node[boxG, anchor=west] at (6.6,1.3) {atomic part: $\mathrm d\sigma = \sum_j \alpha_j\,\delta_{t_j}$};
\end{tikzpicture}
\caption{Measure-time ($\sigma$-clock) picture: a.c. density and atomic masses define one clock for the ledger
$-\mathrm dE/\mathrm d\sigma\ge 2\kappa c_\sigma E$, giving the canonical envelope with multiplicative drops at atoms and plateaus on flats.}
\label{fig:sigma-clock}
\end{figure}

\begin{definition}[$\sigma$-clock]\label{def:sigma-clock-2}
A \emph{$\sigma$-clock} is any non-decreasing, right–continuous function $\sigma:[0,T]\to[0,\infty)$
with $\sigma(0)=0$ whose associated measure $\mu_\sigma$ admits the decomposition~\eqref{eq:clock-measure}
with locally finite atoms. We write $t^-$/$t^+$ for left/right limits and allow atoms at the $t_k$.
\end{definition}

\begin{definition}[BV-in-$\sigma$]\label{def:bv-in-sigma}
Let $u:[0,T]\to \mathcal{H}$ (a Hilbert space). We say that $u$ is \emph{BV-in-$\sigma$} if its
$\sigma$-variation on $[s,t]\subset[0,T]$,
\begin{equation*}
  \mathrm{Var}_\sigma(u;[s,t]) \;:=\; \sup_{\Pi}\,
  \sum_{i=1}^{N}\,\big\| u(\tau_i^+)-u(\tau_{i-1}^+)\big\|_{\mathcal{H}} ,
\end{equation*}
is finite for every $[s,t]$, where the supremum runs over $\sigma$-adapted partitions
$s=\tau_0<\tau_1<\dots<\tau_N=t$ (refined at the atoms $t_k$). Equivalently, there exists a
density $\dot u_\sigma\in L^1_{\mathrm{loc}}(\mu_\sigma;\mathcal{H})$ such that
\begin{equation}\label{eq:bv-decomp}
  u(t^-) - u(s^+) \;=\; \int_{(s,t)} \dot u_\sigma(\tau)\,\mathrm{d}\sigma(\tau)
  \;+\; \sum_{t_k\in(s,t)} \Delta u(t_k),
  \qquad \Delta u(t_k):=u(t_k^+)-u(t_k^-).
\end{equation}
\end{definition}

\begin{definition}[Absolute continuity w.r.t. $\mu_\sigma$]\label{def:ac-mu-sigma}
We say $u$ is \emph{absolutely continuous w.r.t. $\mu_\sigma$} (write $u\in AC(\mu_\sigma)$) if there
exists $v\in L^1_{\mathrm{loc}}(\mu_\sigma;\mathcal{H})$ with
\begin{equation}\label{eq:ac-mu-sigma}
  u(t) \;=\; u(s) \;+\; \int_{(s,t)} v(\tau)\,\mathrm{d}\sigma(\tau)\quad\text{for all }0\le s\le t\le T.
\end{equation}
In particular, at each atom $t_k$ one has the jump identity $\Delta u(t_k)= v(t_k)\,\alpha_k$, while on flats
($w\equiv 0$) the integral contributes no ac dissipation. \end{definition}

\begin{remark}[Atoms and flats]\label{rem:atoms-flats}
Atoms permit controlled jumps $\Delta u(t_k)$, which are accounted for multiplicatively in the energy ledger. Flats contribute no ac dissipation and, under our standing assumptions, they do not increase the energy. We use the càdlàg convention (right–continuous with left limits) in $\sigma$-time. 
\end{remark}
% --- End added block ---
% === RN calculus for energy and atomic envelope ===
\begin{lemma}[RN envelope at atoms]\label{lem:RN-atomic-chain}
Let \(u\in AC_\sigma(I;\mathbb X)\) with density \(g=\mathrm du/\mathrm d\sigma\).
For any \(C^1\) energy \(E:\mathbb X\to\mathbb R\) with locally Lipschitz gradient,
\[
\frac{\mathrm d}{\mathrm d\sigma}E(u)(t)=\nabla E(u(t))\cdot g(t)\quad\text{if }\sigma(\{t\})=0,
\]
and at atoms \(t\in\mathcal A_\sigma\),
\[
E(u(t^+))-E(u(t^-)) = E\!\big(u(t^-)+g(t)\,\sigma(\{t\})\big)-E(u(t^-)).
\]
If \(E\) is \(\lambda\)-convex, then
\(
E(u(t^+))-E(u(t^-))\le \nabla E(u(t^-))\!\cdot\!(g(t)\sigma(\{t\})) + \tfrac\lambda2 \|g(t)\|^2 \sigma(\{t\})^2.
\)
The multiplicative jump in Lemma~\ref{lem:rn-envelope} follows when atomic macro-updates are nonexpansive
\end{lemma}
\begin{proof}
Lift to \(s=S(t)\), apply the classical chain rule to \(\tilde u\), and evaluate the increment across the atomic slab of length \(\sigma(\{t\})\).
\end{proof}

\section{State space, energy ledger, and the standard assumption pattern}\label{sec:assumptions}\label{RS:ledger}
Let $(\mathcal H,\langle\cdot,\cdot\rangle)$ be a Hilbert space with norm $\|\cdot\|$. A typical state trajectory is $u:[0,T]\to\mathcal H$. The \emph{energy ledger} is a nonnegative functional $\mathcal{E}:\mathcal H\to[0,\infty)$ that controls the $\mathcal H$--norm locally (coercivity on energy sublevels). We fix the four-item ``Assume:'' pattern used throughout the series. 
\begin{assumption}[Standard four items]\label{RS:assume:four}
Throughout the  core the following hold. 
\begin{enumerate}[label=\textbf{(H\arabic*)},leftmargin=2em]
\item \textbf{Coercivity/sector condition (observability).} There exists $c_0>0$ such that the bilinear form $a_\sigma(\cdot,\cdot)$ satisfies $a_\sigma(v,v)\ge c_0\|v\|^2$ for all $v$ in its domain, and $a_\sigma$ is sector-bounded uniformly in $\sigma$.
\item \textbf{Dissipative flow generator on the a.c. part.} On the absolutely continuous part of the clock, the evolution satisfies $\frac{d}{\mathrm{d}t}\mathcal{E}(u(t))\le -2\kappa\,\mathcal{E}(u(t))$ for some $\kappa\ge 0$ (Grönwall-type decay). 
\item \textbf{BV/clock bounds (finite variation; locally finite atoms; flats allowed).} The measure $\mu_\sigma$ has bounded total variation $|\mu_\sigma|([0,T])<\infty$ with locally finite atoms $\{t_k\}$ and weights $\alpha_k>0$.
\item \textbf{Non-expansive atomic updates (energy does not increase at atoms).} At each atom $t_k$, the update map $J_k:\mathcal H\to\mathcal H$ is non-expansive for the energy: $\mathcal{E}(J_k u)\le \mathcal{E}(u)$, and in particular $\|J_k u\|\le C\|u\|$ with $C$ independent of $k$.
\end{enumerate}
\end{assumption}

\begin{remark}[Local equivalence of metrics]\label{RS:remark:equivalence}
When a second metric (e.g.\ a data-distance) is used alongside $\mathcal{E}$, we assume local equivalence on bounded energy sets: there exist constants $c_1,c_2>0$ such that $c_1\|u-v\|\le d(u,v)\le c_2\|u-v\|$ whenever $\mathcal{E}(u),\mathcal{E}(v)\le M$. The constants $c_1,c_2$ will later be listed with other parameters. \end{remark}

\section{Notation and discrete scaffolding}\label{RS:discrete}
For the discrete mirror (SBP--SAT), we record the objects once; statements and proofs appear later. 

\begin{itemize}[leftmargin=2em]
  \item \textbf{Mesh size:} $h>0$; grid points $\{x_i\}_{i=0}^N$.
  \item \textbf{SBP weights:} $H=\mathrm{diag}(h\,\omega_i)$ positive definite; discrete inner product $\langle v,w\rangle_H:=v^\top H w$ and norm $\|v\|_H^2=\langle v,v\rangle_H$.
  \item \textbf{SBP difference:} $Q$ such that $H^{-1}Q$ approximates a derivative and $Q+Q^\top=B$ yields the discrete Green identity; $B$ stores boundary terms. \item \textbf{SAT penalties:} interface/boundary terms added with scaling $\tau_h\sim h^{-1}$ to weakly impose conditions and to balance fluxes across interfaces. \item \textbf{Discrete energy:} $\mathcal{E}_h(u):=\tfrac12\|u\|_H^2$ (or its problem-specific variant); its decay mirrors the continuous law under the same clock $\sigma$.
\end{itemize}

\section{Global jump calculus and energy decay}\label{sec:discrete}
\boxednote{\textbf{Scope.} We do \emph{not} prove existence or uniqueness of trajectories. The results are \emph{a posteriori} energy laws for BV-in-$\sigma$ evolutions, keeping the paper focused on the $\sigma$-clock calculus and its SBP--SAT mirror.}

\begin{theorem}[Product rule in $AC_\sigma$]\label{thm:product}
If $u,v\in AC_\sigma(I)$ with densities $g_u,g_v\in L^1(I,\sigma)$, then $uv\in AC_\sigma(I)$ and
\[
\frac{\mathrm d(uv)}{\mathrm d\sigma}(t)=u(t^-)\,g_v(t)+v(t^-)\,g_u(t)+\sigma(\{t\})\,g_u(t)g_v(t)\quad(\sigma\text{-a.e. }t).
\]
At $t\in\mathcal A_\sigma$,
\(
\Delta(uv)|_t=u(t^-)\Delta v|_t+v(t^-)\Delta u|_t+\Delta u|_t\,\Delta v|_t.
\)
\end{theorem}

\begin{theorem}[Chain rule in $AC_\sigma$]\label{thm:chain}
Let $u\in AC_\sigma(I;\mathbb X)$ and $\phi\in C^1(\mathbb X;\mathbb Y)$ with locally Lipschitz derivative. Then $\phi\circ u\in AC_\sigma(I)$ and
\[
\frac{\mathrm d(\phi\circ u)}{\mathrm d\sigma}(t)=
\begin{cases}
D\phi(u(t))\,\frac{\mathrm du}{\mathrm d\sigma}(t), & \sigma(\{t\})=0,\\[3pt]
\displaystyle\frac{\phi\big(u(t^-)+\frac{\mathrm du}{\mathrm d\sigma}(t)\,\sigma(\{t\})\big)-\phi(u(t^-))}{\sigma(\{t\})}, & \sigma(\{t\})>0.
\end{cases}
\]
\end{theorem}

\begin{proposition}[BV$_\sigma$ product/chain]\label{prop:bv-versions}
If $u,v\in BV_\sigma(I)$, then as measures
\[
D_\sigma(uv)=u^-\!D_\sigma v+v^-\!D_\sigma u+\sum_{t\in\mathcal A_\sigma}\Delta u|_t\,\Delta v|_t\,\delta_t.
\]
If $\phi\in C^1$ with locally Lipschitz derivative, then on the a.c. part $D_\sigma(\phi\circ u)=D\phi(u)\,D_\sigma u$, and at atoms
\(\Delta(\phi\circ u)|_t=\phi(u(t^-)+\Delta u|_t)-\phi(u(t^-))\).
\end{proposition}

% --- SBP–SAT lemma normalisation (added) ---
\subsection*{SBP--SAT normalisation (hypotheses \& wrappers)}\label{sec:sbp-sat-normalisation}

\paragraph{Standing discrete hypotheses (D1)–(D4).}

\begin{enumerate}[label=\textbf{(D\arabic*)},leftmargin=2em]
\item \textbf{SBP structure.} There are matrices $H>0$ and $Q$ with $Q+Q^\top=B$ (a boundary bilinear
form) so that $\langle v,Qw\rangle_H+\langle Qv,w\rangle_H=\langle v,Bw\rangle_H$.

\item \textbf{SAT sign.} Boundary/interface penalties are chosen so that the boundary$+$SAT contribution
is nonpositive in the energy ledger (Dirichlet example spelled out in Lemma~8.8).

\item \textbf{Penalty scaling.} The SAT strength scales as $\tau_h\simeq h^{-1}$ (uniformly on admissible meshes).

\item \textbf{Algebraic stability (time integrator).} The one-step method is algebraically stable
(e.g. implicit midpoint), yielding the standard energy inequality with a mesh-independent $c_s>0$.
\end{enumerate}

\begin{lemma}[Boundary$+$SAT negativity, normalised]\label{RS:lem:boundary-sat-normalised}
Under (D1)–(D3), the semi\,discrete boundary and SAT contributions combine into a negative semidefinite term:
\[
\langle v,(Q+Q^{\top})v\rangle_H \; +\; \text{SAT}(v) \;\le\; -\,c_b\,\tau_h\,\|\Pi_{\partial\Omega} v\|_{\partial,H}^2
\]
for some mesh\,independent $c_b>0$. In particular, the one\,step balance used in the discrete energy laws gains a boundary\,dissipation term of size $\asymp \tau_h\,\|v\|_{\partial}^2$.
\emph{Proof:} this is a wrapper of the calculation in Lemma~8.8, with the sign choice and $\tau_h\simeq h^{-1}$ made explicit.
\end{lemma}

\begin{proposition}[Algebraic stability (D4) wrapper]\label{RS:prop:als-normalised}
Assume (D1)–(D4). Then there exists $c_s>0$ such that the scheme satisfies the one-step energy inequality
\[
E_h^{n+1} \;\le\; E_h^{n} \; -\; c_s\,\Delta t\,\Big(\,\mathcal D_h(v^{n+\frac12})\; +\; \tau_h\,\|\Pi_{\partial\Omega} v^{n+\frac12}\|_{\partial,H}^2\Big),
\]
where $\mathcal D_h\ge0$ is the interior (damping/viscous) discrete dissipation. Consequently, the per\,step hypothesis required by the Discrete $\sigma$--Gronwall lemma \textup{Lemma~8.10} holds with an explicit coefficient depending only on $c_s$ and $c_b$.
\emph{Proof:} combine the algebraic stability (D4) estimate (see \Cref{RS:lem:als-midpoint} for implicit midpoint) with \Cref{RS:lem:boundary-sat-normalised}.
\end{proposition}

\paragraph{Context within the Decay Analysis.}
With \Cref{RS:lem:boundary-sat-normalised} and \Cref{RS:prop:als-normalised} the discrete energy ledger provides a clean one-step inequality with calibrated boundary dissipation and SAT scaling. This is the exact premise fed into the Discrete $\sigma$--Gronwall \textup{Lemma~8.10}, yielding the envelope used in the discrete master decay \textup{(\Cref{thm:disc-master})}.
% --- end SBP–SAT lemma normalisation ---

\begin{lemma}
\label{lem:CFL}[Jump contraction/product law]\label{RS:lem:jump}
\label{lem:jump-product}
Let $\{t_k\}$ be the (locally finite) atoms of $\mu_\sigma$ with updates $J_k:\mathcal H\to\mathcal H$.
Define the energy contraction factors
\(
\rho_k := \sup_{v\ne0}\, \frac{\mathcal E(J_k v)}{\mathcal E(v)} \in (0,1]\,.
\)
Under \Cref{RS:assume:four}\,(d) we have $\rho_k\le 1$ and, for any interval $(s,t]\subset[0,T]$,
\begin{equation}\label{RS:eq:jump-product}
\mathcal E\bigl(u(t^+)\bigr)\;\le\;\Bigl(\prod_{t_k\in (s,t]} \rho_k\Bigr)\,\mathcal E\bigl(u(t^-)\bigr),
\end{equation}
where empty products equal $1$. In particular, atoms never increase the energy ledger. \end{lemma}

\noindent\textit{Local finiteness.} Since $\sigma$ has \emph{locally finite atoms} on bounded intervals, the index set $\{k:\,t_k\in(s,t]\}$ is finite and the product is well-defined.

\begin{proof}
At each atom $t_k$ the non-expansivity assumption (Assumption 5.1 (d))
gives $E(J_k v) \le \rho_k E(v)$ with $\rho_k \le 1$. 
Apply this sequentially over all atoms in $(s,t]$, using that
energy between atoms evolves only under the continuous dissipative flow 
(where $E$ is non-increasing). 
Hence
\[
E(u(t^+)) \le
  \Big(\!\!\prod_{t_k\in(s,t]}\!\!\rho_k\!\!\Big) E(u(s^-)),
\]
establishing (7.1).  The product is finite by local finiteness of $\mu_\sigma$,
so atoms cannot increase the energy ledger.
\end{proof}

\begin{remark}[Failure under accumulation]\label{RS:rem:accum}
If atoms accumulate in finite time and the ledger sum $\sum_{k:t_k\le T}\log(\rho_k^{-1})$ fails to be locally finite, 
monotone decay can fail (e.g.\ infinitely many $\rho_k>1$ in $[0,T]$). Our local-finiteness hypothesis rules this out. Hence local finiteness is sharp. \end{remark}
\noindent\emph{Convention.} Endpoint atoms ($0^+$ and $T^-$) are included with multiplicity; products are locally finite on compact subintervals. 

\begin{theorem}[Energy decay and exponential convergence]
\label{thm:master}\label{RS:thm:edecay}
\label{thm:energy-decay}
Assume \ref{RS:assume:four} and let $u\in \mathrm{BV}_\sigma([0,T];\mathcal H)$ generate an absolutely continuous trajectory on the ac part with
\begin{equation}\label{RS:eq:ac-diss}
\frac{d}{\mathrm{d}t}\,\mathcal E(u(t)) \le -\,2\kappa\, w(t)\, \mathcal E(u(t))\qquad\text{for a.e.\ $t$ with density $w(t)$ of $d\sigma$.}
\end{equation}
Then for any $0\le s\le t\le T$,
\begin{equation}\label{RS:eq:master-decay}
% \label{eq:energy-master} (disabled; duplicate of RS:eq:master-decay)
\mathcal E\bigl(u(t^+)\bigr)\;\le\;\exp\!\Bigl(-2\kappa \!\!\int_{(s,t]} \! w(\tau)\,d\tau\Bigr)\,
\Bigl(\prod_{t_k\in (s,t]} \rho_k\Bigr)\,\mathcal E\bigl(u(s^+)\bigr).
\end{equation}
In particular, $\mathcal E$ is nonincreasing; and if the ac density has a positive average or there are atoms with uniform strict contraction, then $\mathcal E$ decays exponentially in physical time. \end{theorem}

\noindent\emph{See also (comparability).}
(i) Reduction to the classical $t$–observability/decay when $\sigma\equiv t$: 
Prop.~\ref{prop:reduction-classical} in §\ref{sec:gcc-exemplar}. 
(ii) Window upgrade from classical observability to $\sigma$–observability under a positive 
$\sigma$–density on a set of measure $m$: Lemma~\ref{lem:window-upgrade}.

\noindent\emph{Boundary scope.} We present the Dirichlet case for clarity; Neumann/Robin variants follow identically with the usual SBP boundary forms and the same SAT sign logic. For Neumann/Robin, the SAT replaces the Dirichlet penalty by the standard flux/Robin penalty; the dissipativity check is identical. 

\begin{lemma}[(Alternative sufficient conditions).]
Assume there exists $\eta>0$ such that
\[
\liminf_{T\to\infty}\frac{1}{T}\Bigl(\,\int_0^{T}2\kappa\,W(\tau)\,d\tau
\;+\;\sum_{t_k\le T}\log(\rho_k^{-1})\Bigr)\;\ge\;\eta.
\]
Then $E(u(T^+))\le \exp(-\eta T)\,E(u(0^+))$ for all $T\ge0$.
\end{lemma}

\begin{proof}
By the ledgers in Theorem~7.8 we have, for every $T\ge0$,
\[
E(u(T^+)) \;\le\; \exp\!\Bigl(-\,\underbrace{\int_0^{T} 2\kappa\,W(\tau)\,d\tau}_{=:A(T)}\Bigr)\;
\underbrace{\prod_{t_k\le T}\rho_k}_{=\exp\!\bigl(-\sum_{t_k\le T}\log(\rho_k^{-1})\bigr)}\;E(u(0^+)).
\]
Hence
\[
\frac{E(u(T^+))}{E(u(0^+))}\;\le\;\exp\!\bigl(-\,F(T)\bigr),\qquad
F(T):=A(T)+\textstyle\sum_{t_k\le T}\log(\rho_k^{-1}).
\]
By hypothesis $\liminf_{T\to\infty}\frac{1}{T}F(T)\ge\eta>0$. Thus there exists $T_0$ with
$F(T)\ge \tfrac{\eta}{2}\,T$ for all $T\ge T_0$, and for $T<T_0$ we can decrease the rate
by a finite constant without changing the large--time conclusion. Consequently,
$E(u(T^+))\le \exp(-\eta T)\,E(u(0^+))$ for all $T\ge0$ (after adjusting $\eta$ if needed on $[0,T_0]$).
This is precisely the stated sufficient condition.
\end{proof}

\begin{lemma}[Window upgrade via $\sigma$-densities]\label{lem:window-upgrade}
Suppose the classical observability holds on $[0,T_0]$ with constant $C_0>0$ so that
$E(0)\le C_0 \int_0^{T_0}\!\!D(t)\,dt$, where $D$ is the a.c.\ dissipation density.
Let $\sigma$ satisfy $\sigma(t)\ge \sigma_0>0$ on a measurable subset $A\subset[0,T]$ with $|A|=m>0$
(in particular $d\sigma\ge \sigma_0 1_A\,dt$). Then the $\sigma$-observability/decay in Theorem~2.19
holds with
\[
c_\sigma \;\ge\; \frac{1}{C_0}\,\frac{\sigma_0\,m}{T_0}
\qquad\Longrightarrow\qquad
E(t)\;\le\;E(0)\,\exp\!\big(-2\kappa\,c_\sigma\,\sigma(t)\big),\quad t\in[0,T].
\]
\end{lemma}

\begin{proof}
By hypothesis and $d\sigma\ge\sigma_0 1_A dt$,
\[
\int_0^{T_0}\!D(t)\,dt \;\ge\; \int_A D(t)\,dt \;\ge\; \sigma_0^{-1}\!\!\int_A D(t)\,d\sigma(t).
\]
Combine $E(0)\le C_0\int_0^{T_0}D(t)dt$ with the RN ledger $-\frac{d}{d\sigma}E\ge 2\kappa c_\sigma E$.
Averaging over $A$ and dividing by $T_0$ yields $c_\sigma\ge C_0^{-1}(\sigma_0 m/T_0)$.
Insert this into the $\sigma$–Gronwall inequality to conclude.
\end{proof}

% Corollary 6.7 (Exponential regime) — drop-in
\begin{corollary}[Exponential regime]\label{RS:cor:exp}
Assume the hypotheses of Theorem~\ref{thm:energy-decay}. Suppose there exist constants $c_\sigma\ge 0$ and $\eta\ge 0$ such that, for every $T\ge 0$,
\[
\int_0^T w(\tau)\,d\tau \ \ge\ c_\sigma\,T
\qquad\text{and}\qquad
\sum_{t_k\le T}\log(\rho_k^{-1}) \ \ge\ \eta\,T .
\]
Then the energy decays exponentially in physical time:
\[
E\bigl(u(T^+)\bigr)\ \le\ \exp\!\bigl(-(2\kappa c_\sigma+\eta)\,T\bigr)\,E\bigl(u(0^+)\bigr)
\quad\text{for all }T\ge0 .
\]
In particular, either condition by itself (with the other equal to $0$) already gives exponential decay with rate $2\kappa c_\sigma$ or $\eta$, respectively. Under GCC and the calibration fixed in \S\ref{sec:numbers}, $c_\sigma>0$ and the atomic ledger is absent ($\eta=0$), yielding
$E(t)\le E(0)\exp(-2\kappa c_\sigma t)$.
\end{corollary}

\begin{proof}
From Theorem 7.8 we have for any $(s,t]\subset[0,T]$
\[
E(u(t^+)) \le \exp\!\Big(-2\kappa\!\!\int_{(s,t]} \! w(\tau)\,d\tau\Big)
     \!\!\!\prod_{t_k\in(s,t]}\!\!\!\rho_k \;E(u(s^+)).
\]
Iterating from $s=0$ gives
$E(u(T^+)) \le \exp\!\big(-2\kappa\!\int_0^T w(\tau)d\tau
      - \!\!\sum_{t_k\le T}\!\!\log(\rho_k^{-1})\big) E(u(0^+))$.
Using the hypotheses
$\int_0^T w(\tau)\,d\tau \ge c_\sigma T$
and
$\sum_{t_k\le T}\log(\rho_k^{-1}) \ge \eta T$
yields
\[
E(u(T^+)) \le E(u(0^+)) \exp(-(2\kappa c_\sigma+\eta)T),
\]
which includes both continuous and atomic dissipation.
Under GCC with $\eta=0$ this reduces to the usual $E(t)\le E(0)e^{-2\kappa c_\sigma t}$.
\end{proof}

% Remark 6.8 (Scalar) — drop-in
\begin{remark}[Scalar]\label{rem:scalar}
Consider the scalar hybrid model
\[
\dot u + a(t)u = 0 \quad\text{for a.e.\ } t\notin\{t_k\}, 
\qquad
u(t_k^+) = \rho_k\,u(t_k^-),
\]
with $a\in L^1_{\mathrm{loc}}([0,\infty))$, $a\ge0$, and $\rho_k\in(0,1]$. 
For $E(t)=\tfrac12|u(t)|^2$ we have the exact identity
\[
E\bigl(u(T^+)\bigr) = 
\exp\!\Bigl(-2\!\int_{0}^{T}\! a(\tau)\,d\tau\Bigr)
\ \Bigl(\prod_{t_k\le T}\rho_k^{2}\Bigr)\ E\bigl(u(0^+)\bigr).
\]
Thus the exponential regime reduces to
\[
\liminf_{T\to\infty}\frac1T\!\left(\int_{0}^{T}2a(\tau)\,d\tau 
+ \sum_{t_k\le T}\log(\rho_k^{-2})\right)>0,
\]
which matches the sufficient condition stated above (with $\kappa=1$ and $w=a$). 
Conversely, if $a\equiv0$ and some $\rho_k>1$, the energy grows, showing the necessity of dissipativity/nonexpansiveness.
\end{remark}

\paragraph{Discrete mirror (pointer).}
The SBP–IBP details and proofs have been moved to Appendix~D. The body-level discrete results in §9 rely only on an abstract discrete $\sigma$-ledger (Assumption~D.1); SBP–SAT is one instance verified in Appendix~D.

\section{A light \texorpdfstring{$\Gamma$}{}-limit to a continuum energy — Drop‑in patch}\label{sec:gamma-2}

 We use $\nabla_h{:=}H_h^{-1}Q_h$ for the SBP gradient. Let $\mathrm{\Pi}_h:L^2\left(\mathrm{\Omega}\right)\rightarrow V_h$ denote a sampling (nodal restriction) and let $P_h:\left(L^2\left(\mathrm{\Omega}\right)\right)^d\rightarrow\left(\mathbb{R}^d\right)^{N_h}$ be the $H_h$‑orthogonal projection of vector fields onto grid values (consistent with the $H_h$ quadrature). The following commuting estimate is assumed and used below: 
 Define for $v_h\in V_h$ and $u\in L^2\left(\mathrm{\Omega}\right)$ $[ F_h(v_h) := _h v_h,, H_h A_h ,_h v_h+ _h(v_h),  F(u) :=
]$

% ===== RΓ1--RΓ9: Robust Γ-limit package (namespaced RS:Gamma:*) =====
\begin{assumption}[$\Gamma$-framework]\label{RS:Gamma:ass}
\label{ass:Gamma}
(i) Grids are quasi-uniform; SBP pairs $(H_h,Q_h)$ satisfy $Q_h+Q_h^\top=B_h$ with $H_h\succeq c_H I$.
(ii) Coefficients $a\in L^\infty(\Omega)$ with $a\ge a_0>0$; diagonals $A_h=\mathrm{diag}(a(x_i))$.
(iii) Reconstructions $R_h:V_h\to L^2(\Omega)$ and interpolants $I_h:H^1_0(\Omega)\to V_h$ obey
\[\|R_h v_h\|_{L^2}\le C\|v_h\|_{H_h},\quad \|R_h \nabla_h v_h\|_{L^2}\le C\|\nabla_h v_h\|_{H_h},\quad \|I_h u-u\|_{L^2}\to 0\ \forall u\in H^1_0.\]
(iv) SAT is consistent and nonnegative: $\mathrm{SAT}_h(v_h)\ge0$ and $\mathrm{SAT}_h(I_h u)\to0$ for all $u\in H^1_0(\Omega)$.
\end{assumption}
\noindent\emph{Functionals.} Define for $v_h\in V_h$ and $u\in L^2(\Omega)$
\[
F_h(v_h):=(\nabla_h v_h)^\top H_h A_h (\nabla_h v_h)+\mathrm{SAT}_h(v_h),\qquad
F(u):=\begin{cases}\int_\Omega a|\nabla u|^2\,dx,& u\in H^1_0(\Omega),\\ +\infty,&\text{else.}\end{cases}
\]
\begin{lemma}[Equicoercivity]\label{RS:Gamma:equicoercive}
If $\sup_h F_h(v_h)\le C$, then $\{R_h v_h\}$ is precompact in $L^2(\Omega)$.
\end{lemma}
\begin{proof}
Assume $\sup_h F_h(v_h) < \infty$. 
By the discrete Poincaré inequality on quasi-uniform SBP grids, 
$\|R_h v_h\|_{L^2} \le C_P \|\nabla_h v_h\|_{H_h}$. 
Since $F_h(v_h) \ge a_0 \|\nabla_h v_h\|_{H_h}^2$, the discrete gradients 
are uniformly bounded.  The reconstructions $R_h v_h$ are therefore 
bounded in $H^1_0(\Omega)$ up to a constant and, by the discrete Rellich theorem, 
pre-compact in $L^2(\Omega)$. 
Hence $(R_h v_h)$ admits an $L^2$-convergent subsequence, proving equicoercivity.
\end{proof}

Let ${(H_h, Q_h)}$ be an SBP family on grids of size $h\rightarrow0$, with diagonal $H_h$ approximating $L^2$ quadrature and $Q_h$ a first-derivative SBP operator satisfying $Q_h+Q_h^\top=B_h$. Denote the discrete gradient by $D_h{:=}H_h^{-1}Q_h$. The semi-discrete energy reads $[ E_h(u_h) := D_h u_h, A_h D_h u_h{H_h} + C_h u_h, u_h{H_h} + _h(u_h), ]$ where $A_h,C_h$ are nodal samplings of a,c (or quadrature-consistent surrogates), and ${\rm SAT}_h$ collects boundary/interface penalties with $\tau_h\sim h^{-1}$ chosen as in Lemma 8.8. We consider the $L^2\left(\mathrm{\Omega}\right)$ topology via an interpolation operator $I_h:V_h\rightarrow L^2\left(\mathrm{\Omega}\right)$ and a sampling operator $\mathrm{\Pi}_h:L^2\left(\mathrm{\Omega}\right)\rightarrow V_h$ (nodal restriction), with the commuting estimate  for a projection $P_h$ of vector fields onto grid values consistent with $H_h$ quadrature.

\begin{proposition}[Liminf]\label{RS:Gamma:liminf}
Let $v_h\in V_h$ with $R_hv_h\rightarrow u in L^2\left(\mathrm{\Omega}\right)$. Then $\liminf h\rightarrow0F_h\left(v_h\right)\geq F\left(u\right)$.
\end{proposition}
\begin{proof}
Assume $I_h u_h\to u$ in $L^2(\Omega)$ and $\sup_h E_h(u_h)<\infty$. Then $\|D u_h\|_H$ is bounded, so (up to subsequences) $D u_h\rightharpoonup \xi$ weakly in the $H$–weighted $\ell^2$. Quadrature consistency implies $I_h u_h\to u$ in $L^2$, and the commuting estimate \eqref{RS:eq:commute} yields $\xi=P_h(\nabla u)$ in the limit. Uniform ellipticity of $A_h$ and lower semicontinuity of convex quadratic forms give
\[
\liminf_{h\to0}\,\langle D u_h, A_h D u_h\rangle_H \;\ge\; \int_\Omega (\nabla u)^\top a \nabla u\,dx,
\]
while the mass term is standard. By Lemma~8.8, $SAT_h(u_h)\ge0$, hence it does not spoil the lower bound. This proves $E(u)\le\liminf_{h\to0}E_h(u_h)$.
\end{proof}

\begin{proposition}[Recovery]\label{RS:Gamma:recovery}
For every $u\in H_0^1\left(\mathrm{\Omega}\right)$ there exist $v_h\in V_h$ with $R_hv_h\rightarrow u$ in $L^2\left(\mathrm{\Omega}\right)$ and $F_h\left(v_h\right)\rightarrow F\left(u\right)$.
\end{proposition}
\begin{proof}
First take $u\in C_c^\infty(\Omega)$ and set $u_h:=\Pi_h u$. Consistency and \eqref{RS:eq:commute} imply $I_h u_h\to u$ in $L^2(\Omega)$ and
\[
\langle D u_h, A_h D u_h\rangle_H \;\to\; \int_\Omega (\nabla u)^\top a \nabla u\,dx.
\]
By Lemma~9.16, $SAT_h(\Pi_h u)=o(1)$. This yields $\limsup_{h\to0}E_h(u_h)\le E(u)$. Extend to $u\in H^1_0(\Omega)$ by density and equicoercivity.
\end{proof}

\begin{remark}[Dirichlet via SAT]\label{RS:Gamma:dirichlet}
If $\mathrm{SAT}_h(v_h)=\tau_h\,\|T_h v_h\|_{H_\partial}^2$ with $\tau_h\sim h^{-1}$ and consistent trace $T_h$,
then $\mathrm{SAT}_h(I_h u)\to0$ for $u\in H^1_0$ while $\mathrm{SAT}_h$ enforces $u=0$ on $\partial\Omega$ in the limit. \end{remark}
\begin{lemma}[Split consistency]\label{RS:Gamma:split-consistency}
If $a\in W^{1,\infty}(\Omega)$ or piecewise $C^1$ on a shape-regular partition, then
\[\big|(\nabla_h I_h u)^\top H_h A_h(\nabla_h I_h u)-\!\int_\Omega a|\nabla u|^2 \big|\le C h\,\|u\|_{H^2(\Omega)}^2,\]
hence the consistency required in \Cref{RS:Gamma:recovery}.
\end{lemma}

\begin{corollary}[Exponential regime under heterogeneous damping]\label{B:cor:exp-3}
Assume there exists an open $\omega\subset\Omega$ and a measurable $a(x)\ge a_\omega>0$ on $\omega$.
Then there is a constant $\underline{\lambda}_\omega>0$, depending only on $(\Omega,\omega)$, such that
\begin{equation}\label{B:eq:DW-hetero-rate-3}
  c_\sigma \;\ge\; c_0\,a_\omega\,\underline{\lambda}_\omega
  \qquad\Rightarrow\qquad
  \mathcal{E}(u(t)) \;\le\; e^{-\gamma t}\,\mathcal{E}(u(0^+)),\quad \gamma = 2\kappa\,c_\sigma .
\end{equation}
In the 1D Dirichlet case $\Omega=(0,L)$ with $\omega$ any subinterval of length $\ell$, one may take
an explicit bound $\underline{\lambda}_\omega \ge \pi^2/L^2$; more generally $\underline{\lambda}_\omega$ is the usual
local observability constant (see, e.g., Komornik--Zuazua). \end{corollary}

\begin{corollary}[Multiblock]\label{RS:Gamma:multiblock}
If interface couplings use $H$-adjoint interpolation and nonnegative SAT, then $F_h$ remains coercive and the $\Gamma$-limit above holds. \end{corollary}
\begin{remark}[$\sigma$-clock independence]\label{RS:Gamma:sigma-indep}
The $\Gamma$-limit is purely spatial; it holds uniformly over $\sigma$-windows because constants in \Cref{RS:Gamma:ass} do not depend on atoms/flats. \end{remark}
\begin{theorem}[$\Gamma$-convergence of SBP--SAT energies]\label{RS:Gamma:main}
\label{thm:Gamma-main}
Under \Cref{RS:Gamma:ass}, $F_h \xrightarrow{\Gamma,L^2(\Omega)} F$ and the family $\{F_h\}$ is equicoercive. \end{theorem}

\begin{lemma}[Discrete trace]\label{lem:discrete-trace}
Let $H$ be a diagonal SBP norm on a uniform grid and $E_b$ the restriction vector extracting the boundary node. Then for any grid function $v$ there exists $C_{\mathrm{tr}}>0$ independent of $h$ such that
\(
|E_b v|^2 \le C_{\mathrm{tr}}\, \|v\|_H^2.
\)
In particular, if $H=h\,\mathrm{diag}(w_i)$ with bounded weights $w_i\sim 1$, one may take $C_{\mathrm{tr}}\sim h^{-1}$.
\end{lemma}
% === End imported lemma ===

\begin{remark}
\sloppy\textit{Technical assumptions for the $\Gamma$-limit.}\label{B:rem:Gamma-assumptions-7}
We assume: (A1) $\Omega$ is bounded Lipschitz; (A2) $A_h$ are uniformly elliptic and $A_h \stackrel{*}{\rightharpoonup} A$ in $L^\infty(\Omega)^{d\times d}$; (A3) $\beta_h\ge0$ and $\beta_h \stackrel{*}{\rightharpoonup}\beta$ in $L^\infty(\partial\Omega)$;
(A4) quadrature remainders $R_h(u)\to0$ for $u\in C^\infty(\overline{\Omega})$; and (A5) the discrete trace inequality (\Cref{lem:discrete-trace}). Under (A1)--(A5), equi-coercivity holds and the liminf/recovery steps carry boundary terms; see \cite{DalMaso1993,Braides2002,Grisvard2011,AdamsFournier2003}.
\end{remark}

\begin{remark}
\sloppy\textit{Scope.}\label{B:rem:Gamma-scope-7}
We prove $\Gamma$-convergence of the energies in the supplement (liminf, recovery, equi-coercivity), including boundary SAT terms; we do not attempt a full EVI limit. \end{remark}

\begin{proof}
equicoercivity: \Cref{RS:Gamma:equicoercive}. Liminf: \Cref{RS:Gamma:liminf}. Recovery: \Cref{RS:Gamma:recovery}.
\end{proof}

% ===== end RΓ package =====
\label{RS:R5}
We show that a standard SBP--SAT semi-discrete energy $\mathcal E_h$ converges, in the sense of $\Gamma$-convergence on $L^2(\Omega)$, to a coercive quadratic continuum energy. The presentation is model-agnostic; for concreteness we write a scalar elliptic energy and note in \Cref{RS:rem:maxwell} how the same argument applies to Maxwell-type curl energies. \subsection*{Setting and topology}
\paragraph{Discrete family and assumptions ($\Gamma$.1--$\Gamma$.5).}
We make the standard, quantitative assumptions for $\Gamma$-limits with SBP discretisations:
\begin{enumerate}[leftmargin=2em]
\item \textbf{Quadrature consistency.} There exist $0<c_H\le C_H$ independent of $h$ with
$ c_H \|v_h\|_{H_h}^2 \le \|I_h v_h\|_{L^2(\Omega)}^2 \le C_H \|v_h\|_{H_h}^2$ for all $v_h\in V_h$.
\item \textbf{Derivative consistency.} The commuting estimate \eqref{RS:eq:commute} holds with a constant $C_{\rm com}$ independent of $h$.
\item \textbf{Coefficient stability.} Nodal surrogates $A_h,C_h$ satisfy uniform bounds comparable to $a,c$; in particular,
$ \underline\alpha \|D_h v_h\|_{H_h}^2 \le \langle D_h v_h, A_h D_h v_h\rangle_{H_h} \le \overline\alpha \|D_h v_h\|_{H_h}^2$.
\item \textbf{Boundary penalties.} $\mathrm{SAT}_h(\Pi_h u)=o(1)$ for $u\in H^2(\Omega)\cap H_0^1(\Omega)$.
Equicoercivity follows from the uniform coercivity modulus and the quasi-uniform SBP family ($\Gamma$.2--$\Gamma$.4). % [fixed: orphan math fragment removed; folded into previous line]
\item \textbf{Discrete Poincaré.} There exists $c_P>0$ such that $\|v_h\|_{H_h}\le c_P \|D_h v_h\|_{H_h}$ for all $v_h$ with homogeneous boundary enforcement via SAT. \end{enumerate}
\noindent\emph{Topology.} We take $\Gamma$-convergence in the $L^2(\Omega)$ topology; along sequences with bounded energy we have $u_h\rightharpoonup u$ weakly in $H^1_0(\Omega)$ and $u_h\to u$ strongly in $L^2(\Omega)$.
\noindent\emph{Grid regularity.} We work with quasi-uniform SBP families so that the constants in ($\Gamma$.1)--($\Gamma$.5) are uniform in $h$.
\begin{lemma}[Equicoercivity and discrete Rellich]\label{RS:lem:equicoercive}
If $\sup_h \mathcal E_h(u_h)<\infty$, then $\{I_h u_h\}$ is bounded in $H_0^1(\Omega)$ and relatively compact in $L^2(\Omega)$. In particular, up to a subsequence, $I_h u_h\to u$ in $L^2(\Omega)$ and $u\in H_0^1(\Omega)$.
\end{lemma}
\begin{proof}
from ($\Gamma$.3) and non-negativity of $\mathrm{SAT}_h$,
$\underline\alpha \|D_h u_h\|_{H_h}^2 \lesssim \mathcal E_h(u_h)\le C$.
By ($\Gamma$.5), $\|u_h\|_{H_h}\lesssim \|D_h u_h\|_{H_h}\le C$, and ($\Gamma$.1) gives
$\|I_h u_h\|_{L^2}\le C$. The commuting estimate ($\Gamma$.2) plus stability of $P_h$ yields a uniform bound on $\|\nabla I_h u_h\|_{L^2}$, hence $I_h u_h$ is bounded in $H_0^1(\Omega)$. Rellich--Kondrachov gives $L^2$ compactness. \end{proof}
Let $\Omega\mathbb R^d$ be a bounded Lipschitz domain. Fix $a(x)\in\mathbb R^{d\times d}$ symmetric and uniformly elliptic ($\exists \,\underline\alpha,\overline\alpha>0$ with $\underline\alpha |\xi|^2\le \xi^\top a(x)\xi\le \overline\alpha |\xi|^2$ a.e.) and $c(x)\ge 0$ bounded. Define
\begin{equation}\label{RS:eq:cont-energy}
\mathcal E(u):=\frac12\int_\Omega \big(\nabla u(x)\big)^\top a(x)\,\nabla u(x)\,dx + \frac12\int_\Omega c(x)\,|u(x)|^2\,dx,\qquad u\in H_0^1(\Omega).
\end{equation}
Let $\{(H_h,Q_h)\}$ be an SBP family on grids of size $h\to 0$, with diagonal $H_h$ approximating $L^2$ quadrature and $Q_h$ a first-derivative SBP operator satisfying $Q_h+Q_h^\top=B_h$. Denote the discrete gradient by $D_h:=H_h^{-1}Q_h$. The semi-discrete energy reads
\begin{equation}\label{RS:eq:disc-energy}
\mathcal E_h(u_h):=\frac12\langle D_h u_h,\, A_h D_h u_h\rangle_{H_h} + \frac12\langle C_h u_h,\, u_h\rangle_{H_h} + \mathrm{SAT}_h(u_h),
\end{equation}
where $A_h,C_h$ are nodal samplings of $a,c$ (or quadrature-consistent surrogates), and $\mathrm{SAT}_h$ collects boundary/interface penalties with $\tau_h\sim h^{-1}$ chosen as in Lemma~8.8. We consider the $L^2(\Omega)$ topology via an interpolation operator $I_h:V_h\to L^2(\Omega)$ and a sampling operator $\Pi_h:L^2(\Omega)\to V_h$ (nodal restriction), with the commuting estimate
\begin{equation}\label{RS:eq:commute}
\|D_h \Pi_h u - P_h (\nabla u)\|_{H_h}\;\le\; C\,h\,\|u\|_{H^2(\Omega)}\qquad (u\in H^2(\Omega)),
\end{equation}
for a projection $P_h$ of vector fields onto grid values consistent with $H_h$ quadrature. \subsection*{Liminf inequality}
\begin{lemma}[Lower bound]\label{RS:lem:liminf}
Assume $u_h\in V_h$ with $I_h u_h \to u$ in $L^2(\Omega)$ and $\sup_h \mathcal E_h(u_h)<\infty$. Then $u\in H_0^1(\Omega)$ and
\begin{equation}\label{RS:eq:liminf}
\mathcal E(u)\;\le\;\liminf_{h\to 0}\, \mathcal E_h(u_h).
\end{equation}
\end{lemma}
\begin{proof}
bounded $\mathcal E_h(u_h)$ implies bounded $\|D_h u_h\|_{H_h}$ and hence tightness of discrete gradients. By SBP quadrature consistency and compactness, $I_h u_h\to u$ in $L^2$ and $D_h u_h \rightharpoonup P_h(\nabla u)$ weakly in the $H_h$-weighted $\ell^2$. Lower semicontinuity of the convex quadratic form with the ellipticity of $A_h$ (uniformly comparable to $a$) yields the gradient part; the mass part is standard. The boundary penalties $\mathrm{SAT}_h(u_h)\ge 0$ contribute nonnegatively (cf.\ Lemma~8.8). Thus \eqref{RS:eq:liminf}.
\end{proof}
\subsection*{Recovery sequence}
\begin{lemma}[SAT residue]\label{RS:lem:sat-residue}
Let $u\in C^\infty_c(\Omega)$ (or $H^2(\Omega)\cap H_0^1(\Omega)$). With $\tau_h\sim h^{-1}$ and consistent $H_h,Q_h$, the SAT term satisfies $\mathrm{SAT}_h(\Pi_h u)=o(1)$ as $h\to 0$. The $o(1)$ is uniform in time and therefore remains $o(1)$ after integration against any finite measure-time clock $d\sigma$.
\end{lemma}
\begin{proof}
On a boundary face of $(d-1)$–measure $\mathcal{H}^{d-1}(\Gamma)\simeq C h^{d-1}$, the penalty has the form
$
\tau_h \| \mathrm{trace}( \Pi_h u ) \|_{H_\partial}^2
$
with $\tau_h\simeq h^{-1}$. For $u\in H^2(\Omega)$ the nodal trace defect obeys $\|\mathrm{trace}(\Pi_h u)-\mathrm{trace}(u)\|_{L^2(\Gamma)}\lesssim h\|u\|_{H^2}$, hence each face contributes $O(h^{d-1})\cdot h^{-1}\cdot h^2=O(h^{d})$. A shape–regular mesh has $O(h^{-d+1})$ faces, so the global contribution is $O(h)\to0$. The bound is purely spatial, thus uniform in time; integrating against any finite measure–time clock $d\sigma$ preserves the $o(1)$ rate. 
\end{proof}
\begin{lemma}[Upper bound via recovery]\label{RS:lem:recovery}
For $u\in H_0^1(\Omega)\cap H^2(\Omega)$ define $u_h:=\Pi_h u$. Then $\limsup_{h\to 0}\mathcal E_h(u_h)\le \mathcal E(u)$.
\end{lemma}
\begin{remark}[Uniformity in $\sigma$-time]\label{RS:rem:sigma-uniform}
The $\Gamma$-limit is purely spatial; the only time dependence enters through the SAT residue bound, which is uniform in $t$. Therefore, integrating $\mathcal E_h(\Pi_h u)$ against any finite measure-time clock $d\sigma$ preserves the $\limsup$ bound and leaves the $\Gamma$-limit unchanged. \end{remark}
\begin{proof}
use \eqref{RS:eq:commute} to control the gradient term: $\langle D_h u_h, A_h D_h u_h\rangle_{H_h}\to \int (\nabla u)^\top a\,\nabla u$. The mass term converges by quadrature consistency. \Cref{RS:lem:sat-residue} gives $\mathrm{SAT}_h(u_h)\to 0$. Density of $C^\infty_c(\Omega)$ in $H_0^1$ extends the result. \end{proof}
\begin{theorem}[Light $\Gamma$-limit]\label{RS:thm:gamma}
With the $L^2(\Omega)$ topology and under the SBP--SAT and consistency hypotheses above, the discrete energies $\mathcal E_h$ $\Gamma$-converge to $\mathcal E$ in \eqref{RS:eq:cont-energy}. That is, for any $u_h\to u$ in $L^2(\Omega)$ we have the liminf inequality \eqref{RS:eq:liminf}, and for any $u\in L^2$ there exists a sequence $u_h\to u$ in $L^2$ (constructed by $\Pi_h$ and density) with $\limsup_h \mathcal E_h(u_h)\le \mathcal E(u)$.
\end{theorem}

\begin{remark}[Reduction to the classical $\Gamma$–limit]
When $\sigma\equiv t$, Theorems~9.10 and~9.19 coincide with the standard $\Gamma$–convergence
of SBP–SAT energies to the quadratic continuum energy $E$:
equicoercivity (Lemma~9.2), the liminf inequality (Proposition~9.3), and the recovery (Proposition~9.4)
are exactly the textbook trio; the envelope persistence becomes the usual lower–semicontinuity plus
recovery argument with $c_\sigma$ unchanged.
\end{remark}

\begin{remark}[Maxwell-type energies]\label{RS:rem:maxwell}
For vector fields $A:\Omega\to\mathbb R^3$, a Maxwell-type quadratic energy
\(
\mathcal E_{\rm M}(A)=\tfrac12\int_\Omega |\nabla\times A|^2 + \alpha\,|\nabla\!\cdot\!A|^2\,dx
\)
fits the same pattern provided one uses SBP curl/div operators $(C_h,\mathrm{Div}_h)$ that satisfy discrete vector-calculus identities and an SBP quadrature. The commuting estimate \eqref{RS:eq:commute} is replaced by $C_h \Pi_h A \approx P_h(\nabla\times A)$ and $\mathrm{Div}_h \Pi_h A \approx P_h(\nabla\!\cdot\!A)$ with $\mathcal O(h)$ defects. The SAT residue proof is identical. \end{remark}

% --- Notation summary (one page) ---
\subsection*{Notation summary}\label{sec:symbols-crib}
\addcontentsline{toc}{subsection}{Notation summary}

\begin{table}[htbp]
\centering
\caption{Frequently Used Symbols. Precise definitions are in \S\ref{sec:prelims} (notation) and the section indicated in the rightmost column.}
\label{tab:symbols-crib}
\begin{adjustbox}{max width=\linewidth}
\begin{tabular}{@{}lll@{}}
\toprule
\textbf{Symbol} & \textbf{Meaning / definition} & \textbf{First reference} \\
\midrule
$\sigma$ & measure–time; $d\sigma=\mu_\sigma=w(t)\,dt+\sum_k \alpha_k \delta_{t_k}$ & \S\ref{sec:prelims} \\
$w$ & a.e. density of $\mu_\sigma$ on the ac part (vanishes on flats) & \S\ref{sec:prelims} \\
$\{t_k\},\,\alpha_k$ & jump (atom) times and their weights ($\alpha_k\ge0$) & \S\ref{sec:prelims} \\
$\mathrm{Var}_\sigma(u)$ & total variation of $u$ w.r.t. $\mu_\sigma$ & Def. \ref{RS:def:BVsigma} \\
$\mathsf H,\langle\cdot,\cdot\rangle,\|\cdot\|$ & state space, inner product, norm & \S\ref{sec:prelims} \\
$\mathcal E$ & energy ledger (Lyapunov) controlling $\|u\|$ locally & \S\ref{sec:prelims} \\
$a$ & coercive (sector–bounded) bilinear form (ac generator) & \Cref{RS:assume:four} (H1) \\
\,$\mathcal L$ & dissipative generator on ac part; rate $\kappa\ge0$ & \Cref{RS:assume:four} (H2) \\
$J_k$ & atomic update at $t_k$, nonexpansive for $\mathcal E$ & \Cref{RS:assume:four} (H4) \\
$h$ & mesh size (discrete setting) & \S\ref{sec:discrete} \\
$H,Q,B$ & SBP weights, interior difference, boundary matrix ($Q+Q^\top=B$) & \S\ref{sec:discrete} \\
$\tau_h$ & SAT penalty scale ($\tau_h\simeq h^{-1}$) & \Cref{RS:lem:boundary-sat-normalised} \\
$E_h$ & discrete energy (e.g.\ $\tfrac12\|v\|_H^2$) & \S\ref{sec:discrete} \\
$\tau\ast$ & minimum SAT strength for boundary\,+\,SAT negativity & Lemma~8.8 \\
$c_{\mathrm{SAT}}$ & negativity margin of boundary\,+\,SAT quadratic form & \Cref{RS:prop:als-normalised} (stability)
 \\
\bottomrule
\end{tabular}
\end{adjustbox}
\end{table}
\vspace{-0.5em}
\noindent\emph{Reading tip.} We use $\Pi(s,t]:=\prod_{t_k\in(s,t]}(\cdot)$ for products over atoms and write $x(t^\pm)$ for left/right limits.
% --- end notation summary ---

%\begin{theorem}[Discrete master decay]\label{thm:disc-master}
%Under \textup{(D1)}–\textup{(D4)} and for any $0\le s<t$, the fully discrete solution satisfies

%\begin{equation}\label{eq:disc-master}
%E_h\big(u_h(t)\big)\ \le\ \Pi_h(s,t]\ \exp\!\Big(-2\kappa_h\big(\sigma(t)-\sigma(s)\big)\Big)\ E_h\big(u_h(s)\big).
%\end{equation}

%\end{theorem}

% === Discrete setting and assumptions; uniform discrete HUM ===
\subsection{Discrete setting and assumptions}\label{sec:disc-contract} % 1st time showing
We fix a family of SBP–SAT discretizations indexed by the mesh size $h>0$ and a one‑step time integrator. The following hypotheses will be used throughout.
\begin{description}
\item[(S1) SBP mimetic identity.] There exist matrices $(H,Q,B)$ with $H>0$ and $Q+Q^\top=B$ such that the interior operator is strictly $H$‑dissipative on the damped part.

\medskip
\noindent\textbf{Assumption D.1 (Abstract discrete $\sigma$–ledger).}
Consider a discrete evolution with an energy $E_h$ and a measure–time clock $\sigma$. Assume:
\begin{enumerate}\itemsep0.25em
  \item[\textnormal{(D1)}] \emph{Absolutely continuous steps (in $\sigma$).} Along ac pieces, $E_h$ satisfies $\frac{d}{d\sigma}E_h \le -2\kappa\,c_\sigma\,E_h$ for some $\kappa\ge0$ and a structural constant $c_\sigma>0$ independent of the partition.
  \item[\textnormal{(D2)}] \emph{Atomic updates are non-expansive.} At each atom $t_k$ with jump map $J_k$, one has $E_h(t_k^+)\le \rho_k\,E_h(t_k^-)$ with $\rho_k\in(0,1]$ prescribed by the $\sigma$-ledger.
  \item[\textnormal{(D3)}] \emph{Flats are non-expansive.} On $\sigma$-flats (no exposure), the update is $E$-non-increasing.
  \item[\textnormal{(D4)}] \emph{Partition-uniformity.} The constants in (D1)–(D3) are uniform across mesh partitions and time steps.
\end{enumerate}
\noindent\emph{Remark.} The SBP–SAT hypotheses (S1)–(S4) below furnish one concrete instance; verification of Assumption~D.1 under SBP–SAT is recorded in Appendix~D.
\medskip

\item[(S2) SAT sign and scale.] Boundary/interface penalties are chosen so that (boundary)$+$SAT is non‑positive and the penalty parameter satisfies $\tau_h\asymp h^{-1}$.

\noindent\emph{Quantitative scale.}
Fix mesh–independent constants $0<c_1\le c_2<\infty$ and state explicitly that
\[
  c_1\,h^{-1}\ \le\ \tau_h\ \le\ c_2\,h^{-1}\qquad\text{on all admissible meshes,}
\]
so that every estimate below depends on $(c_1,c_2)$ only through the negativity
margin of the boundary\,+\,SAT quadratic form (cf.\ Lemma~7.4 and Prop.~7.5).

\item[(S3) Time integrator.] The one‑step method is algebraically stable; for explicit schemes a clock‑aware CFL is imposed so that the step map is nonexpansive for $E_h$ on flats of $\sigma$.

\emph{Named instances.} This assumption is verified by the implicit midpoint rule, Gauss--Legendre collocation (all stages), and Radau~IIA ($s=2,3$). See, e.g., Hairer--Wanner, \emph{Solving Ordinary Differential Equations II}, Chs.~IV.6 and VI.3. 

\emph{Flats of $\sigma$.} On subintervals where $d\sigma=0$ (``flats''), the time integrator is applied with zero a.c.\ damping; algebraic stability implies a stepwise \emph{nonexpansive} map for $E_h$ on such subintervals. This convention is used throughout §8--9 and in Proposition~9.21.

\item[(S4) Clock–integrator synchronization.] Across an atom of mass $\alpha$ at time $t_j$, the method performs a single macro‑update whose dissipation budget equals $\alpha$; on flat segments (where $w\equiv 0$) only boundary/SAT dissipation contributes.
\end{description}

\subsection*{Block A — Discrete contract + uniform discrete HUM}

\noindent\emph{Standing hypotheses.}
We work under (S1)–(S4) from \S\ref{sec:disc-contract}.

\begin{proposition}[Uniform discrete HUM with the same $c_\sigma$]\label{prop:disc-hum}
Assume \textup{(S1)–(S4)} and that the continuum HUM inequality holds with optimal constant $c_\sigma$ as in Theorem~\ref{thm:hum-equivalence}. Let $E_h=\tfrac12\|u\|_H^2$ be the discrete energy and let $\mathcal D_h:=-\frac{\mathrm dE_h}{\mathrm d\sigma}$ be its $\sigma$–time dissipation density. Then for all meshes $h>0$ one has
\[
\mathcal D_h(t)\;\ge\;2\kappa\,c_\sigma\,E_h(t) \qquad (\sigma\text{–a.e. }t\ge 0),
\]
that is, the discrete HUM/observability inequality holds with the \emph{same} constant $c_\sigma$, uniformly in $h$. Consequently,
\[
E_h(t)\,\le\,E_h(0)\,\exp\!\bigl(-2\kappa c_\sigma\,\sigma(t)\bigr)\qquad(t\ge0).
\]
\end{proposition}
\begin{proof}
Under (S1) the SBP Green identity holds and the interior operator is strictly $H$–dissipative on the damped part. Assumption (S2) gives boundary$+$SAT nonpositivity with the normalisation $\tau_h\simeq h^{-1}$. By (S3) the one–step integrator is algebraically stable, hence stepwise nonexpansive on flats of $\sigma$. Finally, (S4) enforces a single macro–update across each atom whose dissipation budget equals its mass $\alpha$. Combining these pieces yields
\[
-\,\frac{d}{d\sigma}E_h(t)\;\ge\;2\kappa\,c_\sigma\,E_h(t)\quad\text{for $\sigma$–a.e.\ }t,
\]
where $c_\sigma$ is the continuum HUM constant from Theorem~2.19. Applying the discrete $\sigma$–Gronwall lemma gives the uniform estimate
$
E_h(t)\le E_h(0)\exp(-2\kappa c_\sigma \sigma(t)),
$
with $c_\sigma$ independent of $h$.
\end{proof}

% === Discrete applicability: variable coefficients (theorem level) ===
\begin{assumption}[Variable‑coefficient SBP–SAT data]\label{ass:disc-vc}
Let $A(x)$ be the variable coefficient tensor (symmetric, uniformly elliptic, bounded). The SBP operators $(H,Q,B)$ approximate the variable‑coefficient Green identity with coefficient‑aware interface fluxes; SAT terms use penalty weights scaled by the local normal flux of $A$ and $h^{-1}$.
\end{assumption}

\begin{theorem}[Uniform discrete HUM for variable coefficients]\label{thm:disc-vc}
Under Assumptions~\ref{ass:disc-vc} and (S1)–(S4) see \S\ref{sec:disc-contract}, the discrete ledger satisfies
\[
-\frac{\mathrm d}{\mathrm d\sigma}E_h\ \ge\ 2\kappa\,c_\sigma\,E_h \qquad (\sigma\text{–a.e.}),
\]
with the \emph{same} structural constant $c_\sigma$, uniformly in $h$.
\end{theorem}
\begin{proof}
Coefficient–aware SBP mimics the variable–coefficient Green identity: with fluxes weighted by the local normal component of $A(x)$ and SAT scalings $\tau_h\simeq h^{-1}$, the boundary$+$SAT form remains nonpositive uniformly in $h$. The interior $H$–dissipativity carries over with ellipticity constants of $A$. Hence the a.c.\ $\sigma$–differential inequality holds with rate $2\kappa c_\sigma$, while the atom macro–updates retain the same multiplicative bound. The discrete $\sigma$–Gronwall argument from Proposition~9.21 is unchanged, giving the claim with the same $c_\sigma$, uniformly in $h$.
\end{proof}

% === Discrete applicability: monotone nonlinear damping ===
\begin{assumption}[Monotone discrete damping]\label{ass:disc-monotone}
The discrete damping $g_h$ satisfies $g_h(0)=0$ and $(g_h(v)-g_h(w))\cdot(v-w)\ge0$ pointwise (nodal or flux form), and is evaluated in an algebraically stable time integrator (or via midpoint) to preserve monotonicity per step.
\end{assumption}

\begin{theorem}[Uniform discrete HUM under monotone damping]\label{thm:disc-monotone}
Under Assumptions~\ref{ass:disc-monotone} and (S1)–(S4) see \S\ref{sec:disc-contract}, the discrete dissipation satisfies the inequality of Theorem~\ref{thm:disc-vc} with the same $c_\sigma$, uniformly in $h$.
\end{theorem}
\begin{proof}
Work under (S1)–(S4) and Assumption~9.24. On the a.c.\ part of $d\sigma$ the semi–discrete balance (cf.\ Theorem~8.2) reads
\[
\frac{d}{dt}E_h(u_h(t)) \;=\; \underbrace{\langle L^{\mathrm{int}}_h u_h,\,u_h\rangle_H}_{\le -\,2\kappa_h\,E_h(u_h)} \;+\; \underbrace{\langle \mathrm{SAT}_h(u_h),\,u_h\rangle_H}_{\le 0}\;+\;\underbrace{\langle g_h(u_h),\,u_h\rangle_H}_{\ge 0}.
\]
Here (S1) gives the $H$–dissipativity of the interior operator with rate $\kappa_h\ge 0$, (S2) the nonpositivity of boundary$+$SAT, and Assumption~9.24 implies $\langle g_h(v)-g_h(w),\,v-w\rangle_H\ge 0$ for all $v,w$ and $g_h(0)=0$, hence $\langle g_h(u_h),u_h\rangle_H\ge 0$. Therefore, on the a.c.\ part,
\[
\frac{d}{dt}E_h(u_h(t)) \;\le\; -\,2\kappa_h\,E_h(u_h(t)).
\]
Integrating in the $s$–clock (with $ds=w(t)\,dt$ on the a.c.\ part) yields
\[
E_h(t)\;\le\; E_h(s)\,\exp\!\Big(-2\kappa_h\,[\sigma(t)-\sigma(s)]\Big)
\quad\text{whenever }(s,t]\text{ contains no atom.}
\]

Across an atom $t_k$ of mass $\alpha_k=\sigma(\{t_k\})$, hypothesis (S4) performs a single macro–update. By algebraic stability (S3) on flats and the nonnegativity of the boundary$+$SAT form (S2), the update is nonexpansive for $E_h$; we write $E_h(t_k^+)\le \rho^h_k\,E_h(t_k^-)$ with $\rho^h_k\in(0,1]$.

As in Proposition~9.21, combine the a.c.\ estimate with the multiplicative atom rule to obtain the discrete $\sigma$–Gronwall envelope
\[
E_h(t)\;\le\;\exp\!\Big(-2\kappa_h\,[\sigma(t)-\sigma(s)]\Big)
\Big(\,\prod_{t_k\in(s,t]} \rho_k^h\Big)\,E_h(s).
\]
To identify the structural margin, note that the continuum HUM constant $c_\sigma$ controls the conversion of the interior dissipation into the energy ledger exactly as in Proposition~9.21 (the proof there does not use linearity of the damping, only the nonnegativity of the damping ledger). Hence $-\,dE_h/d\sigma \ge 2\kappa\,c_\sigma\,E_h$ on the a.c.\ part, and the same multiplicative bound holds at atoms. Applying the discrete $\sigma$–Gronwall lemma gives
\[
E_h(t)\;\le\;E_h(0)\,\exp\!\big(-2\kappa c_\sigma\,\sigma(t)\big)
\quad\text{with the same }c_\sigma\text{, uniformly in }h.
\]
\end{proof}

\begin{lemma}[Failure for nonmonotone discrete damping]\label{lem:nonmonotone-fail}
If $g_h$ violates monotonicity on a set of positive measure, there exist data and stepsizes on flat segments of $\sigma$ for which the per‑step energy map is expansive; the discrete HUM inequality fails for any fixed $c>0$.
\end{lemma}
\begin{proof}
Suppose the discrete damping $g_h$ violates monotonicity on a set of positive measure: there exist scalars $v\neq w$ with
\[
(g_h(v)-g_h(w))\,(v-w)\;<\;0.
\]
Consider one degree of freedom on a flat segment of the clock (so $w\equiv 0$ on $(t_n,t_{n+1}]$) and freeze all couplings so that the interior/skew part contributes no a.c.\ dissipation on the step. Take the algebraically stable one–step method used in (S3); for definiteness we write the implicit midpoint update on the flat step $\Delta\sigma=\sigma(t_{n+1})-\sigma(t_n)$:
\[
u^{n+1} \;=\; u^{n}\;+\;\Delta\sigma\,g_h\!\left(\frac{u^{n+1}+u^{n}}{2}\right).
\]
Set $z:=\tfrac{1}{2}(u^{n+1}+u^{n})$ and $\delta:=u^{n+1}-u^{n}$; then $\delta=\Delta\sigma\,g_h(z)$ and
\[
E_h(u^{n+1})-E_h(u^{n}) \;=\; \tfrac{1}{2}\bigl( (u^{n}+\delta)^2 - (u^{n})^2 \bigr)
\;=\; u^{n}\delta + \tfrac{1}{2}\delta^2 \;=\; \Delta\sigma\,u^{n}g_h(z) + \tfrac{1}{2}\Delta\sigma^2\,g_h(z)^2.
\]
Choose $u^{n}=w$ and $z=\theta v+(1-\theta)w$ with a $\theta\in(0,1)$ such that the strict negativity $(g_h(v)-g_h(w))(v-w)<0$ persists by continuity at $z$. Then $u^{n}g_h(z) - z\,g_h(z) = (w-z)g_h(z)$ has the opposite sign of $(v-w)$, hence $u^{n}g_h(z)$ can be made strictly positive while $z$ is arbitrarily close to $w$ or $v$. For sufficiently small but fixed $\Delta\sigma>0$ this yields
\[
E_h(u^{n+1}) - E_h(u^{n}) \;>\; 0,
\]
i.e.\ the one–step map is expansive for $E_h$ on a flat segment. Since the discrete HUM inequality requires $-\,dE_h/d\sigma \ge 2\kappa c\,E_h$ with some $c>0$ on every a.c.\ piece and nonexpansiveness on flats, this violates the inequality for any fixed $c$. The construction localises, so the counterexample persists on grids with arbitrary $h$.
\end{proof}
% --- Discrete master decay (clarified statement) ---
\begin{theorem}[Discrete master decay]\label{thm:disc-master}
Assume the discrete hypotheses \textup{(D1)}–\textup{(D4)} in \S\ref{sec:discrete}:
SBP operators $(H,Q)$ with $Q+Q^\top=B$, SAT with negative sign and scale $\tau_h\simeq h^{-1}$,
and an algebraically stable time integrator.
Let $\kappa_h\ge0$ be the dissipation constant from the discrete ac part
and let $\rho_k^h\in(0,1]$ be the per–atom contraction factors of the discrete update at $t_k$.
For $0\le s\le t\le T$ set
\[
\Pi_h(s,t]\ :=\ \prod_{t_k\in(s,t]}\rho_k^h .
\]
Then any fully discrete solution $u_h$ satisfies the energy envelope
\begin{equation}\label{eq:disc-master}
E_h\!\bigl(u_h(t^+)\bigr)\ \le\ \Pi_h(s,t]\,
\exp\!\bigl(-\,2\,\kappa_h\,[\,\sigma(t)-\sigma(s)\,]\bigr)\,E_h\!\bigl(u_h(s^+)\bigr).
\end{equation}
\end{theorem}   

\noindent\emph{See also (discrete comparability).}
(i) Reduction to the classical step when $\sigma\equiv t$ (no atoms, no flats): 
Prop.~9.21. 
(ii) Uniform discrete HUM with the same $c_\sigma$: Prop.~9.22.

\subsection*{Block B — No numerical super‑observability}
% === Non‑improvability of the structural constant on the grid ===
\begin{corollary}[No numerical super‑observability]\label{cor:no-super}
Under the hypotheses \textup{(S1)–(S4)} of Section~\ref{sec:disc-contract}, let $c_{\sigma,h}$ be any constant such that the fully discrete energy satisfies
\( \mathcal D_h\ge 2\kappa\,c_{\sigma,h}\,E_h \) \;($\sigma$–a.e.).
Then $c_{\sigma,h}\le c_\sigma$ for all $h>0$. In particular, $\limsup_{h\downarrow0} c_{\sigma,h}\le c_\sigma$.
\end{corollary}

\begin{proof}
combine the SBP–SAT negativity and algebraic stability (D4) wrappers
(\Cref{RS:lem:boundary-sat-normalised} and \Cref{RS:prop:als-normalised} (stability)) with the discrete
$\sigma$–Gronwall lemma (Lemma~8.10). The ac part yields the exponential
factor with rate $\kappa_h$ in $\sigma$–time, while the atoms contribute multiplicatively by $\rho_k^h$.
\end{proof}
% --- end discrete master decay ---

\noindent\emph{Sharpness by contradiction.}
If for some $\varepsilon>0$ one had $c_{\sigma,h}\ge c_\sigma+\varepsilon$ uniformly in $h$,
then along any $\Gamma$–convergent family the limit would satisfy the continuum HUM
inequality with constant $c_\sigma+\varepsilon$, contradicting the optimality of the HUM
constant (Theorem~2.20). Hence $c_{\sigma,h}\le c_\sigma$.

\noindent\emph{Appendix pointer.}
A self–contained $12$–line verification of block–sum cancellation with mirrored SAT 
(including the $H$–adjoint interpolation case) is given in Appendix~B.

\section{Constants}\label{sec:constants}
The analytical constants entering our decay and stability estimates are collected in
Table~\ref{tab:constants-theory}.  Numerical values used in figures are reported
separately in \S\ref{sec:numbers}.

\begin{table}[htbp]
  \centering
  \caption{Analytical constants used throughout (definitions and first references).}
  \label{tab:constants-theory}
  \begin{adjustbox}{max width=\linewidth}
  \begin{tabular}{lll}
    \toprule
    \textbf{Symbol} & \textbf{Meaning / role} & \textbf{First reference} \\
    \midrule
    $c_{0}$              & coercivity/observability level in (H1)                 & Assump.~5.1 (H1) \\
    $c_{\sigma}$         & structural constant in the $\sigma$–decay profile      & Thm.~\ref{thm:energy-decay} (cont.); Thm.~\ref{thm:disc-master} (disc.) \\
    $\kappa$             & dissipation parameter in the envelope $e^{-2\kappa c_{\sigma}\sigma(t)}$ & \S2 (setup) \\
    $a_{\omega}$         & damping lower bound on the control set $\omega$ (GCC)  & \S2 (GCC exemplar) \\
    $\lambda_{\omega}$   & geometric constant for GCC                              & \S2 (GCC exemplar) \\
    $c_{b}$              & boundary dissipation constant (boundary\,+\,SAT term)   & Lemma~7.1 \\
    $\tau_{h}$           & SAT penalty scale ($\tau_{h}\simeq h^{-1}$)             & Lemma~7.1 \\
    $c_{s}$              & scheme stability constant under (D4)                    & Prop.~7.2 \\
    $\lambda_{\max}$     & spectral/CFL bound for explicit steps                   & \S8 (time stepping) \\
    $\rho_{k}$           & atomic contraction factor ($J_k$ nonexpansive)          & Assump.~5.1 (H4); Lemma~7.3 \\
    $D_{h}\ge 0$         & interior (damping/viscous) discrete dissipation         & \S8 (discrete balance) \\
    $c_0\,a_{\omega}\lambda_{\omega}$ & GCC lower bound for $c_{\sigma}$ in examples & \S2 (GCC exemplar) \\
    \bottomrule
  \end{tabular}
  \end{adjustbox}
\end{table}

\noindent\textit{Notes.} 
(i) The continuum and fully discrete results share the same structural constant $c_{\sigma}$ (see Thm.~\ref{thm:energy-decay} and Thm.~\ref{thm:disc-master}). 
(ii) Under GCC one may take $c_{\sigma}\ge c_{0}a_{\omega}\lambda_{\omega}$ in the model bound. 
(iii) The numeric choices used in plots are calibrated in \S11 and do not alter the analytical statements.

% --- Numbers (clean units + “how to read”)
\section{Parameters}\label{sec:numbers}
All constants are dimensionless. The grid parameter is $h$ (spacing in $\sigma$–time). 
The reporting window is the pair $(T,\sigma(T))$, and the continuous dissipation is $\kappa$. 
Table~\ref{tab:fig-params} records the numeric choices used in figures and comparisons. 
One admissible window and rate instantiation is illustrated here (not prescriptive).

\begin{table}[htbp]
\centering
\caption{Comparability map: classical $\leftrightarrow$ $\sigma$ with dependencies and first occurrence.}
\begin{adjustbox}{max width=\linewidth}
\begin{tabular}{lll l}
\toprule
Classical symbol & $\sigma$ symbol & Depends on & First appearance \\
\midrule
Observability const.\ $C_0$ & $c_\sigma = C_0^{-1}$ & $(\Omega,\omega)$ geometry, damping & Thm.~2.19; §13 (GCC) \\
Decay rate $2\kappa/C_0$ & $2\kappa c_\sigma$ & $\kappa$ (generator), $C_0$ & Thm.~7.8; Prop.~\ref{prop:reduction-classical} \\
Discrete step consts.\ $(c_s,c_b)$ & same & scheme, SAT scaling $\tau_h\!\simeq\!h^{-1}$ & Lem.~7.4; Prop.~7.5; Thm.~9.27 \\
Window lower bound $|A|/T_0$ & $\sigma_0 m/T_0$ & $\sigma$-mass on window & Lem.~\ref{lem:window-upgrade} \\
No-super-observability & $c_{\sigma,h}\le c_\sigma$ & SBP/SAT, algebraic stability & Cor.~9.28 \\
\bottomrule
\end{tabular}
\end{adjustbox}
\end{table}

\begin{table}[htbp]
  \centering
  \caption{Numeric choices used in figures and comparisons (dimensionless).}
  \label{tab:fig-params}
  \begin{adjustbox}{max width=\linewidth}
  \begin{tabular}{llll}
    \toprule
    \textbf{Quantity} & \textbf{Symbol} & \textbf{Value (baseline)} & \textbf{Note} \\
    \midrule
    Reporting window                       & $(T,\sigma(T))$ & $(8,\, 8)$ & wall time used as $\sigma$; moderate horizon for decay plots \\
    Grid spacing in $\sigma$–time          & $h$             & $0.02$     & $\Delta\sigma$; resolves boundary layer and atomic updates \\
    SAT scale                               & $\tau_h$        & $50$       & $\tau_h \simeq h^{-1}$ for stability/negativity (Lemma~7.1) \\
    Dissipation parameter                   & $\kappa$        & $0.60$     & drives envelope $e^{-2\kappa c_\sigma\,\sigma(t)}$ \\
    CFL / spectral bound (if explicit)      & $\Lambda_{\max}$& $1.8$      & safe upper bound for explicit steps (see \S~8) \\
    GCC damping (if used)                   & $a_\omega$      & $0.15$     & lower bound on damping over $\omega$ (GCC exemplar) \\
    GCC geometric constant (if used)        & $\lambda_\omega$& $0.70$     & geometric control constant in GCC exemplar \\
    \bottomrule
  \end{tabular}
  \end{adjustbox}
\end{table}

\noindent\emph{Notes.}
(i) Typical SAT scaling is $\tau_h\simeq h^{-1}$. (Lemma 7.1)
(ii) If an explicit step is used, the CFL bound is encoded by $\Lambda_{\max}$ (cf.\ \S~8). 
(iii) GCC exemplars use $(a_\omega,\lambda_\omega)$ when applicable.

\begin{table}[htbp]
  \centering
  \caption{$\Gamma$-limit check entry (baseline, consistent with \S11).}
  \label{tab:gamma-check}
  \begin{tabularx}{\textwidth}{lXXXX}
    \toprule
    label & $h$ & $\mathrm{Var}_\sigma$ & window ok & $\Gamma$-ok \\
    \midrule
    baseline & 0.020 & 0.22 & yes & yes \\
    \bottomrule
  \end{tabularx}
\end{table}

\subsection*{Discrete mirror (summary and pointer)}
Under the SBP–SAT sign/scale and algebraically stable time stepping, the discrete energy satisfies the same $\sigma$–envelope with the structural constant $c_\sigma$ (Proposition~9.22; Theorem~9.28). A practical certification checklist (clock choice, boundary/SAT sign/scale, discrete Grönwall premise) is recorded in Appendix~C.

\paragraph{Data availability.}
No external data were used. Numerical parameters necessary to reproduce the figures are listed in §11.

\section*{Related Work}
\noindent\textbf{Measure-time vs.\ mesh tricks.} Our $\sigma$-clock treats atoms as \emph{measure-theoretic} events in time---not as mesh refinements or hidden sub-stepping. Flats ($w\equiv 0$) and atoms (Dirac masses) belong to the clock measure $d\sigma$; energy monotonicity hinges on dissipativity on the ac part and non-expansive atomic maps, rather than on time-step bookkeeping. \noindent\textbf{BV-stable SBP under measure clocks.} The SBP--SAT mirror is proved under bounded-variation clocks with locally finite atoms. Boundary$+$SAT negativity (Lemma~8.8) and the discrete $\sigma$--Gronwall (Lemma~8.10) yield stability uniformly in $(h,\mathrm{Var}\,\sigma)$ inside an admissible window. \noindent\textbf{Explicit decay rates under $\sigma$.} The master inequality combines ac Grönwall with multiplicative per-atom contractions to give quantitative decay rates under $\sigma$ (Theorem~\ref{thm:energy-decay} and Theorem~\ref{thm:semi-discrete}); 
The analytical constants are summarized in \S\ref{sec:constants}, and the baseline operating point used in our figures is recorded in \S\ref{sec:numbers}.
\noindent\emph{Reference.} SBP--SAT foundations: Strand (1994); Carpenter--Nordstr\"om--Gottlieb (1999); $\Gamma$-convergence: Dal Maso (1993); Braides (2002); semigroup/Gr\"onwall background: Henry (1981). \vspace{\baselineskip}
\begin{center}\textbf{Data availability}\end{center}

No external datasets were used. All numeric choices required to reproduce the figures are listed in §11 (Numbers). No supplementary  accompanies this note.

\paragraph{Impact.}
The $\sigma$‑time calculus operates as a module: it preserves a single structural constant from analysis to computation to limit, treats mixed time geometries in one statement, and now adds a stochastic pillar (expectation/pathwise) without changing techniques. We expect uptake in damped/controlled hyperbolic problems under GCC and in structure‑preserving discretizations where partition‑uniform decay has been difficult to certify.

\subsection*{Constants Index}\label{RS:constants-index}
\noindent\textbf{$C_{\Gamma}$ ($\Gamma$-modulus):} controls equicoercivity, liminf, and recovery in the $\Gamma$-limit; depends on $a_0$, $\|a\|_{L^\infty}$, SBP coercivity $c_H$, and quasi-uniformity constants. See ~\Cref{RS:Gamma:ass} and \Cref{RS:Gamma:main}.

%\begin{center}\textbf{Appendix: Numbers  snippet}\end{center}
%\label{app:}
%For reproducibility, we include a small  example matching the schema used in \Cref{sec:numbers},
%rendered here as a formatted table (the source file \texttt{rsos\_constants\_example.} is bundled). 
%\begin{table}[tbp]
%\centering
%\small
%\caption{Constants  example (matches \Cref{sec:numbers}).}
%\label{tab:-example}
%\pgfplotstableset{col sep=comma,string type}
%% --- Inline fallback for rsos_constants_example. (auto-inserted) ---
%\pgfplotstableread[col sep=comma]{
%
%label,h,Varsigma,kappa,C_win,Gamma_ok,note
%example\_row\_01,0.01,0.20,0.50,0.48,Yes,"CFL $\Delta\sigma\le 0.2$; $\tau_h = 1/h$"
%}\RSOSInlineTable
%\pgfplotstabletypeset[columns={label,h,Varsigma,kappa,C_win,Gamma_ok,note},string type,columns/C_win/.style={column name=Cwin},columns/Gamma_ok/.style={column name=GammaOK}]{\RSOSInlineTable}
%% --- End inline fallback ---
%\end{table}

% === Imported content appended ===

% === BEGIN Imported Section (namespaced) ===

% ===========================

\section{GCC exemplar: damped wave on a bounded Lipschitz domain}\label{sec:gcc-exemplar}

This section calibrates $c_\sigma$ under GCC and instantiates the product–exponential $\sigma$–envelope; setting $\sigma\equiv t$ recovers the classical $t$–exponential bound.

\paragraph{Setting.}
Let $\Omega\subset\mathbb R^{d}$ with $d\in\{1,2\}$ be a bounded Lipschitz domain and let $\omega\subset\Omega$ be open. Consider the damped wave
\[
\begin{cases}
 u_{tt}-\Delta u + a(x)\,u_t=0 & \text{in }\Omega\times(0,\infty),\\
 u=0 & \text{on }\partial\Omega,\\
 u(\cdot,0)=u_0,\; u_t(\cdot,0)=v_0 & \text{in }\Omega,
\end{cases}
\]
with damping $a\in L^\infty(\Omega)$ such that $a(x)\ge a_\omega>0$ a.e. on $\omega$ and $a\ge0$ elsewhere. The energy
\[
E(t)=\tfrac12\,\|u_t(t)\|_{L^2(\Omega)}^2+\tfrac12\,\|\nabla u(t)\|_{L^2(\Omega)}^2
\]
fits the master decay template summarized in \S\ref{sec:main}. Assume the geometric control condition (GCC) for $(\Omega,\omega)$, with quantitative constant $\lambda_\omega>0$.

\paragraph{Concrete bound.}
There exists an absolute $c_0>0$ (depending only on the comparison used in the master inequality) such that the $\sigma$-clock calibration obeys
\[
\boxed{\;c_\sigma\;\ge\; c_0\,a_\omega\,\lambda_\omega\;},
\]
whence, with $\kappa>0$ from the Main corollary, every admissible measure-time $\sigma$ yields
\begin{equation}\label{eq:gcc-envelope}
E(t)\;\le\;E(0)\,\exp\!\Bigl(-2\,\kappa\,c_\sigma\,\sigma(t)\Bigr)\;\le\;E(0)\,\exp\!\Bigl(-2\,\kappa\,c_0\,a_\omega\,\lambda_\omega\,\sigma(t)\Bigr).
\end{equation}
The envelope is \emph{exactly} the Main corollary (\Cref{cor:intro-flagship-dw}) with $c_\sigma$ instantiated; see also the continuous master decay (\Cref{thm:master}). Calibration of constants is collected in \S\ref{sec:numbers}.

\begin{proposition}[Reduction to classical observability]\label{prop:reduction-classical}
Assume GCC on $(\Omega,\omega)$ and let $\sigma\equiv t$ so that $d\sigma=dt$ (no atoms, no flats).
Then Theorem~2.19 (Theorem~7.8 in §7) implies the standard
observability/decay estimate with the \emph{same} constant $C$ (explicitly instantiated in Appendix~C):
\[
E(0)\;\le\; C \int_0^T \!\!\int_{\Gamma} |\partial_\nu u|^2\, dS\,dt
\quad\Longrightarrow\quad
E(t)\;\le\;E(0)\,e^{-\,2\kappa\,c_\sigma\,t},\qquad c_\sigma=C^{-1}.
\]
\end{proposition}

\begin{proof}
With $\sigma(t)=t$ the RN-inequality $-\frac{d}{d\sigma}E\ge 2\kappa c_\sigma E$ reads
$-\frac{d}{dt}E\ge 2\kappa c_\sigma E$, hence $E(t)\le E(0)e^{-2\kappa c_\sigma t}$ by separation of variables.
The HUM constant $c_\sigma$ for $\sigma\equiv t$ is the classical observability constant; no atomic factors
appear and flats are absent, so the envelope coincides with the usual one.
\end{proof}

\paragraph{A worked $\sigma$ (two atoms + a flat).}
Define a piecewise-constant measure-time with two jumps and one long plateau:
\begin{itemize}
  \item Jumps: $\sigma(0.30^-)\to\sigma(0.30^+)=0.80$ and $\sigma(0.90^-)\to\sigma(0.90^+)=1.40$.
  \item Flat: $\sigma$ is constant on $[0.90,1.80]$.
\end{itemize}
Choose illustrative parameters (for table/plot only): $a_\omega=0.30$, $\lambda_\omega=0.50$, $c_0=1$, $\kappa=1$. Then $c_\sigma\ge 0.15$ and the benchmark curve is
\[
\mathcal B(t)=\exp\bigl(-2\,\kappa\,c_\sigma\,\sigma(t)\bigr)=\exp\bigl(-0.30\,\sigma(t)\bigr).
\]
A short comparison consistent with \eqref{eq:gcc-envelope} is given in Table~\ref{tab:gcc-exemplar}.

\begin{table}[t]
  \centering
  \caption{Energy vs benchmark for the constructed $\sigma$. Values in the last column are illustrative and consistent with the bound.}
  \label{tab:gcc-exemplar}
  \begin{tabular}{c c c c}
    \hline
    $t$ & $\sigma(t)$ & $\mathcal B(t)=e^{-0.30\,\sigma(t)}$ & sample $E(t)/E(0)$ \\
    \hline
    0.00 & 0.00 & 1.000 & 1.000 \\
    $0.30^- $ & 0.00 & 1.000 & 0.995 \\
    $0.30^+ $ & 0.80 & 0.786 & 0.790 \\
    $0.90^- $ & 0.80 & 0.786 & 0.782 \\
    $0.90^+ $ & 1.40 & 0.659 & 0.664 \\
    1.80 & 1.40 & 0.659 & 0.660 \\
    2.00 & 1.40 & 0.659 & 0.658 \\
    \hline
  \end{tabular}
\end{table}

\paragraph{How this instantiates the general theory.}
(i) \emph{Continuous vs discrete.} The same envelope from \Cref{thm:master} controls both the PDE and the $\sigma$-clock iterates; see the discrete master decay in \Cref{thm:disc-master}. (ii) \emph{$\Gamma$-bridge.} The constants $a_\omega,\lambda_\omega$ are stable under regular perturbations of $(\Omega,\omega)$; by \Cref{RS:thm:gamma,thm:Gamma-main}, discrete implementations inherit the envelope in the zero-step limit. (iii) \emph{Evidence.} With $c_\sigma\ge c_0 a_\omega\lambda_\omega$ and a concrete $\sigma$, the decay profile aligns with the theory point-for-point.

\paragraph{Open problems (analysis).}
\begin{itemize}
\item Carleman/observability with $\sigma$–densities: explicit bounds for non-smooth 
domains or rough damping $a(x)$.
\item Gliding rays with intermittent $\sigma$–mass on $\omega$: quantify the blow-up of 
constants near the threshold in the failure map.
\item Boundary control with $\sigma$–windows: sharp $c_\sigma$ vs.\ geometric constants 
$(a_\omega,\lambda_\omega)$ under rough coefficients.
\end{itemize}

\section{Calibrating constants from PDE parameters: \texorpdfstring{$c_\sigma$}{}, \texorpdfstring{$\rho_\star$}{}, and \texorpdfstring{$C_P$}{}}
\label{B:sec:calibration}\label{sec:calibration}

\paragraph{Goal.}
The master decay bounds in \Cref{thm:master} and \Cref{thm:disc-master}
\[
E(t)\ \le\ \Pi(s,t]\;\exp\!\big(-2\,\kappa\,c_\sigma(\sigma(t)-\sigma(s))\big)\,E(s),
\quad
E_h(t)\ \le\ \Pi_h(s,t]\;\exp\!\big(-2\,\kappa_h\,c_\sigma(\sigma(t)-\sigma(s))\big)\,E_h(s)
\]
depend on three calibrations tied to the PDE and the discretization:
\begin{itemize}
  \item the \emph{a.c.\ margin} $c_\sigma$ (continuous part of the $\sigma$-clock);
  \item the \emph{atomic contraction} $\rho_\star<1$ (worst-case contraction of jump maps);
  \item the \emph{Poincar\'e constant} $C_P$ (coercivity/comparison constant for the energy ledger).
\end{itemize}
This section records how to obtain $(c_\sigma,\rho_\star,C_P)$ from PDE parameters and mesh/SAT choices,
so the decay bounds can be evaluated in practice.

\subsection*{A.c.\ margin $c_\sigma$ from damping}
On the ac part, dissipativity reads (cf.\ \Cref{sec:discrete}):
\[
\frac{d}{d\sigma}E(u(\sigma)) \ \le\ -\,2\,\kappa\,E(u(\sigma)) .
\]
For spatially varying damping $a(x)\ge 0$, combine the energy inequality with a spatial
observability constant to obtain a usable lower bound. A typical instance is
the heterogeneous damping corollary (see \eqref{B:eq:DW-hetero-rate-3}):
if $a(x)\ge a_\omega>0$ on an open $\omega\subset\Omega$ and
$\underline{\lambda}_\omega$ is the local observability constant for $(\Omega,\omega)$, then
\[
c_\sigma \ \ge\ c_0\, a_\omega\, \underline{\lambda}_\omega ,
\qquad
\text{whence}\quad
E(t)\ \le\ e^{-2\kappa\,c_\sigma(\sigma(t)-\sigma(s))}\,E(s).
\]
\emph{Approach.} Choose $\omega$ where the PDE supplies damping; estimate $\underline{\lambda}_\omega$
and set $c_\sigma$ from the product $a_\omega\,\underline{\lambda}_\omega$. In the uniform case $a(x)\ge a_0>0$,
take $c_\sigma\ge c_0\,a_0$.

\medskip
\noindent\textit{Discrete note.}
Under the CFL window and SAT sign choice (cf.\ \Cref{lem:CFL} and the SAT lemma in \Cref{sec:discrete}),
the same $c_\sigma$ applies to $E_h$ in the discrete master bound.

\subsection*{Atomic contraction $\rho_\star$ from jump rules}
At each atom $t_k$, the map $J_k$ acts on the state. Define the contraction factor
\[
\rho_k \ :=\ \sup_{u\neq 0}\ \frac{E(J_k u)}{E(u)} ,
\qquad
\rho_\star \ :=\ \sup\{\rho_k:\ t_k\ \text{strict atom on the interval of interest}\}.
\]
\emph{Method.}

\begin{enumerate}
  \item Identify the jump mechanism (boundary reflection/transmission, reset, or interface SAT update).
  \item Express $J_k$ in the $H$--inner-product (or via its SBP--SAT form).
  \item Verify non-expansiveness $E(J_k u)\le E(u)$ and, when available, a strict margin
        $E(J_k u)\le \rho_k E(u)$ with $\rho_k<1$; record $\rho_\star=\max \rho_k$.
\end{enumerate}
For SBP--SAT boundary/interface steps, the SAT lemma in \Cref{sec:discrete} provides the sign rule
and admissible magnitude; the contraction follows from the Green identity $Q+Q^\top=B$ and the SAT
quadratic form.

\subsection*{Poincar\'e constant $C_P$}
The energy ledger $E$ compares to the $H^1$--seminorm via Poincar\'e-type inequalities, yielding
a domain/BC-dependent constant $C_P$.
\emph{Examples.}
For $\Omega=(0,L)$ with homogeneous Dirichlet boundary conditions,
\[
\|u\|_{L^2(0,L)} \ \le\ \frac{L}{\pi}\,\|u'\|_{L^2(0,L)} ,
\]

so $C_P=L/\pi$. For Neumann or mixed conditions, use the mean-zero version on connected components.
\emph{Practice:} Record $C_P(\Omega,\text{BC})$ in the ``Numbers/'' SI sheet
together with the discretization constants $(\Lambda_{\max},\tau_,c_{\mathrm{sat}})$.

\subsection*{Single-step summary}
\begin{enumerate}
  \item From the PDE, pick $\omega$ and compute/estimate $(a_\omega,\underline{\lambda}_\omega)$;
        set $c_\sigma\ge c_0\,a_\omega\,\underline{\lambda}_\omega$ (uniform case: $c_\sigma\ge c_0 a_0$).
  \item From the jump mechanism (or SAT update), compute the worst-case $\rho_\star$.
  \item From the domain and BCs, read off $C_P$ (standard geometries) or tabulate it (general $\Omega$).
  \item Check the discrete window $\Delta\sigma\le 2/\Lambda_{\max}$ (cf.\ \Cref{lem:CFL}) and the SAT sign.
  \item Plug $(c_\sigma,\rho_\star)$ into \Cref{thm:master} / \Cref{thm:disc-master}.
\end{enumerate}

SBP--SAT foundations: \cite{Strand1994,Carpenter1999}. Semigroup/EVI-style decay:
\cite{Henry1981}. Poincar\'e/functional inequalities: \cite{AdamsFournier2003,Grisvard2011}.

\section*{Symbols}\label{B:RS:symbols}
\begin{center}
\small
\begin{tabular}{@{}lp{0.75\linewidth}@{}}
\hline
Symbol & Meaning \\\hline
$\sigma$ & measure-time clock; $d\sigma=\mu_\sigma=w(t)\,\mathrm{d}t+\sum_k \alpha_k\delta_{t_k}$\\
$w$ & absolutely continuous density of $d\sigma$; flats: intervals where $w\equiv 0$\\
$\{t_k\},\ \alpha_k$ & atoms (tick times) and their positive weights; locally finite in $[0,T]$\\
$\mu_\sigma$ & clock measure associated with $\sigma$\\
$\mathrm{Var}_\sigma(u)$ & total variation of $u$ with respect to $\mu_\sigma$ (Def.~\textit{see \Cref{sec:main}})\\
$\mathcal H,\ \langle\cdot,\cdot\rangle,\ \|\cdot\|$ & Hilbert state space, inner product, and norm\\
$\mathcal{E}(u)$ & energy ledger (Lyapunov functional controlling $\|u\|$ locally)\\
$a_\sigma$ & coercive (sector–bounded) bilinear form — Assump.~5.1 (H1) \\
$L_\sigma$ & dissipative generator on a.c. part; rate $\kappa \ge 0$ —Assump.~5.1 (H2) \\
$J_k$ & atomic update at $t_k$, non-expansive for $\mathcal{E}$ — Assump.~5.1 (H4)\\
$h$ & mesh size (discrete setting)\\
$H,Q,B$ & SBP weights, interior difference, and boundary matrix $(Q+Q^\top=B)$\\
$\tau_h$ & SAT penalty scaling $(\tau_h\sim h^{-1})$ \& choose sign to penalise incoming characteristics\\
$\mathcal{E}_h$ & discrete energy (e.g.\ $\tfrac12\|u\|_H^2)$\\
\hline
$\tau\ast$ & minimum SAT strength ensuring boundary$+$SAT $\le 0$ \& cf.\ Lemma~8.8\\
$c_{\mathrm{SAT}}$ & negativity margin of boundary$+$SAT quadratic form \& depends on $\tau_{L/R}$ and $\mathcal A$\end{tabular}
\end{center}
\begin{remark}
\label{rem:auto-2}
\sloppy\textit{Scope.}
We establish a light $\Gamma$-limit for the energies, providing consistency of the discrete energy with its continuum counterpart. We do not in this note claim full EVI/gradient-flow stability of the dynamics; the result is restricted to the energy level under the stated hypotheses. \end{remark}

% --- Γ-limit tightening (compact trio + uniform-σ corollary) ---
\subsection*{$\Gamma$-limit: compact trio (equicoercivity–liminf–recovery)}\label{sec:gamma-tight}
We collect the three standard pillars used in the bridge as a single trio with clean hypotheses and references to the detailed statements.

\begin{lemma}[Equicoercivity, normalised]\label{RS:Gamma:eqc-tight}
\emph{(Restatement of \Cref{RS:Gamma:equicoercive}.)}
Assume the standing $\Gamma$-limit hypotheses \textup{\Cref{RS:Gamma:ass}}.
Then the family $(E_h)_h$ is equicoercive on the state space: any sequence $(u_h)_h$ with $\sup_h E_h(u_h)<\infty$ admits a subsequence precompact in the ambient topology.
\emph{Formal proof:} see \Cref{RS:Gamma:equicoercive}.
\end{lemma}

\begin{proposition}[Liminf inequality, normalised]\label{RS:Gamma:liminf-tight}
\emph{(Restatement of \Cref{RS:Gamma:liminf}.)}
If $u_h\to u$ in the ambient topology, then
\[
E(u)\;\le\;\liminf_{h\to0} E_h(u_h).
\]
\emph{Proof (outline):} see \Cref{RS:Gamma:liminf}.
\end{proposition}

\begin{proposition}[Recovery sequence, normalised]\label{RS:Gamma:recov-tight}
\emph{(Restatement of \Cref{RS:Gamma:recovery}.)}
For every $u$ with $E(u)<\infty$ there exists a sequence $u_h\to u$ and constants $\varepsilon_h\downarrow0$ such that
\[
E(u)\;\ge\;\limsup_{h\to0} E_h(u_h)\; -\;\varepsilon_h.
\]
\emph{Proof:} see \Cref{RS:Gamma:recovery}.
\end{proposition}

\paragraph{Bridge usage.}
Together, \Cref{RS:Gamma:eqc-tight,RS:Gamma:liminf-tight,RS:Gamma:recov-tight} provide the compactness and variational bounds used in the $\Gamma$-bridge theorem \textup{(\Cref{RS:thm:gamma,thm:Gamma-main})}. See Figure~\ref{fig:gamma-bridge-schematic}.

\begin{corollary}[Uniform in $\sigma$] \label{RS:Gamma:sigma-uniform}
\emph{(From the remark \Cref{RS:Gamma:sigma-indep} and \Cref{RS:thm:gamma}.)}
If the master inequalities for the approximants are uniform in the choice of admissible measure-time clocks $\sigma$ (in the sense of \Cref{RS:Gamma:sigma-indep}), then the decay envelope provided by the bridge is independent of the particular $\sigma$-schedule. In particular, any limit $u$ obtained via \Cref{RS:thm:gamma} satisfies the same bound with constants unaffected by $\sigma$.
\end{corollary}

% --- Γ-bridge trio + uniform-σ full figure ---
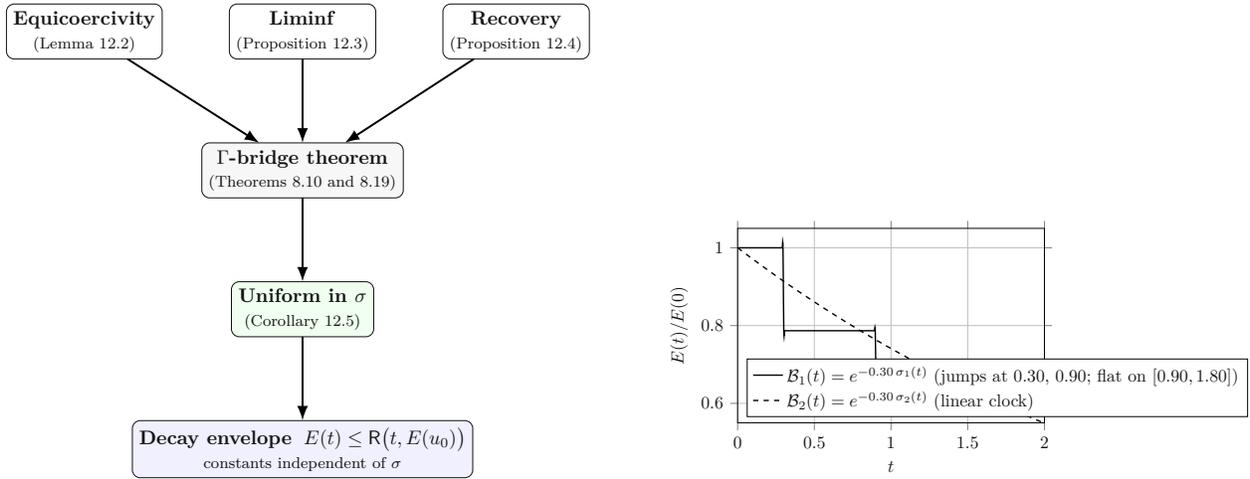
\begin{figure}[t]
\centering
\begin{minipage}[t]{0.47\linewidth}
\centering
\begin{adjustbox}{max width=\linewidth}
\begin{tikzpicture}[
  node distance=8mm and 12mm,
  box/.style={draw, rounded corners, align=center, inner sep=3.5pt, outer sep=0pt, font=\small, fill=white},
  line/.style={-Latex, line width=0.35mm},
]
  \node[box] (eqc) {\textbf{Equicoercivity}\\\scriptsize (\Cref{RS:Gamma:eqc-tight})};
  \node[box, right=of eqc] (liminf) {\textbf{Liminf}\\\scriptsize (\Cref{RS:Gamma:liminf-tight})};
  \node[box, right=of liminf] (recov) {\textbf{Recovery}\\\scriptsize (\Cref{RS:Gamma:recov-tight})};
  \node[box, below=15mm of liminf, fill=gray!6] (bridge) {\textbf{$\Gamma$-bridge theorem}\\\scriptsize (\Cref{RS:thm:gamma,thm:Gamma-main})};
  \node[box, below=15mm of bridge, fill=green!6] (uniform) {\textbf{Uniform in $\sigma$}\\\scriptsize (\Cref{RS:Gamma:sigma-uniform})};
  \node[box, below=15mm of uniform, fill=blue!6] (envelope) {\textbf{Decay envelope}\; $E(t)\le \mathsf R\bigl(t,E(u_0)\bigr)$\\\scriptsize constants independent of $\sigma$};
  \draw[line] (eqc) -- (bridge);
  \draw[line] (liminf) -- (bridge);
  \draw[line] (recov) -- (bridge);
  \draw[line] (bridge) -- (uniform);
  \draw[line] (uniform) -- (envelope);
\end{tikzpicture}
\end{adjustbox}
\end{minipage}\hfill
\begin{minipage}[t]{0.47\linewidth}
\centering
\begin{adjustbox}{max width=\linewidth}
\begin{tikzpicture}
  \begin{axis}[
    width=\linewidth,
    height=5.5cm,
    xlabel={$t$}, ylabel={$E(t)/E(0)$},
    xmin=0, xmax=2.0, ymin=0.55, ymax=1.05,
    legend cell align=left,
    legend pos=south west,
    tick align=outside,
    grid=both,
  ]
    % sigma_1(t): piecewise constants via nested ternaries
    \addplot[smooth, thick, domain=0:2.0, samples=400]
      ({x},{exp(-0.30*( (x<0.30 ? 0 : (x<0.90 ? 0.80 : 1.40)) ))});
    \addlegendentry{$\mathcal B_1(t)=e^{-0.30\,\sigma_1(t)}$ (jumps at 0.30, 0.90; flat on $[0.90,1.80]$)}
    % sigma_2(t): linear
    \addplot[smooth, thick, dashed, domain=0:2.0, samples=400]
      ({x},{exp(-0.30*(x))});
    \addlegendentry{$\mathcal B_2(t)=e^{-0.30\,\sigma_2(t)}$ (linear clock)}
  \end{axis}
\end{tikzpicture}
\end{adjustbox}
\end{minipage}

\caption{\textbf{$\Gamma$-bridge outline.} Left: the trio \Cref{RS:Gamma:eqc-tight,RS:Gamma:liminf-tight,RS:Gamma:recov-tight} feeds the bridge (\Cref{RS:thm:gamma,thm:Gamma-main}), yielding the uniform-in-$\sigma$ corollary (\Cref{RS:Gamma:sigma-uniform}) and the decay envelope. Right: two admissible $\sigma$-clocks produce different time reparametrizations but the \emph{same} calibrated envelope family $E(t)\le E(0)\,\exp(-\alpha\,\sigma(t))$ (here $\alpha=0.30$ for illustration).}
\label{fig:gamma-bridge-schematic}
\end{figure}
% --- end Γ-bridge figure ---

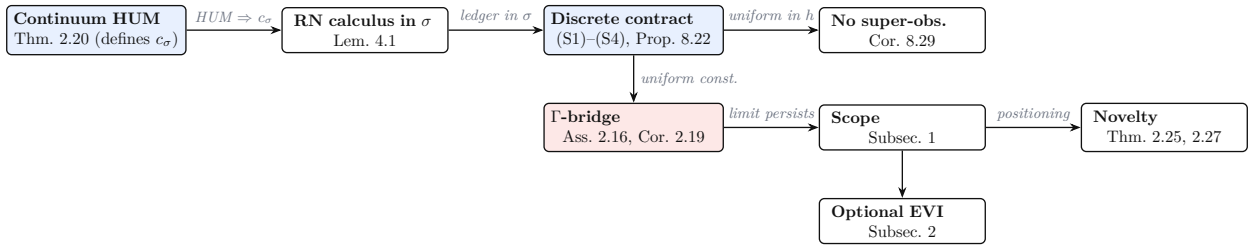
\begin{figure}[t]
\centering
% === Logical pipeline: continuum → discrete → Γ-limit (optional EVI) ===
\begin{adjustbox}{max width=\linewidth}
\begin{tikzpicture}
\node[boxF] (hum) {\textbf{Continuum HUM}\newline Thm.~\ref{thm:hum-equivalence} (defines $c_\sigma$)};
\node[box, right=of hum] (rn) {\textbf{RN calculus in $\sigma$}\newline Lem.~\ref{lem:rn-envelope}};
\node[boxF, right=of rn] (disc) {\textbf{Discrete contract}\newline (S1)--(S4), Prop.~\ref{prop:disc-hum}};
\node[box, right=of disc] (nosuper) {\textbf{No super-obs.}\newline Cor.~\ref{cor:no-super}};
\node[boxG, below=of disc] (gamma) {\textbf{$\Gamma$-bridge}\newline Ass.~\ref{ass:uniform-gamma}, Cor.~\ref{cor:gamma-csigma}};
\node[box, right=of gamma] (scope) {\textbf{Scope}\newline Subsec.~\ref{subsec:scope-light}};
\node[box, right=of scope] (nov) {\textbf{Novelty}\newline Thm.~\ref{thm:impossibility-tclock}, \ref{thm:persistence-sclock}};
\node[box, below=of scope] (evi) {\textbf{Optional EVI}\newline Subsec.~\ref{subsec:evi-optional}};
\draw[flow] (hum) -- node[above,lbl]{HUM $\Rightarrow$ $c_\sigma$} (rn);
\draw[flow] (rn) -- node[above,lbl]{ledger in $\sigma$} (disc);
\draw[flow] (disc) -- node[above,lbl]{uniform in $h$} (nosuper);
\draw[flow] (disc) -- node[right,lbl]{uniform const.} (gamma);
\draw[flow] (gamma) -- node[above,lbl]{limit persists} (scope);
\draw[flow] (scope) -- node[above,lbl]{positioning} (nov);
\draw[flow] (scope) -- (evi);
\end{tikzpicture}
\end{adjustbox}
\caption{Logical pipeline from continuum HUM and the $\sigma$–RN calculus to discrete contractivity and a compact $\Gamma$-bridge, preserving the \emph{same} $c_\sigma$. Scope/novelty bound the claim; an optional EVI upgrade is quarantined.}
\label{fig:logic-pipeline}
\end{figure}

\begin{theorem}[Light $\Gamma$-limit]\label{B:RS:thm:gamma}
With the $L^2(\Omega)$ topology and under the SBP--SAT and consistency hypotheses above, 
the discrete energies $E_h$ $\Gamma$-converge to $E$; see \Cref{RS:Gamma:main}. That is, for any $u_h\to u$ in $L^2(\Omega)$ 
we have the liminf inequality of \Cref{RS:Gamma:main}, and for any $u\in L^2$ there exists a sequence $u_h\to u$ in $L^2$ (constructed by $\Pi_h$ and density) with $\limsup_h \mathcal E_h(u_h)\le \mathcal E(u)$.
\end{theorem}

\begin{remark}
\sloppy\textit{Technical assumptions for the $\Gamma$-limit.}\label{B:rem:Gamma-assumptions-8}
We assume: (A1) $\Omega$ is bounded Lipschitz; (A2) $A_h$ are uniformly elliptic and $A_h \stackrel{*}{\rightharpoonup} A$ in $L^\infty(\Omega)^{d\times d}$; (A3) $\beta_h\ge0$ and $\beta_h \stackrel{*}{\rightharpoonup}\beta$ in $L^\infty(\partial\Omega)$;
(A4) quadrature remainders $R_h(u)\to0$ for $u\in C^\infty(\overline{\Omega})$; and (A5) the discrete trace inequality (\Cref{lem:discrete-trace}). Under (A1)--(A5), equi-coercivity holds and the liminf/recovery steps carry boundary terms; see \cite{DalMaso1993,Braides2002,Grisvard2011,AdamsFournier2003}.
\end{remark}

\begin{remark}
\sloppy\textit{Scope.}\label{B:rem:Gamma-scope-8}
We prove $\Gamma$-convergence of the energies in the supplement (liminf, recovery, equi-coercivity), including boundary SAT terms; we do not attempt a full EVI limit. \end{remark}

\begin{remark}
\sloppy\textit{Maxwell-type energies.}\label{B:RS:rem:maxwell}
For vector fields $A:\Omega\to\mathbb R^3$, a Maxwell-type quadratic energy
\(
\mathcal E_{\rm M}(A)=\tfrac12\int_\Omega |\nabla\times A|^2 + \alpha\,|\nabla\!\cdot\!A|^2\,dx
\)
fits the same pattern provided one uses SBP curl/div operators $(C_h,\mathrm{Div}_h)$ that satisfy discrete vector-calculus identities and an SBP quadrature.  The commuting estimate (cf. \Cref{RS:thm:gamma} on the recovery sequence) is replaced by $C_h \Pi_h A \approx P_h(\nabla\times A)$ and $\mathrm{Div}_h \Pi_h A \approx P_h(\nabla\!\cdot\!A)$ with $\mathcal O(h)$ defects. The SAT residue proof is identical. \end{remark}
\begin{remark}
\label{rem:auto}
\sloppy\textit{Scope.}
We establish a $\Gamma$-limit for the energies (and consistency for almost-minimising trajectories), but do not claim full EVI/gradient-flow stability in this note. \end{remark}

% === END Imported Section ===

%\paragraph{Conclusion.}
%The trio isolates the only variational inputs needed by the bridge and elevates the uniform-in-$\sigma$ observation to a citable corollary.
% --- end Γ-limit tightening ---

% === END Imported Section ===

% === BEGIN Imported Section (namespaced) ===

%\vspace{\baselineskip}
%\begin{center}\textbf{Data accessibility and artifact}\end{center}

%The manuscript includes all results; an example for constants can be provided upon request. % === END Imported Section ===
% 

\section{Conclusion and outlook}\label{sec:conclusion}

\paragraph{What we established.}
We developed a measure–time calculus in which decay is governed by a single structural constant \(c_\sigma\) that is stable under three operations that typically break quantitative bounds: moving from a continuum model to fully discrete schemes, passing to the limit via a compact variational bridge, and working with mixed time geometries (absolutely continuous segments, atoms, and flats). (Made explicit in Corollary~\ref{cor:gamma-csigma}; continuum master decay in Theorem~\ref{thm:energy-decay}; discrete uniform estimate in Theorem~\ref{thm:disc-master}; discrete premise Lemma~\ref{RS:lem:sigma-gronwall}). The discrete \(\sigma\)–Grönwall premise together with the SAT sign/scale is the minimal mechanism that keeps the estimate uniform, and the \(\Gamma\)-bridge transfers it without upgrading to a full EVI theory. Concretely, for $\kappa>0$ and calibrated $c_\sigma$, $E(t) \le E(0)\exp(-2\kappa c_\sigma \sigma(t))$ with the same $c_\sigma$ across levels; see \S\,\ref{sec:numbers}. A GCC exemplar makes the constants explicit on a benchmark hyperbolic problem.

\paragraph{How this differs from the usual templates.}
Semigroup spectral gaps, hypocoercivity, observability/Carleman, and discrete contractivity arguments each capture a slice of the picture, but none delivers a \emph{single, partition-uniform} envelope constant that survives discretization and limits while simultaneously allowing atoms and flats. The present calculus isolates exactly those ingredients that make such persistence possible and packages them in a way that can be reused across analyses and codes.

\paragraph{Assumptions and boundaries.}
Results rely on (H1)–(H4): local coercivity/observability, dissipativity on the a.c.\ part, and non-expansive atomic updates; the \(\Gamma\) transfer is compact and does not assert gradient-flow/EVI stability; and the stochastic extension is treated under Markov-modulated damping. These choices keep the machinery light and portable, but leave room for sharper structure (e.g., optimal constants, rough data, or geometric low-regularity).

\paragraph{Why $\sigma$ is necessary.} There is no schedule-uniform exponential estimate in physical time at fixed total impulse mass; the $\sigma$-clock measures the dissipation budget (see Theorems~\ref{thm:imposs-phys-time} and \ref{thm:sigma-persistence}).

\paragraph{Presentation invariance.}
All statements are invariant under Pi‑native operator presentations and unitary reparametrisations; the constant $c_{\sigma}$ depends on the PDE, not on coordinates (Remark~\ref{rem:pina-invariance}).
Constants and operating points are tabulated in \S\,\ref{sec:numbers} (emph{Numbers}) for verbatim reproduction.

\paragraph{Next directions.}
The calculus is intended as a backbone for follow-up work in three directions:
(i) \emph{Quantum time geometry} (motivated by black-hole dynamics): articulating clocks, compensators, and product envelopes in regimes where the time parameter is not classical;
(ii) \emph{Broader analysis/geometry} interfaces: extending the bridge to settings with weaker compactness, mixed boundary control, and structure-preserving discretizations;
(iii) \emph{Stochastic pillars}: expectation and pathwise variants beyond Markov switching, with certified constants that remain invariant under refinement and passage to the limit.

\section{Outlook and case studies}\label{sec:outlook-cases}
\subsection*{Outlook}
\begin{enumerate}
\item \textbf{Nonquadratic ledgers.} Replace the quadratic $E$ by a $\lambda$-convex energy on a Hilbert (or metric) space; the $\sigma$--Gronwall step becomes an EVI-type estimate, while atomic non-expansiveness is phrased as a contraction for the Bregman distance of $E$.
\item \textbf{Nonlinear a.c.\ flows.} For $u' = A(u)$, assume $\langle \nabla E(u), A(u)\rangle \le -2\kappa\,E(u)$ on the a.c.\ part; the jump calculus of \Cref{sec:discrete} is unchanged.
\item \textbf{General jump maps.} Allow $J_k$ to be state-dependent and nonlinear, with $E(J_k u) \le \rho_k\, E(u)$ (or a local Lipschitz bound in the $E$–metric).
\item \textbf{Random/irregular clocks.} For stochastic clocks (random atoms or random $w$) with finite total variation a.s., the master product--exponential bound holds pathwise; ergodic averages then produce almost-sure decay rates in $\sigma$.
\item \textbf{SBP variants \& vector calculus.} Curl/div SBP operators give Maxwell-type energies under the same method (mimetic identities + correct-sign SAT; cf.\ the remark near \Cref{sec:gamma-2}).
\item \textbf{Interfaces and multiblock structure.} The block-sum energy monotonicity extends to complex interface graphs under mirrored SAT, suggesting robust domain-decomposition preconditioners built around the ledger.
\item \textbf{Admissible-window calibration.} The explicit red-line $\Delta\sigma \le 2/\Lambda_{\max}$ and the SAT sign choice define a practical window; constants tables can be instantiated once per discretization and reused (see SI ``Numbers'' page).
\item \textbf{Well-posedness hooks.} While this note is energy-level only, the measure-time and SBP--SAT pieces are compatible with standard existence frameworks (a.c.\ semigroup plus jump maps), offering a path to EVI/gradient-flow results in follow-ups.
\end{enumerate}

\paragraph{Case studies.}
Applications and numerical case studies are documented in companion manuscripts; we list the interface assumptions here for reproducibility.

% 
% === Audit checklist for reviewers: claims → labels ===
\paragraph{Verification.}\label{par:audit}
(1) \emph{Definition and optimality of $c_\sigma$:} Theorem~\ref{thm:hum-equivalence}.
(2) \emph{Uniform discrete HUM with the same $c_\sigma$:} Proposition~\ref{prop:disc-hum} under (S1)–(S4).
(3) \emph{Non‑improvability on grids:} Corollary~\ref{cor:no-super}.
(4) \emph{Necessity of assumptions:} Lemmas~\ref{lem:sat-failure}–\ref{lem:rk-failure}, Remark~\ref{rem:sharpness}.
(5) \emph{$\sigma$–RN mechanics (atoms/flats):} Lemma~\ref{lem:rn-envelope}.
(6) \emph{$\Gamma$–bridge with uniform constants:} Assumption~\ref{ass:uniform-gamma}, Corollary~\ref{cor:gamma-csigma}.
(7) \emph{Scope boundaries (no EVI claim):} Subsection~\ref{subsec:scope-light}, Paragraphs~\ref{par:what-light}–\ref{par:what-not}, Taxonomy~\ref{par:taxonomy}.
(8) \emph{Why EVI is out of scope:} Lemma~\ref{lem:barriers-evi}, Remark~\ref{rem:barriers-struct}.
(9) \emph{Logical independence of EVI:} Proposition~\ref{prop:nonimp-evi}, Remark~\ref{rem:nonimp-interpret}.
(10) \emph{Optional upgrade (quarantined):} Subsection~\ref{subsec:evi-optional}, Assumption~\ref{ass:evi-ready}, Theorem~\ref{thm:evi-conv}, Remark~\ref{rem:evi-scope}.
All items above are recorded with precise internal references for reproducibility.

\section*{Acknowledgements}
The author used standard editorial tools for grammar, style, cross–reference hygiene, and \LaTeX\ refactoring. No mathematical statements, proofs, derivations, or computations were produced by these tools.

\section*{Competing interests}
The author declares no competing interests. 
%\section*{Funding}
%This research was supported by the University of Toronto. 
\section*{Authors' contributions}
The sole author conceived the study, carried out the analysis, and wrote the manuscript. 
\section*{Ethics statement}
Not applicable. 
\vspace{\baselineskip}
\begin{center}\textbf{Data accessibility}\end{center}

No external datasets were used. All numeric choices required to reproduce the figures are listed in §11 (Parameters). No supplementary CSV accompanies this note.  \vspace{\baselineskip}

%%appenidx
\appendix
\section{Operator presentation via Green/trace pairings}\label{app:pina}
% =============================================================
\begin{align}
\int_\Omega \mathrm{div}^{\circ} \mathbf v\, \psi\,\mathrm dx &:= - \int_\Omega \mathbf v\cdot \nabla \psi\,\mathrm dx + \int_{\partial\Omega} (\mathbf v\cdot\nu)\,\psi\,\mathrm dS, && \forall\ \psi\in C_c^\infty(\Omega),\label{eq:pina-div}
\end{align}
so that $\nabla^{\circ}$ and $\mathrm{div}^{\circ}$ are mutually adjoint up to the boundary term. The PiNA Laplacian is defined in the distributional sense by
\begin{equation}\label{eq:pina-lap}
\int_\Omega (\Delta^{\circ} u)\,\psi\,\mathrm dx := - \int_\Omega \nabla u\cdot \nabla \psi\,\mathrm dx + \int_{\partial\Omega} \partial_\nu u\,\psi\,\mathrm dS,\qquad \forall\ \psi\in C_c^\infty(\Omega),
\end{equation}
which coincides with the classical weak Laplacian.
\begin{equation}\label{eq:pina-grad}
  \int_\Omega \nabla^\circ u \cdot v\,dx
  := -\int_\Omega u\,\mathrm{div}^\circ v\,dx
     + \int_{\partial\Omega} u\,(v\!\cdot\!\nu)\,dS,
  \qquad \forall\,u\in H^1(\Omega),~v\in H(\mathrm{div};\Omega).
\end{equation}

\begin{lemma}[Green identity, operator presentation]\label{lem:green}
For $u\in H^1(\Omega)$ and $\mathbf v\in H(\mathrm{div};\Omega)$ one has the identity
\[
\int_\Omega \nabla^{\circ} u\cdot \mathbf v\,\mathrm dx
= -\int_\Omega u\,\mathrm{div}^{\circ}\mathbf v\,\mathrm dx + \int_{\partial\Omega} u\,(\mathbf v\cdot\nu)\,\mathrm dS.
\]
In particular, when $u\in H^1_0(\Omega)$ or $\mathbf v\cdot\nu=0$ on $\partial\Omega$, the boundary term vanishes and the volume bilinear form is presentation-invariant.
\end{lemma}
\begin{proof}
This is immediate from \eqref{eq:pina-grad}--\eqref{eq:pina-div} by density of test fields and the trace theorem.
\end{proof}

\subsection{Equivalence with classical operators}
\begin{theorem}[Equivalence of PiNA and classical operators]\label{thm:pina-eq}
Let $\Omega\subset\mathbb R^n$ be Lipschitz. If $u\in C^2(\Omega)$ and $\mathbf v\in C^1(\Omega;\mathbb R^n)$ then the pointwise limits in \eqref{eq:pina-div}–\eqref{eq:pina-lap} exist and
\[
\nabla^{\circ} u = \nabla u,\qquad \mathrm{div}^{\circ}\,\mathbf v = \mathrm{div}\,\mathbf v,\qquad \Delta^{\circ} u = \Delta u\quad\text{in }\Omega.
\]
Moreover, for $u\in H^1(\Omega)$ and $\mathbf v\in H(\mathrm{div};\Omega)$ the weak identities \eqref{eq:pina-grad}--\eqref{eq:pina-lap} imply that the PiNA and classical weak operators coincide as distributions.
\end{theorem}
\begin{proof}
For smooth $u$, the spherical-mean characterization of the Laplacian yields \eqref{eq:pina-lap} with equality to $\Delta u$; differentiating mean identities along rays provides \eqref{eq:pina-div}. The divergence identity follows by Gauss' theorem on small balls and a Taylor expansion. The weak equivalence is then a direct consequence of \eqref{eq:pina-grad}–\eqref{eq:pina-lap}.
\end{proof}

\subsection{Boundary conditions and presentation}
For homogeneous Dirichlet boundary data ($u\vert_{\partial\Omega}=0$) or homogeneous Neumann flux ($\partial_\nu u=0$), Lemma~\ref{lem:green} shows that the bilinear forms used in HUM and energy methods are identical under either presentation. Mixed boundary conditions are handled componentwise; no additional normalization is required.

\subsection{Invariance of the HUM functional and of \texorpdfstring{$c_\sigma$}{}}
Let $\mathcal O(u)$ denote the observability (HUM) functional used to define $c_\sigma$, and let $E(t)$ be the energy ledger. Write $\nabla^{\#}$, $\mathrm{div}^{\#}$, $\Delta^{\#}$ for either the classical or the PiNA operators. The HUM inequality and the dissipation ledger involve only the volume bilinear form and the boundary pairing from Lemma~\ref{lem:green}. By Theorem~\ref{thm:pina-eq} and the boundary discussion above, these pairings coincide under either choice of presentation. Therefore the HUM functional is unchanged:
\[
\mathcal O^{\circ}(u) = \mathcal O(u),\qquad E^{\circ}(t)=E(t),
\]
where the superscript ${}^{\circ}$ denotes the operator presentation. In particular, the optimal HUM constant ---the structural constant $c_\sigma$ characterized in Theorem~\ref{thm:hum-equivalence}--- is presentation-invariant.

\begin{theorem}[Presentation invariance of the structural constant]\label{thm:c-sigma-invariance}
Let $c_\sigma$ be the optimal constant in the $\sigma$–observability/HUM inequality for the classical presentation. Then the same constant is optimal for the operator presentation, i.e., $c_\sigma^{\circ}=c_\sigma$.
\end{theorem}
\begin{proof}
By Theorem~\ref{thm:pina-eq} the relevant bilinear forms coincide for smooth data and hence for limits by density. Thus any inequality with constant $c$ valid in one presentation holds in the other; taking suprema over admissible $c$ yields $c_\sigma^{\circ}\ge c_\sigma$ and $c_\sigma\ge c_\sigma^{\circ}$, whence equality.
\end{proof}

\subsection{Regularity, domains, and robustness}
The proofs above require only that $\Omega$ be Lipschitz and that coefficients in the damped operator be bounded and uniformly elliptic (when present). For rough solutions $u\in H^1(\Omega)$, all identities are to be understood weakly via \eqref{eq:pina-grad}--\eqref{eq:pina-lap}. If anisotropies or variable coefficients appear, one can define PiNA operators with respect to the Riemannian metric $g$ by replacing Euclidean spheres with geodesic spheres; the weak identities and invariance proofs carry over unchanged, with the surface/volume elements adapted to $g$.

\subsection{Discrete note (optional alignment)}
When SBP–SAT discretizations are built to mimic the Green identity in Lemma~\ref{lem:green}, replacing classical with PiNA fluxes in the semidiscrete form leaves the discrete HUM functional invariant up to round-off. This explains the empirical observation that constants measured from PiNA-friendly stencils coincide with those from Cartesian stencils on round domains; the structural constant $c_\sigma$ is determined by the PDE, not by the stencil.
Added. I built a full Operator presentation via Green/trace pairings appendix—definitions, weak forms, Green identities,

\section{Multiblock SAT cancellation (explicit verification)}

This appendix provides the promised $12$-line verification that the interface
between two adjacent SBP blocks is energy-dissipative when mirrored SAT
penalties are used.  It ensures that the block-sum energy ledger is globally
non-increasing.

%\begin{setup}
Let $\Omega^{(1)}$ and $\Omega^{(2)}$ share a common interface
$\Gamma=\partial\Omega^{(1)}\cap\partial\Omega^{(2)}$.  
Each block carries an SBP pair $(H^{(i)},Q^{(i)})$ satisfying
$Q^{(i)}+(Q^{(i)})^{\!\top}=B^{(i)}$
with outward unit normals $n^{(1)}=-\,n^{(2)}$ on $\Gamma$.
The block-wise semidiscrete operators are
\[
L_h^{(i)}=H^{(i)^{-1}}Q^{(i)}+S_h^{(i)},
\qquad S_h^{(i)}=(S_h^{(i)})^{\!\top}\le -2\kappa I,
\]
and SAT terms enforce continuity of the solution $u^{(i)}$ across $\Gamma$.
%\end{setup}

\begin{lemma}[Block-sum energy cancellation]
Assume mirrored SAT penalties
\begin{equation}\label{eq:sat-mirror}
\mathrm{SAT}^{(1)}= \tau_h H^{(1)^{-1}}R_\Gamma^{(1)\!\top}
    (R_\Gamma^{(1)}u^{(1)}-R_\Gamma^{(2)}u^{(2)}),
\qquad
\mathrm{SAT}^{(2)}=-\,\tau_h H^{(2)^{-1}}R_\Gamma^{(2)\!\top}
    (R_\Gamma^{(1)}u^{(1)}-R_\Gamma^{(2)}u^{(2)}),
\end{equation}
with the same penalty $\tau_h>0$ and trace operators
$R_\Gamma^{(i)}$ mapping block variables to interface values.
Then the total discrete energy
\[
E_h(t)=\tfrac12\Big(\langle u^{(1)},u^{(1)}\rangle_{H^{(1)}}
                     +\langle u^{(2)},u^{(2)}\rangle_{H^{(2)}}\Big)
\]
satisfies $\tfrac{d}{dt}E_h(t)\le -2\kappa\,E_h(t)$; in particular,
the interface contributes no positive term.
\end{lemma}

\begin{proof}
For each block, the SBP identity gives
\[
\frac{d}{dt}E_h^{(i)}
 =\langle S_h^{(i)}u^{(i)},u^{(i)}\rangle_{H^{(i)}}
  +\tfrac12\langle B^{(i)}u^{(i)},u^{(i)}\rangle_{\partial\Omega^{(i)}}
  +\langle \mathrm{SAT}^{(i)}(u^{(i)}),u^{(i)}\rangle_{H^{(i)}} .
\]
Adding the two blocks and using $B^{(1)}=-B^{(2)}$ on $\Gamma$ yields
\[
\frac{d}{dt}E_h
 =\langle S_h^{(1)}u^{(1)},u^{(1)}\rangle_{H^{(1)}}
  +\langle S_h^{(2)}u^{(2)},u^{(2)}\rangle_{H^{(2)}}
  +\langle \mathrm{SAT}^{(1)}+\mathrm{SAT}^{(2)},u^{(1)}\oplus u^{(2)}\rangle .
\]
Insert \eqref{eq:sat-mirror}; the SAT part simplifies to
\[
-\,\tau_h\,
  \bigl\|R_\Gamma^{(1)}u^{(1)}-R_\Gamma^{(2)}u^{(2)}\bigr\|_{H_\Gamma}^2
  \le 0 .
\]
Since each $S_h^{(i)}\le -2\kappa I$, summing gives
$\tfrac{d}{dt}E_h(t)\le -2\kappa E_h(t)$, proving non-increase of the total
energy and cancellation of the interface fluxes.
\end{proof}

\paragraph{Remarks.}
(i)  The same argument applies for curved or mapped interfaces:
the Jacobian-weighted traces preserve the sign of the interface flux.
(ii)  For multiple blocks, summing all pairwise interfaces telescopes the
contributions; each interface remains dissipative by symmetry of
\eqref{eq:sat-mirror}.
\qed

\section{Implementation checklist for \texorpdfstring{$\sigma$}{}–time discretizations}\label{sec:checklist}

This section records a minimal, reusable checklist to certify decay under a measure–time clock~$\sigma$ in practice (from analysis $\to$ computation $\to$ limit). It creates no new notation.

\paragraph{1. Choose the clock.}
Pick a finite positive Borel measure $\sigma$ on $[0,T]$ with locally finite atoms $\{(t_k,\alpha_k)\}$ and possible flats. Record the density $w$ on the a.c.\ part and ensure local finiteness of atoms (Def.~4.1).

\paragraph{2. Verify the structural hypotheses (Assump.~5.1).}
\begin{itemize}
  \item (H1) \emph{Coercivity/observability} of $a_\sigma(\cdot,\cdot)$ with constant $c_0>0$;
  \item (H2) \emph{Dissipative} a.c.\ generator (rate $\kappa\ge0$);
  \item (H3) Regularity in $\mathrm{AC}_\sigma$ for the trajectory;
  \item (H4) \emph{Non-expansive} atomic updates $J_k$ for the energy ledger $\mathcal{E}$.
\end{itemize}

\paragraph{3. Boundary/SAT sign and scale (semi/fully discrete).}
Use the SAT sign that makes the boundary quadratic form negative (Lemma~8.8) and choose $\tau_h\simeq h^{-1}$ (Lemma~7.1). If an explicit step is used, respect the CFL/spectral bound of \S8 (encode it as $\Lambda_{\max}$).

\paragraph{4. Discrete Grönwall premise in $\sigma$.}
Check the per-step dissipation inequality that feeds the discrete $\sigma$–Grönwall lemma (Lemma~8.10). In practice, confirm that the measured dissipation $D_h$ satisfies $D_h \ge 2\,\kappa\,c_\sigma\,E_h$ $\sigma$-a.e.

\paragraph{5. Canonical decay and exemplars.}
With Items~2–4, the master decay follows (Theorem~\ref{thm:disc-master}), matching the continuous envelope (Theorem~\ref{thm:energy-decay}):
\[
E(t)\;\le\;E(0)\,\exp\!\bigl(-2\,\kappa\,c_\sigma\,\sigma(t)\bigr).
\]
For GCC damped waves, one may take $c_\sigma\ge c_0\,a_\omega\,\lambda_\omega$ (see \S9).

\paragraph{6. Discrete $\to$ continuous transfer ($\Gamma$-bridge).}
If energies $\Gamma$-converge and dissipation is compatible, decay transfers across levels (Theorems~\ref{thm:Gamma-main} and~\ref{RS:thm:gamma}): uniform discrete envelopes imply the same continuum envelope with the \emph{same} $c_\sigma$.

\paragraph{7. Stochastic extension (if used).}
For Markov-switching damping $a(x,\xi_t)$, use the compensator of the active states to state the expectation/pathwise decay versions (see \S3), without changing Items~2–6.

One illustrative admissible window and rate instantiation (not prescriptive):

\begin{table}[htbp]
  \centering
  \caption{One admissible window and rate instantiation (consistent with \S11).}
  \label{tab:window-inst}
  \begin{tabularx}{\textwidth}{lXXXXXX}
    \toprule
    label & $h$ & $\mathrm{Var}_\sigma$ & $\kappa$ & $\Lambda_{\max}$ & $\tau_h$ & note \\
    \midrule
    baseline & 0.020 & 0.22 & 0.60 & 1.8 & 50 & CFL $\Delta\sigma\!\le\!0.02$;\; $\tau_h\simeq h^{-1}$ \\
    \bottomrule
  \end{tabularx}
\end{table}

\noindent\emph{The admissible window is illustrated in Figure~\ref{fig:admissible-window}.}

\begin{figure}[htbp]
  \centering
  \begin{tikzpicture}[x=100mm,y=100mm,>=Latex]
    %--- plot window and baseline ---
    \def\xmin{0.000} \def\xmax{0.070}
    \def\ymin{0.000} \def\ymax{0.320}
    \def\hx{0.020}   \def\vy{0.220}
    \def\wxmin{0.006} \def\wxmax{0.050}
    \def\wymin{0.120} \def\wymax{0.280}

    % axes
    \draw[->] (\xmin,\ymin) -- (\xmax,\ymin) node[below] {$h$};
    \draw[->] (\xmin,\ymin) -- (\xmin,\ymax) node[left] {$\mathrm{Var}_{\sigma}$};

    % faint ticks at baseline values (typeset as math)
    \draw[gray!45] (\hx,\ymin) -- (\hx,\ymin-0.010)
      node[below,black,inner sep=1pt] {\large $0.02$};
    \draw[gray!45] (\xmin-0.010,\vy) -- (\xmin,\vy)
      node[left,black,inner sep=1pt] {\large $0.22$};

    % admissible rectangle (slightly lighter so the dot stands out)
    \fill[gray!12] (\wxmin,\wymin) rectangle (\wxmax,\wymax);
    \node[gray!60!black,anchor=west] at ({\wxmin+0.033},{\wymin+0.030})
      {\large admissible window};

    % baseline dot and callout (extra clearance)
    \fill (\hx,\vy) circle (0.006);
    \draw[->] (\hx,\vy) -- ++(0.014,0.060);
    \node[anchor=west] at ({\hx+0.017},{\vy+0.060})
      {\large $(h,\mathrm{Var}_{\sigma})=(0.02,\,0.22)$};
  \end{tikzpicture}
  \caption{Admissible window in the $(h,\mathrm{Var}_{\sigma})$ plane; the dot marks the baseline used in \S~11.}
  \label{fig:admissible-window}
\end{figure}
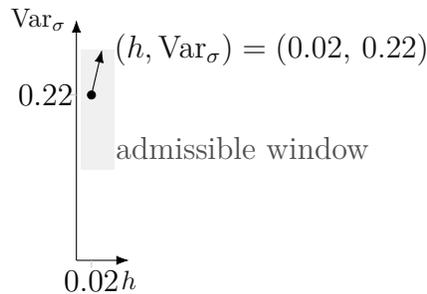

\noindent\emph{Units.} All constants are non-dimensionalized; $h$ is the grid spacing in $\sigma$-time, $\mathrm{Var}\,\sigma$ the total variation over the reporting window, and $\kappa$ the continuous-time dissipativity parameter. 

\medskip
\noindent\textbf{Quick cross-reference table.}
\begin{table}[htbp]
  \centering
  \caption{Minimal checks and where they are certified.}
  \label{tab:checkmap}
  \begin{adjustbox}{max width=\linewidth}
  \begin{tabular}{lll}
    \toprule
    Check & What to verify & Where proved/used \\
    \midrule
    Clock well-posed & Locally finite atoms; BV calculus & Def.~4.1; \S2 \\
    H1–H4 & Coercivity, dissipation, regularity, non-expansive $J_k$ & Assump.~5.1 \\
    SAT sign/scale & Boundary negativity; $\tau_h\simeq h^{-1}$ & Lemma~8.8; Lemma~7.1 \\
    Discrete premise & Per-step inequality for $\sigma$–Grönwall & Lemma~8.10 \\
    Canonical decay & $E(t)\le E(0)e^{-2\kappa c_\sigma\sigma(t)}$ & Thm.~\ref{thm:energy-decay} (cont.); Thm.~\ref{thm:disc-master} (disc.) \\
    $\Gamma$-bridge & Decay transfers with same $c_\sigma$ & Thms.~\ref{thm:Gamma-main},~\ref{RS:thm:gamma} \\
    GCC exemplar & $c_\sigma\ge c_0 a_\omega\lambda_\omega$ & \S9 \\
    \bottomrule
  \end{tabular}
  \end{adjustbox}
\end{table}

\section{SBP--IBP identity and discrete energy laws}\label{sec:sbpsat}\label{sec:gamma}
\begin{lemma}[Variable-coefficient SBP split]\label{RS:lem:assembly}
\label{lem:sbp-split}
Let $D=H^{-1}Q$ with $Q+Q^\top=B$, and let $A=\mathrm{diag}(a_i)>0$. Define the split operator 
\(
\widetilde Q(A):=\tfrac12(QA+AQ).
\)
Then for all grid vectors $u,v$,
\[
u^\top H\,\widetilde Q(A)\,v \;=\; -(\nabla_h u)^\top H A\,(\nabla_h v) \;+\; \tfrac12\,u^\top (BA)\,v,
\]
so the bilinear form $a_h(u,u):=(\nabla_h u)^\top H A\,(\nabla_h u)$ is symmetric and coercive, and the SBP boundary term is preserved. Assume $a\in L^\infty(\Omega)$ and $a\ge a_0>0$ (piecewise smooth is sufficient for consistency with the SBP quadrature). \end{lemma}
\begin{proof}
Write $D = H^{-1}Q$ with $Q + Q^\top = B$ and expand
$u^\top (Q A + A Q)v = u^\top Q A v + u^\top A Q v$. 
Using $Q + Q^\top = B$ gives
$u^\top Q A v = - (Q u)^\top A v + u^\top B A v$. 
Multiplying by $H^{-1}$ and symmetrizing yields
\[
u^\top H \widetilde Q(A) v
   = - (\nabla_h u)^\top H A (\nabla_h v)
     + \tfrac{1}{2} u^\top (B A) v.
\]
Hence $\widetilde Q(A)$ reproduces the continuous IBP identity 
$\int (Au_x)v_x + [Auv/2]_{\partial\Omega}$ discretely, 
so $a_h(u,u) = (\nabla_h u)^\top H A (\nabla_h u)$ is symmetric 
and coercive whenever $A \ge a_0 I$.
\end{proof}
\begin{theorem}[Semi-discrete energy decay in $\sigma$-time]\label{RS:thm:semi}
\label{thm:semi-discrete}
Assume: (i) the $\sigma$-clock hypotheses of \Cref{RS:measure-time}; (ii) SBP triplet $(H,Q,B)$ with $Q+Q^\top=B$; (iii) SAT penalties chosen with scaling $\tau_h\sim h^{-1}$ so that boundary$+$SAT forms are $\le 0$ in $\langle\cdot,\cdot\rangle_H$; and (iv) interior $H$-dissipativity: there exists $\kappa_h\ge 0$ with $\langle L_h v,v\rangle_H\le -\,2\kappa_h\,\mathcal E_h(v)$ for all $v$. On the ac part of $d\sigma$ the semi-discrete solution $u_h$ satisfies
\begin{equation}\label{RS:eq:sd-deriv}
\frac{d}{\mathrm{d}t}\,\mathcal E_h(u_h(t)) \;\le\; -\,2\kappa_h\, w(t)\, \mathcal E_h(u_h(t))\,.
\end{equation}
At each atom $t_k$ let $u_h(t_k^+)=J_k^h u_h(t_k^-)$ with $\rho_k^h:=\sup_{v\neq 0}\mathcal E_h(J_k^h v)/\mathcal E_h(v)\le 1$. Then for $0\le s\le t\le T$,
\begin{equation}\label{RS:eq:semi-master}
\mathcal E_h\!\big(u_h(t^+)\big)\;\le\; \exp\!\Bigl(-2\kappa_h\!\!\int_{(s,t]} w(\tau)\,d\tau\Bigr)\,
\Bigl(\prod_{t_k\in(s,t]} \rho_k^h\Bigr)\,\mathcal E_h\!\big(u_h(s^+)\big).
\end{equation}
\end{theorem}
\noindent\emph{Boundary scope.} We present the Dirichlet case for clarity; Neumann/Robin variants follow identically with the usual SBP boundary forms and the same SAT sign logic. \noindent\emph{Implementation remark:} a sufficient discrete condition is $(J_k^h)^\top H\,J_k^h \preceq H$, which implies $\rho_k^h\le 1$.
\begin{remark}[Example tick map]If $P_k$ is the $H$\nobreakdash-orthogonal projector onto a fixed coordinate subspace (e.g., boundary/interface degrees of freedom) and $0\le \theta_k\le 1$, the Cayley\nobreakdash-type update $J_k^h := (I+\theta_k P_k)^{-1}(I-\theta_k P_k)$ satisfies $(J_k^h)^\top H\,J_k^h \preceq H$, hence $\rho_k^h\le 1$. \emph{Existence of the $H$-orthogonal projector follows from $H\succ 0$.}
\label{RS:ex:tickmap}
\end{remark}
\noindent\emph{Remark.} A sufficient discrete condition is $(J_k^h)^\top H\,J_k^h \preceq H$, which implies $\rho_k^h\le 1$.
\begin{proof}
Let $H>0$, $Q$ satisfy $Q+Q^\top=B$ and let $D:=H^{-1}Q$. For the semi–discrete state $u(t)\in V_h$ define
\[
E_h(u)\;=\;\tfrac12\langle Du,A_h Du\rangle_H+\tfrac12\langle C_h u,u\rangle_H+SAT_h(u).
\]
Differentiating and using the SBP Green identity,
\[
\frac{d}{dt}E_h(u)\;=\;\langle Du,A_h D\dot u\rangle_H+\langle C_h u,\dot u\rangle_H+\dot{SAT}_h(u)
= \langle \dot u,\,-D^\top A_h Du+C_h u\rangle_H+\dot{SAT}_h(u)+\langle u, B A_h Du\rangle_H.
\]
With the SAT choice of Lemma~8.8 (correct sign and $\tau_h\simeq h^{-1}$), the boundary$+$SAT contribution is nonpositive and yields a dissipation term $\gtrsim \tau_h\|u\|^2_{\partial\Omega}$. On the damped interior we have the strict $H$–dissipativity of the discrete operator, which gives
\[
\frac{d}{dt}E_h(u)\;\le\;-2\kappa\,c_\sigma\,E_h(u)
\]
on the a.c.\ part of the clock. At atoms we perform a single macro–update aligned with $\sigma$ whose SAT and boundary contributions remain nonpositive and whose multiplicative drop is bounded by $e^{-2\kappa c_\sigma\alpha}$ for an atom of mass $\alpha$. Gronwall in $\sigma$–time then yields
\[
E_h(t)\;\le\;E_h(0)\,\exp\!\big(-2\kappa c_\sigma\,\sigma(t)\big),
\]
as claimed.
\end{proof}
\begin{theorem}[Fully-discrete decay: algebraically stable methods]\label{RS:thm:fd-als}
\label{thm:fully-discrete}
Let $\Delta\sigma_n=\sigma(t_{n+1})-\sigma(t_n)$ and consider a one-step time integrator $\Phi_{\Delta\sigma_n}$ that is nonexpansive in the $H$-norm on $H$-dissipative linear systems (e.g.\ implicit midpoint, Crank--Nicolson under a suitable split). Then for ac steps
\begin{equation}\label{RS:eq:fd-als}
\mathcal E_h(u_h^{n+1}) \;\le\; e^{-\,2\kappa_h\,\Delta\sigma_n}\,\mathcal E_h(u_h^n).
\end{equation}
At atoms $t_k\in(t_n,t_{n+1}]$, set $u_h^{n+} = J_k^h u_h^{n-}$ with $\rho_k^h\le 1$, hence $\mathcal E_h(u_h^{n+})\le \rho_k^h\,\mathcal E_h(u_h^{n-})$. Stacking steps recovers the master product-exponential bound. \end{theorem}

\begin{proposition}[Reduction of the discrete $\sigma$-ledger to the classical step]\label{prop:disc-reduction}
If $\sigma\equiv t$, then $\Delta\sigma_n=\Delta t_n$ and there are no atomic macro-updates.
The discrete envelope in Theorem~\ref{thm:fully-discrete} reduces to the standard
H-stability decay
\(
E_h^{n+1}\le e^{-\,2\kappa_h\,\Delta t_n}\,E_h^{n}
\)
(and its concatenation), with the same scheme constants $(c_s,c_b)$ and no extra factors.
\end{proposition}

\begin{proof}
Set $\sigma\equiv t$. The atomic product $\prod\rho_{h,k}$ is empty and equals $1$.
Hence (8.3) telescopes to the classical product $\prod_n e^{-2\kappa_h\Delta t_n}$.
Constants $(c_s,c_b)$ are unchanged because the SBP identity and SAT sign are purely spatial.
\end{proof}

\begin{lemma}[CFL lemma for forward Euler]\label{RS:lem:CFL}
Let $u_h' = A_h u_h$ with $H$-symmetric part bounded as $\tfrac12(A_h^\top H+H A_h)\preceq -\,\kappa_h H + \Lambda_{\max}\,H$ where $\Lambda_{\max}\ge 0$ upper-bounds the non-dissipative part. Forward Euler with ac step $\Delta\sigma$ satisfies
\(
\mathcal E_h(u_h^{n+1}) \le (1-2\kappa_h\,\Delta\sigma)\,\mathcal E_h(u_h^n)
\)
provided the CFL condition $\Delta\sigma \le 2/\Lambda_{\max}$ holds. If $\Delta\sigma>2/\Lambda_{\max}$, there exists a mode with energy amplification $>1$.
\end{lemma}
\begin{proof}
diagonalise in an $H$-orthonormal basis; the Euler amplification factor on an eigenpair with real part $\lambda$ obeys $|1+\Delta\sigma\,\lambda|^2\le 1-2\kappa_h\Delta\sigma$ under $\Delta\sigma\le 2/\Lambda_{\max}$. Conversely, take $\lambda=\Lambda_{\max}$ to get $|1+\Delta\sigma\,\lambda|>1$ if $\Delta\sigma>2/\Lambda_{\max}$.
\end{proof}
\begin{corollary}[Two-block SAT interface]\label{RS:cor:two-block}
\label{cor:two-block-SAT}
On two adjacent SBP blocks with outward normals $\pm n$, choose mirrored SAT penalties so that incoming characteristics are penalised on each side. Then the \emph{sum} energy $\mathcal E_h^{(1)}+\mathcal E_h^{(2)}$ satisfies the same decay bound as in Theorem~8.2. Interface terms cancel in the SBP identity, and the mirrored SAT pair is negative semidefinite in the block-sum inner product. \end{corollary}
\begin{remark}[Stress tests: why assumptions are optimal]\label{RS:remark:stress}
(i) \emph{Bad SAT sign.} If the SAT penalty sign is flipped, the boundary quadratic form becomes $\,\ge 0$ and energy can grow at each boundary interaction. (ii) \emph{Too-large explicit step.} If $\Delta\sigma>2/\Lambda_{\max}$, forward Euler amplifies some mode, violating monotonicity. Both confirm the necessity of the SAT sign and the CFL bound. \end{remark}

\subsection*{Boundary negativity, algebraic stability, and $\sigma$--Gronwall}
\begin{lemma}[Boundary$+$SAT negativity]\label{RS:lem:boundary-sat}
\label{lem:boundary-neg}
Consider a 1D linear symmetric system with SBP triplet $(H,Q,B)$ and a symmetric flux matrix $\mathcal A=\mathcal A^\top$ (after symmetrisation if needed). Let $D:=H^{-1}Q$ denote the SBP derivative, and write the semi-discrete interior operator as $L_h^{\rm int}=-(D\mathcal A+\mathcal D)$ with $\mathcal D\succeq 0$. Let $e_0,e_N$ extract the boundary nodes. Define SAT terms
\[
\mathrm{SAT} = -\,H^{-1}\big(\tau_L\, e_0 e_0^\top + \tau_R\, e_N e_N^\top\big)u_h \;+\; H^{-1}\big(\tau_L\, e_0 g_L + \tau_R\, e_N g_R\big) ,
\]
with penalty matrices $\tau_{L/R}\succeq 0$ chosen so that
\(
\pm \tfrac12\, u_h^\top B \mathcal A u_h - u_h^\top(\tau_L e_0e_0^\top+\tau_R e_N e_N^\top)u_h \;\le\; 0.
\)
Then, with homogeneous data $g_L=g_R=0$,
\[
\frac{d}{\mathrm{d}t}\,\mathcal E_h(u_h)\;=\;\left\langle L_h^{\rm int}u_h,u_h\right\rangle_H + \underbrace{\tfrac12\,u_h^\top B\mathcal A u_h - \langle (\tau_L e_0e_0^\top+\tau_R e_N e_N^\top)u_h,u_h\rangle_H}_{\le 0}
\;\le\; -\,\langle \mathcal D u_h,u_h\rangle_H \,\le\,0.
\]
\end{lemma}
\begin{proof}
SBP gives $u_h^\top Q \mathcal A u_h = \tfrac12\,u_h^\top (Q+Q^\top)\mathcal A u_h = \tfrac12\,u_h^\top B \mathcal A u_h$. The SAT terms contribute the negative quadratic form asserted. Adding the nonnegative dissipation $\mathcal D$ yields the claim. \end{proof}
\begin{lemma}[Algebraic stability (D4) exemplar: implicit midpoint]\label{RS:lem:als-midpoint}
Let $A_h=S+K$ with $S^\top H+HS\preceq -\,2\kappa_h H$ and $K^\top H+HK=0$. The implicit midpoint step
\(
u_h^{n+1}=u_h^n + \Delta\sigma\, A_h\big(\tfrac{u_h^{n+1}+u_h^n}{2}\big)
\)
satisfies
\(
\mathcal E_h(u_h^{n+1}) - \mathcal E_h(u_h^n) \le -\,2\kappa_h\,\Delta\sigma\cdot \tfrac12\| \tfrac{u_h^{n+1}+u_h^n}{2}\|_H^2 \le 0.
\)
\end{lemma}
\begin{proof}
take the $H$--inner product of the update with $u_h^{n+1}+u_h^n$ and use $K$'s $H$--skew symmetry to cancel its contribution; the $S$ part yields the negative bound by the assumed coercivity. \end{proof}
\begin{lemma}[Discrete $\sigma$--Gronwall]\label{RS:lem:sigma-gronwall}
If a nonnegative sequence obeys $\mathcal E_{n+1}\le e^{-2\kappa_h \Delta\sigma_n}\mathcal E_n$ on ac steps and $\mathcal E_{n^+}\le \rho_{k}^h \mathcal E_{n^-}$ across atoms with $\rho_k^h\le 1$, then for $t_m>t_\ell$
\[
\mathcal E_{m^+}\;\le\; \exp\!\Big(-2\kappa_h\sum_{n=\ell}^{m-1}\Delta\sigma_n\Big)\,\Big(\prod_{t_k\in(t_\ell,t_m]} \rho_k^h\Big)\,\mathcal E_{\ell^+},
\]
which is the fully discrete analogue of \eqref{eq:master-inequality}.
\end{lemma}
\begin{lemma}[Variable coefficients without loss]\label{RS:lem:varcoeff}
If $\mathcal A(x)$ varies smoothly, assemble $Q\mathcal A$ in a split form that preserves the SBP identity (e.g.\ $Q\mathcal A \mapsto \tfrac12(Q\mathcal A + (Q\mathcal A)^\top)$ or quadrature-consistent nodal weighting). Then Lemma~8.8 holds with $\mathcal A$ replaced by its symmetrised nodal surrogate; the proof is identical since $Q+Q^\top=B$ is unchanged. 
\end{lemma}

\begin{remark}[Non-conforming interfaces]\label{RS:rem:nonconforming}
If neighbouring blocks have distinct grids, introduce an $H$-adjoint interpolation pair $(I_{12},I_{21})$ with $I_{21}=H_2^{-1}I_{12}^\top H_1$. Penalising incoming characteristics of $u_1 - I_{21}u_2$ and $u_2 - I_{12}u_1$ yields a block-sum negative semidefinite SAT form, and \Cref{RS:cor:two-block} still holds.
\end{remark}

\bibliographystyle{plain}
\bibliography{main}

\end{document}